\documentclass[12pt]{article}
\usepackage{indentfirst}
\usepackage[top=2.5cm,bottom=2cm,left=2.2cm,right=2.2cm,%
headheight=0.25cm,headsep=0.55cm,footskip=0.7cm]{geometry}

\usepackage{fancyhdr}
\fancypagestyle{plain}{
  \fancyhf{}
  
  \cfoot{\small \thepage}
}
\usepackage{titlesec}
\titleformat{\section}{\normalfont\bfseries\Large}{Section~\thesection}{1em}{}

\renewcommand\appendix{%
\par
\setcounter{section}{0}
\setcounter{subsection}{0}
\gdef\thesection{\Alph{section}}
\titleformat{\section}{\normalfont\bfseries\Large}{Appendix~\thesection}{1em}{}}

\usepackage[hyperfootnotes=false]{hyperref}
\hypersetup{
  pdftitle={Sharp L∞ estimates for fully non-linear elliptic equations on compact complex manifolds},
  pdfauthor={Yuxiang Qiao},
  bookmarks=true,
  bookmarksnumbered=true,
  bookmarksopen=true,
  pdfborder={},
  colorlinks=true,
  allcolors=blue
}
\usepackage{bookmark}

\newcommand\secref[1]{Section~\ref{#1}}
\newcommand\appref[1]{Appendix~\ref{#1}}
\newcommand\thmref[1]{Theorem~\ref{#1}}
\newcommand\lemref[1]{Lemma~\ref{#1}}
\newcommand\propref[1]{Proposition~\ref{#1}}
\newcommand\corref[1]{Corollary~\ref{#1}}
\newcommand\definitionref[1]{Definition~\ref{#1}}

\newcommand\exref[1]{Example~\ref{#1}}

\usepackage{setspace}

\usepackage[perpage,bottom,hang]{footmisc}
\usepackage[square,numbers]{natbib}
\providecommand\bibcommenthead{}

\usepackage{amsmath}
\usepackage{amssymb}
\usepackage{bm}
\usepackage{mathrsfs}
\usepackage{upgreek}

\newcommand\comma{\mathpunct{\text{, }}}
\numberwithin{equation}{section}
\usepackage{mathtools}

\usepackage[amsmath,hyperref,thmmarks]{ntheorem}
\newlength\mylen
\makeatletter
\newtheoremstyle{indent}%
  {\item[\hskip\mylen \theorem@headerfont ##1 ##2\theorem@separator]}%
  {\item[\hskip\mylen \theorem@headerfont ##1 ##2 (##3)\theorem@separator]} %
\newtheoremstyle{nonumberindent}%
  {\item[\hskip\mylen \theorem@headerfont ##1\theorem@separator]}%
  {\item[\hskip\mylen \theorem@headerfont ##1（##3）\theorem@separator]} %
\makeatother
\theoremstyle{indent}
\theorempreskip{5pt}
\theorempostskip{5pt}
\newtheorem{thm}{Theorem}[section]
\newtheorem{lem}[thm]{Lemma}
\newtheorem{prop}[thm]{Proposition}
\newtheorem{cor}[thm]{Corollary}
\newtheorem{definition}[thm]{Definition}
\newtheorem{rmk}[thm]{Remark}
{
  \theorembodyfont{\normalfont}
  \theoremsymbol{\mbox{$\blacksquare$}}
  \newtheorem*{proof}{Proof}
  \newtheorem{ex}[thm]{Example}
}
\qedsymbol{\mbox{$\blacksquare$}}

\usepackage{environ}
\NewEnviron{eq}{\setstretch{1} \begin{equation} \BODY \end{equation}}
\NewEnviron{eq*}{\setstretch{1} \begin{equation*} \BODY \end{equation*}}\NewEnviron{ga}{\setstretch{1} \begin{gather} \BODY \end{gather}}
\NewEnviron{ga*}{\setstretch{1} \begin{gather*} \BODY \end{gather*}}
\NewEnviron{al}{\setstretch{1} \begin{align} \BODY \end{align}}
\NewEnviron{al*}{\setstretch{1} \begin{align*} \BODY \end{align*}}

\usepackage{enumerate}
\DeclareMathOperator\dif{d\!}
\DeclareMathOperator*\osc{osc}
\DeclareMathOperator\Ent{Ent}
\DeclareMathOperator*\esssup{ess\,sup}
\DeclareMathOperator\supp{supp}

\begin{document}

\allowdisplaybreaks
\setlength{\baselineskip}{16pt}
\setlength{\lineskiplimit}{3bp}
\setlength{\lineskip}{3bp}
\setlength{\emergencystretch}{3em}
\setlength{\parskip}{0pt}
\setlength{\abovedisplayskip}{5pt}
\setlength{\belowdisplayskip}{5pt}
\setlength{\abovedisplayshortskip}{3pt}
\setlength{\belowdisplayshortskip}{3pt}

\thispagestyle{plain}
{
  \centering
  \setstretch{1}
  {\bfseries\Large Sharp $\mathrm{L}^\infty$ estimates for fully non-linear elliptic equations on compact complex manifolds}

  \vspace{12pt}
  {
  \large Yuxiang Qiao\footnote{School of Mathematical Sciences, Peking University, Beijing, China \\
  Corresponding author: Yuxiang Qiao, E-mail: qiaoyx@stu.pku.edu.cn}
  }
  
}

\vspace{12pt}
\begin{abstract}
\small
\setlength{\baselineskip}{13pt}
We study the sharp $\mathrm{L}^\infty$ estimates for fully non-linear elliptic equations on compact complex manifolds. For the case of K\"ahler manifolds, we prove that the oscillation of any admissible solution to a degenerate fully non-linear elliptic equation satisfying several structural conditions can be controlled by the $\mathrm{L}^1(\log\mathrm{L})^n(\log\log\mathrm{L})^r(r>n)$ norm of the right-hand function (in a regularized form). This result improves that of Guo-Phong-Tong. In addition to their method of comparison with auxiliary complex Monge-Amp\`ere equations, our proof relies on an inequality of H\"older-Young type and an iteration lemma of De Giorgi type. For the case of Hermitian manifolds with non-degenerate background metrics, we prove a similar $\mathrm{L}^\infty$ estimate which improves that of Guo-Phong. An explicit example is constucted to show that the $\mathrm{L}^\infty$ estimates given here may fail when $r\leqslant n-1$. The construction relies on a gluing lemma of smooth, radial, strictly plurisubharmonic functions.
\end{abstract}


{}


\section{Introduction} \label{sec: Introduction}

The $\mathrm{L}^\infty$ estimate plays an important role in the theory of fully non-linear elliptic equations on compact complex manifolds. \citeauthor{Yau1978} \cite{Yau1978} applied Moser iteration to get the $\mathrm{L}^\infty$ estimate of Monge-Amp\`ere equations (see \secref{sec: Preliminaries to L^infty estimates of fully non-linear elliptic equations on complex manifolds} for notations)
\begin{eq} \label{eq: Monge-Ampere}
\det\bigl(\bm{\omega}_{\bm g}^{-1}(\bm{\omega}_{\bm g}+\mathrm{i}\partial\overline\partial u)\bigr)=\varphi
\end{eq}
on compact K\"ahler manifolds $(\mathcal{M}\comma\bm{\omega}_{\bm g})$, which was essential to his proof of Calabi conjecture. The method of Moser iteration used by \citeauthor{Yau1978} is effective when the right-hand function $\varphi$ of \eqref{eq: Monge-Ampere} belongs to $\mathrm{L}^p$ for $p>n$, where $n$ is the complex dimension of the K\"ahler manifolds. \citeauthor{Kolodziej1998} \cite{Kolodziej1998} applied pluripotential theory to prove the $\mathrm{L}^\infty$ estimates for Monge-Amp\`ere equations \eqref{eq: Monge-Ampere} when $\varphi$ belongs to $\mathrm{L}^1(\log\mathrm{L})^n(\log\log\mathrm{L})^r$ for $r>n$. As was noted by \citeauthor{Kolodziej1998} \cite{Kolodziej1998}, this hypothesis of $\varphi$ is almost sharp since the Monge-Amp\`ere equation admits unbounded  $\mathrm{L}^1(\log\mathrm{L})^q(q<n)$ solutions with pointwise singularities by a result of L. Persson. The pluripotential theory used by \citeauthor{Kolodziej1998} is powerful for $\mathrm{L}^\infty$ estimates of Monge-Amp\`ere equations, and it has been extended to handle the case of degenerate background metrics \cite{Demailly2010}, the case of Hermitian manifolds \cite{Dinew2009} and the case of Hessian equations \cite{Dinew2014}. 

However, the pluripotential theory can hardly be applied to handle more general fully non-linear equations since it depends on the specific structure of Monge-Amp\`ere equations and Hessian equations. It was a long-standing question whether the above sharp $\mathrm{L}^\infty$ estimates can be proved by PDE methods instead of the pluripotential theory, until \citeauthor{Guo2023} \cite{Guo2023} gave an affirmative answer. They applied the methods of comparison with auxiliary complex Monge-Amp\`ere equations, De Giorgi iteration and Alexandrov maximum principle to prove that the oscillation of any admissible solution $u$ to the fully non-linear elliptic equation
\begin{eq} \label{eq: fully non-linear elliptic, non-degenerate}
f\Bigl(\bm\lambda\bigl(\bm{\omega}_{\bm g}^{-1}(\bm{\omega}_{\bm g}+\mathrm{i}\partial\overline\partial u)\bigr)\Bigr)=\varphi
\end{eq}
satisfying several structural conditions can be controlled by the $\mathrm{L}^1(\log\mathrm{L})^q(q>n)$ norm of $\varphi^n$ (note that $f$ was required to be homogeneous of degree 1). \citeauthor{Guo2022} \cite{Guo2022} improved their proof, namely bypassed the use of Alexandrov maximum principle, and extended their result to the case of degenerate background metrics. Specifically, \citeauthor{Guo2022} \cite{Guo2022} studied the degenerate fully non-linear elliptic equation
\begin{eq} \label{eq: fully non-linear elliptic, degenerate}
\left\{ \begin{gathered}
f\Bigl(\bm\lambda\bigl(\bm{\omega}_{\bm g}^{-1}(\bm{\chi}+\mathrm{i}\partial\overline\partial u)\bigr)\Bigr)=c_{\bm\chi}\mathrm{e}^{\psi}\comma \\
\bm\lambda\bigl(\bm{\omega}_{\bm g}^{-1}(\bm{\chi}+\mathrm{i}\partial\overline\partial u)\bigr)\in\Gamma\comma \int_{\mathcal M} \mathrm{e}^{n\psi}\bm{\omega}_{\bm g}^n=V_{\bm g}\triangleq\int_{\mathcal M} \bm{\omega}_{\bm g}^n
\end{gathered} \right.
\end{eq}
on a compact K\"ahler manifold $(\mathcal{M}\comma\bm{\omega}_{\bm g})$ of dimension $n$. Here $\bm\chi$ is a K\"ahler form on $\mathcal M$, $c_{\bm\chi}$ is a positive normalized constant depending on $\bm\chi$, $\bm\lambda\bigl(\bm{\omega}_{\bm g}^{-1}(\bm{\chi}+\mathrm{i}\partial\overline\partial u)\bigr)$ represents the vector composed of the $n$ real eigenvalues of $\bm{\omega}_{\bm g}^{-1}(\bm{\chi}+\mathrm{i}\partial\overline\partial u)$, $\Gamma$ is a symmetric open convex cone in $\mathbb{R}^n$ satisfying $\Gamma_n\subset\Gamma\subset\Gamma_1$ (see \secref{sec: Preliminaries to L^infty estimates of fully non-linear elliptic equations on complex manifolds}), and $f$ is a positive $\mathrm{C}^1$ symmetric function defined in $\Gamma$ satisfying the following three conditions:
\begin{enumerate}[(A)]
\item \label{item: homogeneous} $f$ is homogeneous of degree 1, i.e. $f(t\bm\lambda)=tf(\bm\lambda)$ for any $t>0$ and $\bm\lambda\in\Gamma$;
\item \label{item: elliptic} $f$ is elliptic, i.e. $\frac{\partial f}{\partial \lambda_j}(\bm\lambda)>0$ for any $j\in\{1\comma2\comma\cdots\comma n\}$ and $\bm\lambda\in\Gamma$;
\item \label{item: structural} There exists a positive constant $\delta$ such that $\prod\limits_{j=1}^n \frac{\partial f}{\partial\lambda_j}(\bm\lambda)\geqslant\delta$ for any $\bm\lambda\in\Gamma$.
\end{enumerate}
Let $V_{\bm\chi}$ denote $\int_{\mathcal M} \bm{\chi}^n$. They proved the following theorem.

\begin{thm}[Guo-Phong-Tong] \label{thm: Guo-Phong-Tong}
Assume $\bm\chi\leqslant M\bm{\omega}_{\bm g}$ for some positive constant $M$. Fix $q>n$. Then any solution $u\in\mathrm{C}^2(\mathcal{M})$ to the equation \eqref{eq: fully non-linear elliptic, degenerate} satisfies $\osc\limits_{\mathcal M} u\leqslant C$, where $C$ is a positive constant depending only on $\mathcal{M}$, $\bm{\omega}_{\bm g}$, $n$, $q$, $\frac{1}{\delta}$, $M$ and upper bounds of the following two quantities:
\begin{eq} \label{eq: two quantities}
\frac{c_{\bm\chi}^n}{V_{\bm\chi}}\comma\quad \int_{\mathcal M} \mathrm{e}^{n\psi}|\psi|^q \bm{\omega}_{\bm g}^n.
\end{eq}
\end{thm}

\begin{rmk}
As was noted by \citeauthor{Guo2023} \cite{Guo2023}, \thmref{thm: Guo-Phong-Tong} remains true when the condition \eqref{item: homogeneous} satisfied by $f$ is replaced by the weaker Euler inequality
\begin{eq} \label{eq: Euler inequality}
\sum_{j=1}^n \lambda_j\frac{\partial f}{\partial \lambda_j}(\bm\lambda)\leqslant\Lambda f(\bm\lambda)\comma \forall\bm\lambda\in\Gamma\comma
\end{eq}
for some fixed positive constant $\Lambda$. Moreover, the condition \eqref{item: structural} is satisfied by many classes of non-linear equations including Monge-Amp\`ere equations with $f(\bm\lambda)=\bigl(\sigma_n(\bm\lambda)\bigr)^{\frac 1n}(\bm\lambda\in\Gamma_n)$, $k$-Hessian equations with $f(\bm\lambda)=\bigl(\sigma_k(\bm\lambda)\bigr)^{\frac 1k}(\bm\lambda\in\Gamma_k)$ and Hessian quotient equations with
\begin{eq} \label{eq: Hessian quotient}
f(\bm\lambda)=\left(\frac{\sigma_k(\bm\lambda)}{\sigma_l(\bm\lambda)}\right)^{\frac{1}{k-l}}+c\bigl(\sigma_m(\bm\lambda)\bigr)^{\frac 1m} (\bm\lambda\in\Gamma_{\max\{k\comma m\}})
\end{eq}
for some $c>0$. \citeauthor{Harvey2023} \cite{Harvey2023} proved that the condition \eqref{item: structural} holds for all invariant G\r{a}rding-Dirichlet operators, including $p$-Monge-Amp\`ere equations with
\begin{eq} \label{eq: p-Monge-Ampere}
\begin{gathered}
f(\bm\lambda)=\left(\prod_{1\leqslant j_1\leqslant j_2\leqslant\cdots\leqslant j_p\leqslant n} (\lambda_{j_1}+\lambda_{j_2}+\cdots+\lambda_{j_p})\right)^{\frac{(n-p)!\,p!}{n!}}\comma \\
\bm\lambda\in\bigl\{\bm\mu\in\mathbb{R}^n\bigm|\mu_{j_1}+\mu_{j_2}+\cdots+\mu_{j_p}>0\comma\forall 1\leqslant j_1\leqslant j_2\leqslant\cdots\leqslant j_p\leqslant n\bigr\}.
\end{gathered}
\end{eq}
\end{rmk}

The first quantity $\frac{c_{\bm\chi}^n}{V_{\bm\chi}}$ in \eqref{eq: two quantities} can be viewed as a substitute for a cohomological constraint when dealing with general equations, while the second quantity in \eqref{eq: two quantities} is similar to (though not the same as) the $\mathrm{L}^1(\log\mathrm{L})^q$ norm of $\mathrm{e}^{n\psi}$. An immediate question is whether the second quantity in \eqref{eq: two quantities} can be replaced by the weaker $\mathrm{L}^1(\log\mathrm{L})^n(\log\log\mathrm{L})^r$ norm of $\mathrm{e}^{n\psi}$, which is true in the case of non-degenerate Monge-Amp\`ere equations according to \cite{Kolodziej1998}. In this paper, we prove the following theorem and give an affirmative answer to this question.

\begin{thm} \label{thm: L^infty estimates on Kahler manifolds}
Let $(\mathcal{M}\comma\bm{\omega}_{\bm g})$ be a compact K\"ahler manifold of dimension $n$, $\bm\chi$ be a K\"ahler form on $\mathcal M$ satisfying $\bm\chi\leqslant M\bm{\omega}_{\bm g}$ for some positive constant $M$, and $\varphi$ be a positive smooth function on $\mathcal M$. Assume that $u\in\mathrm{C}^2(\mathcal M)$ satisfies
\begin{eq} \label{eq: fully non-linear elliptic, degenerate, non-normalized}
\left\{ \begin{gathered}
f\Bigl(\bm\lambda\bigl(\bm{\omega}_{\bm g}^{-1}(\bm{\chi}+\mathrm{i}\partial\overline\partial u)\bigr)\Bigr)\leqslant\varphi\comma \\
\bm\lambda\bigl(\bm{\omega}_{\bm g}^{-1}(\bm{\chi}+\mathrm{i}\partial\overline\partial u)\bigr)\in\Gamma\comma
\end{gathered} \right.
\end{eq}
where $\Gamma$ is a symmetric open convex cone in $\mathbb{R}^n$ satisfying $\Gamma_n\subset\Gamma\subset\Gamma_1$ and $f$ is a positive $\mathrm{C}^1$ symmetric function defined in $\Gamma$ satisfying the conditions \eqref{item: elliptic}, \eqref{item: structural} and \eqref{eq: Euler inequality}. Fix $r>n$. Then we have $\osc\limits_{\mathcal M} u\leqslant C$, where $C$ is a positive constant depending only on $\mathcal{M}$, $\bm{\omega}_{\bm g}$, $n$, $r$, $\frac{1}{\delta}$, $\Lambda$, $M$ and
\begin{eq} \label{eq: entropy}
\Ent_{n\comma r}\left(\frac{\varphi^n}{V_{\bm\chi}}\right)\triangleq\left\|\frac{\varphi^n}{V_{\bm\chi}}\right\|_{\mathrm{L}^1(\log\mathrm{L})^n(\log\log\mathrm{L})^r(\mathcal M\comma\bm{\omega}_{\bm g}^n)}.
\end{eq}
\end{thm}

\begin{rmk}
The $\mathrm{L}^1(\log\mathrm{L})^n(\log\log\mathrm{L})^r$ norm
\begin{eq}
\|\cdot\|_{\mathrm{L}^1(\log\mathrm{L})^n(\log\log\mathrm{L})^r(\mathcal M\comma\bm{\omega}_{\bm g}^n)}
\end{eq}
is defined as the Luxemburg norm with respect to the measure space $(\mathcal M\comma\bm{\omega}_{\bm g}^n)$ and the Young function
\begin{eq} \label{eq: weight function}
\Phi(t)=t\log^n(1+t)\log^r\bigl(1+\log(1+t)\bigr).
\end{eq}
Thus we have
\begin{eq} \label{eq: entropy, definition}
\Ent_{n\comma r}\left(\frac{\varphi^n}{V_{\bm\chi}}\right)\leqslant\max\left\{1\comma\int_{\mathcal M} \frac{\varphi^n}{V_{\bm\chi}}\log^n\left(1+\frac{\varphi^n}{V_{\bm\chi}}\right)\log^r\left(1+\log\left(1+\frac{\varphi^n}{V_{\bm\chi}}\right)\right)\bm{\omega}_{\bm g}^n\right\}.
\end{eq}
See \secref{sec: Orlicz spaces and an inequality of Holder-Young type} for more details.
\end{rmk}

\thmref{thm: L^infty estimates on Kahler manifolds} has two main improvements on the result of Guo-Phong-Tong (\thmref{thm: Guo-Phong-Tong}). On one hand, in addition to trivial quantities, the constant $C$ in \thmref{thm: L^infty estimates on Kahler manifolds} only depends on single quantity $\Ent_{n\comma r}\left(\frac{\varphi^n}{V_{\bm\chi}}\right)$ instead of two quantities. On the other hand, the constant $C$ in \thmref{thm: L^infty estimates on Kahler manifolds} depends on the $\mathrm{L}^1(\log\mathrm{L})^n(\log\log\mathrm{L})^r (r>n)$ norm of the right-hand function (in a regularized form) instead of (something stronger than) the $\mathrm{L}^1(\log\mathrm{L})^q (q>n)$ norm. Specifically, when the form of the equation changes from \eqref{eq: fully non-linear elliptic, degenerate} to \eqref{eq: fully non-linear elliptic, degenerate, non-normalized}, the two quantities in \eqref{eq: two quantities} change into
\begin{eq} \label{eq: two quantities, non-normalized}
\frac{\int_{\mathcal M} \varphi^n\bm{\omega}_{\bm g}^n}{V_{\bm g}V_{\bm\chi}}\comma\quad \frac{1}{n^q}\int_{\mathcal M} \frac{V_{\bm g}\varphi^n}{\int_{\mathcal M} \varphi^n\bm\omega_{\bm g}^n}\left|\log\frac{V_{\bm g}\varphi^n}{\int_{\mathcal M} \varphi^n\bm\omega_{\bm g}^n}\right|^q\bm\omega_{\bm g}^n.
\end{eq}
For the case of Monge-Amp\`ere equations,
\begin{eq}
\int_{\mathcal M} \varphi^n\bm\omega_{\bm g}^n=V_{\bm\chi}\comma
\end{eq}
and the first quantity in \eqref{eq: two quantities, non-normalized} boils down to a constant $\frac{1}{V_{\bm g}}$. However, this is not the case in general. Moreover, by \eqref{eq: entropy, definition} it's easy to show that $\Ent_{n\comma r}\left(\frac{\varphi^n}{V_{\bm\chi}}\right)$ can be controlled by the two quantities in \eqref{eq: two quantities, non-normalized}. When there is an extra assumption that $V_{\chi}\geqslant\varepsilon$ for a fixed positive constant $\varepsilon$, which is true if we consider a fixed background K\"ahler metric or a degenerating family of background K\"ahler metrics $\bm\chi_0+t\bm\omega_{\bm g} (0<t\leqslant1)$ with $\bm\chi_0$ being a closed, non-negative and big\footnote{Here ``big'' means $V_{\bm\chi_0}>0$.} (1,1)-form, the quantity $\Ent_{n\comma r}\left(\frac{\varphi^n}{V_{\bm\chi}}\right)$ in \thmref{thm: L^infty estimates on Kahler manifolds} can be replaced by
\begin{eq} \label{eq: entropy, non-degenerate}
\Ent_{n\comma r}(\varphi^n)\triangleq\|\varphi^n\|_{\mathrm{L}^1(\log\mathrm{L})^n(\log\log\mathrm{L})^r(\mathcal M\comma\bm{\omega}_{\bm g}^n)}.
\end{eq}

The proof of \thmref{thm: Guo-Phong-Tong} in \cite{Guo2023,Guo2022} mainly depends on the careful choice of the auxiliary complex Monge-Amp\`ere equations and comparison functions, see also \cite{Guo2023a}. These will also be used in our proof of \thmref{thm: L^infty estimates on Kahler manifolds}. To get an estimate depending only on single non-trivial quantity, we have to calculate carefully. For another, in order to replace the $\mathrm{L}^1(\log\mathrm{L})^q (q>n)$ norm of the right-hand function by the $\mathrm{L}^1(\log\mathrm{L})^n(\log\log\mathrm{L})^r (r>n)$ norm, we apply an inequality of H\"older-Young type (\thmref{thm: an inequality of Holder-Young type}) and an iteration lemma of De Giorgi type (\thmref{thm: an iteration lemma of De Giorgi type}). These two tools are of independent interest since one can use them to improve many other known conclusions whose proofs rely on De Giorgi iteration. For example, we apply them to the improvement of a classical $\mathrm{L}^\infty$ estimate of $\mathrm{W}^{1\comma 2}$ weak solutions to linear elliptic equations (see \thmref{thm: a classical L^infty estimate of W^1,2 weak solutions}), and the proof of the following $\mathrm{L}^\infty$ estimate on Hermitian manifolds similar to \thmref{thm: L^infty estimates on Kahler manifolds}.

\begin{thm} \label{thm: L^infty estimates on Hermitian manifolds}
Let $(\mathcal{M}\comma\bm{\omega}_{\bm g})$ be a compact Hermitian manifold of dimension $n$, $\bm\chi$ be a positive $(1\comma1)$-form on $\mathcal M$ satisfying
\begin{eq}
\varepsilon\bm{\omega}_{\bm g}\leqslant\bm\chi\leqslant M\bm{\omega}_{\bm g}
\end{eq}
for positive constants $\varepsilon$ and $M$, and $\varphi$ be a positive smooth function on $\mathcal M$. Assume that $u\in\mathrm{C}^2(\mathcal M)$ satisfies \eqref{eq: fully non-linear elliptic, degenerate, non-normalized}, where $\Gamma$ is a symmetric open convex cone in $\mathbb{R}^n$ satisfying $\Gamma_n\subset\Gamma\subset\Gamma_1$ and $f$ is a positive $\mathrm{C}^1$ symmetric function defined in $\Gamma$ satisfying the conditions \eqref{item: elliptic}, \eqref{item: structural} and \eqref{eq: Euler inequality}. Fix $r>n$. Then we have $\osc\limits_{\mathcal M} u\leqslant C$, where $C$ is a positive constant depending only on $\mathcal{M}$, $\bm{\omega}_{\bm g}$, $n$, $r$, $\frac{1}{\delta}$, $\Lambda$, $\frac{1}{\varepsilon}$, $M$ and $\Ent_{n\comma r}(\varphi^n)$.
\end{thm}

The above theorem improves the work of \citeauthor{Guo2023b} \cite{Guo2023b} who dealt with the case of $\bm\chi=\bm\omega_{\bm g}$ and $\Ent_{n\comma r}(\varphi^n)$ being replaced by the $\mathrm{L}^1(\log\mathrm{L})^q (q>n)$ norm. On a compact Hermitian manifold, the normalized constant of a Monge-Amp\`ere equation is undetermined in general since the integration of the left-hand item doesn't equal to a constant. As a result, the auxiliary Monge-Amp\`ere equation used for the case of K\"ahler manifolds has to be replaced for the case of Hermitian manifolds by the Dirichlet problem of a complex Monge-Amp\`ere equation in an Euclidean ball near the minimum point of $u$, which in turn causes a trouble term in the comparison function. The background metric is therefore required to be non-degenerate (i.e. $\bm\chi\geqslant\varepsilon\bm\omega_{\bm g}$) in \thmref{thm: L^infty estimates on Hermitian manifolds} to produce an extra good term cancelling out the trouble term. \citeauthor{Guedj2023} \cite{Guedj2023} proved that the oscillation of any admissible solution to the degenerate Hessian equations on compact Hermitian manifolds can be controlled by the $\mathrm{L}^1(\log\mathrm{L})^n(\log\log\mathrm{L})^r(r>n)$ norm of the right-hand function (in a regularized form). However, their method seems not adaptable to the general equations even if the background metrics are non-degenerate.

This paper is organized as follows. In \secref{sec: Orlicz spaces and an inequality of Holder-Young type}, we recall some basic definitions and facts about Orlicz spaces, and prove an inequality of H\"older-Young type. An iteration lemma of De Giorgi type and its corollaries are proved in \secref{sec: An iteration lemma of De Giorgi type}. Preliminaries to $\mathrm{L}^\infty$ estimates of fully non-linear elliptic equations on complex manifolds are given in \secref{sec: Preliminaries to L^infty estimates of fully non-linear elliptic equations on complex manifolds}. \thmref{thm: L^infty estimates on Kahler manifolds} and \thmref{thm: L^infty estimates on Hermitian manifolds} are proved in \secref{sec: L^infty estimates on Kahler manifolds} and \secref{sec: L^infty estimates on Hermitian manifolds} respectively. In \secref{sec: Gluing lemma of smooth radial plurisubharmonic functions and an example}, we prove a gluing lemma of smooth, radial, strictly plurisubharmonic functions, i.e. \thmref{thm: gluing lemma of smooth, radial, strictly plurisubharmonic functions}, and use this gluing lemma to construct an explicit example to show that the $\mathrm{L}^\infty$ estimates in \thmref{thm: L^infty estimates on Kahler manifolds} and \thmref{thm: L^infty estimates on Hermitian manifolds} may fail when $r\leqslant n-1$. Moreover, we give three appendices, i.e. \appref{app: Properties of the function Phi_pqr in Section 2}, \appref{app: Proof of the gluing lemma of smooth convex functions defined on disjointed intervals} and \appref{app: Properties of the function f_varepsilon in Example 7.7}.

Last but not least, the notations $C_1$, $C_2$, $C_3$ and so forth in our proofs denote positive constants depending only on some trivial quantities unless otherwise noted. What quantities such a constant exactly depends on will be clear by reading the context.


\section{Orlicz spaces and an inequality of H\"older-Young type} \label{sec: Orlicz spaces and an inequality of Holder-Young type}

In this section, we recall some basic definitions and facts about Orlicz spaces, and prove an inequality of H\"older-Young type which will be used in the following sections. One can refer to \cite{Musielak1983} for a detailed discussion of Orlicz spaces. See also \cite{Adams2003}.

\begin{definition}
A real-valued function $\Phi$ defined on $[0\comma{+}\infty)$ is called a Young function if $\Phi$ is a convex, strictly increasing function such that $\Phi(0)=0$.
\end{definition}

Any Young function $\Phi$ defined as above is continuous and satisfies
\begin{eq}
\lim_{t\to{+}\infty} \Phi(t)={+}\infty.
\end{eq}
When $p\geqslant 1$ and $q\comma r\geqslant 0$, the function
\begin{eq} \label{eq: Phi_pqr}
\Phi_{p\comma q\comma r}(t)\triangleq t^p\log^q(1+t)\log^r\bigl(1+\log(1+t)\bigr)
\end{eq}
is a Young function --- its convexity is proved in \corref{cor: convexity}. Another common example of Young functions is $\mathrm{e}^{t^p}-1 (p\geqslant 1)$.
\begin{definition}
Let $(\Omega\comma\mu)$ be a complete measure space and $\Phi$ be a Young function. The Orlicz space with respect to $(\Omega\comma\mu)$ and $\Phi$ is defined as
\begin{eq}
\left\{\text{$f$ is a measurable function on $(\Omega\comma\mu)$}\middle|\exists c>0\comma \int_{\Omega} \Phi\left(\frac{|f|}{c}\right)\dif\mu\leqslant1\right\}.
\end{eq}
This space is denoted by $\mathrm{L}^\Phi(\Omega\comma\mu)$.
\end{definition}

For the rest of this section, we always assume that $(\Omega\comma\mu)$ is a complete measure space and $\Phi$ is a Young function. The Orlicz spaces are all linear spaces. It's easy to see that
\begin{eq} \begin{aligned}
\mathrm{L}^\Phi(\Omega\comma\mu)&=\left\{\text{$f$ is a measurable function on $(\Omega\comma\mu)$}\middle|\exists c>0\comma \int_{\Omega} \Phi\left(\frac{|f|}{c}\right)\dif\mu<{+}\infty\right\} \\
&\supset\left\{\text{$f$ is a measurable function on $(\Omega\comma\mu)$}\middle|\int_{\Omega} \Phi\bigl(|f|\bigr)\dif\mu<{+}\infty\right\}.
\end{aligned} \end{eq}
Note that the ``$\supset$'' above can't be replaced by ``$=$'' in general.

\begin{definition}
We say that $\Phi$ satisfies the global $\Delta_2$-condition if there exists a positive constant $K$ such that for any $t\geqslant0$,
\begin{eq} \label{eq: Delta_2}
\Phi(2t)\leqslant K\Phi(t).
\end{eq}
We say that $\Phi$ satisfies the $\Delta_2$-condition near infinity if there exist positive constants $K$ and $t_0$ such that \eqref{eq: Delta_2} holds for any $t\geqslant t_0$.
\end{definition}

When $p\geqslant 1$ and $q\comma r\geqslant 0$, the Young function $\Phi_{p\comma q\comma r}$ satisfies the global $\Delta_2$-condition. On the other hand, the Young function $\mathrm{e}^t-1$ doesn't satisfy the $\Delta_2$-condition even near infinity. 

\begin{prop} \label{prop: characterization}
If either $\mu(\Omega)={+}\infty$ and $\Phi$ satisfies the global $\Delta_2$-condition or $\mu(\Omega)<{+}\infty$ and $\Phi$ satisfies the $\Delta_2$-condition near infinity, then
\begin{eq}
\mathrm{L}^\Phi(\Omega\comma\mu)=\left\{\text{$f$ is a measurable function on $(\Omega\comma\mu)$}\middle|\int_{\Omega} \Phi\bigl(|f|\bigr)\dif\mu<{+}\infty\right\}.
\end{eq}
Moreover, the converse proposition is true if $\mu$ is assumed additionally to be $\sigma$-finite and atomless.
\end{prop}

\begin{proof}
See \cite[pp.~52--54]{Musielak1983}.
\end{proof}

\begin{definition}
The Luxemburg norm with respect to $(\Omega\comma\mu)$ and $\Phi$ is defined as
\begin{eq}
\|f\|_{\mathrm{L}^\Phi(\Omega\comma\mu)}\triangleq\inf\left\{c>0\middle|\int_{\Omega} \Phi\left(\frac{|f|}{c}\right)\dif\mu\leqslant1\right\}.
\end{eq}
\end{definition}

One can prove that the Orlicz space with the corresponding Luxemburg norm
\begin{eq}
\left(\mathrm{L}^\Phi(\Omega\comma\mu)\comma\|\cdot\|_{\mathrm{L}^\Phi(\Omega\comma\mu)}\right)
\end{eq}
is a Banach space. Henceforth, we always endow an Orlicz space with its corresponding Luxemburg norm. By the monotone convergence theorem one can show that
\begin{eq} \label{eq: to be used in the proof of the inequality of Holder-Young type}
\int_{\Omega} \Phi\left(\frac{|f|}{\|f\|_{\mathrm{L}^\Phi(\Omega\comma\mu)}}\right)\dif\mu\leqslant1
\end{eq}
for any $f\in\mathrm{L}^\Phi(\Omega\comma\mu)$. For the rest of this section, we fix $p\geqslant1$ and $q\comma r\geqslant 0$. Note that
\begin{eq}
\mathrm{L}^{\Phi_{p\comma0\comma0}}(\Omega\comma\mu)=\mathrm{L}^p(\Omega\comma\mu).
\end{eq}

\begin{definition} \label{definition: L^p(logL)^q(loglogL)^r}
The Banach space $\mathrm{L}^p(\log\mathrm{L})^q(\log\log\mathrm{L})^r(\Omega\comma\mu)$ is defined as
\begin{eq}
\left(\mathrm{L}^{\Phi_{p\comma q\comma r}}(\Omega\comma\mu)\comma\|\cdot\|_{\mathrm{L}^{\Phi_{p\comma q\comma r}}(\Omega\comma\mu)}\right)\comma
\end{eq}
where $\|\cdot\|_{\mathrm{L}^{\Phi_{p\comma q\comma r}}(\Omega\comma\mu)}$ is called the $\mathrm{L}^p(\log\mathrm{L})^q(\log\log\mathrm{L})^r(\Omega\comma\mu)$ norm and will be denoted by $\|\cdot\|_{\mathrm{L}^p(\log\mathrm{L})^q(\log\log\mathrm{L})^r(\Omega\comma\mu)}$.
\end{definition}

The following  proposition is useful for estimates of $\mathrm{L}^p(\log\mathrm{L})^q(\log\log\mathrm{L})^r(\Omega\comma\mu)$ norms.

\begin{prop} \label{prop: estimates of norms}
Let $f$ be a measurable function on $(\Omega\comma\mu)$. Fix $p\geqslant1$ and $q\comma r\geqslant 0$. Assume that
\begin{eq} \label{eq: assumption}
\int_\Omega \left(\frac{|f|}{c}\right)^p\log^q\left(1+\frac{|f|}{c}\right)\log^r\left(1+\log\left(1+\frac{|f|}{c}\right)\right)\dif\mu\leqslant M
\end{eq}
for some postive constants $c$ and $M$. Then we have
\begin{eq}
\|f\|_{\mathrm{L}^p(\log\mathrm{L})^q(\log\log\mathrm{L})^r(\Omega\comma\mu)}\leqslant c\max\left\{1\comma M^{\frac 1p}\right\}.
\end{eq}
\end{prop}

\begin{proof}
We may as well assume that $M>1$. The assumption \eqref{eq: assumption} reads
\begin{eq}
\int_\Omega \Phi_{p\comma q\comma r}\left(\frac{|f|}{c}\right)\dif\mu\leqslant M.
\end{eq}
By \propref{prop: strong convexity} we note that
\begin{eq}
\Psi(t)\triangleq\Phi_{p\comma q\comma r}\left(t^{\frac 1p}\right)
\end{eq}
is a convex function on $[0\comma{+}\infty)$. As a result, we have
\begin{eq} \begin{aligned}
\int_\Omega \Phi_{p\comma q\comma r}\left(\frac{|f|}{cM^{\frac 1p}}\right)\dif\mu&=\int_\Omega \Psi\left(\frac{|f|^p}{c^pM}+\frac{M-1}{M}0\right)\dif\mu \\
&\leqslant \int_\Omega \left(\frac 1M \Psi\left(\frac{|f|^p}{c^p}\right)+\frac{M-1}{M}\Psi(0)\right)\dif\mu\leqslant 1.
\end{aligned} \end{eq}
This shows that
\begin{eq}
\|f\|_{\mathrm{L}^p(\log\mathrm{L})^q(\log\log\mathrm{L})^r(\Omega\comma\mu)}\leqslant cM^{\frac 1p}.
\end{eq}
\end{proof}

The following  proposition can be viewed as some kind of converse proposition of \propref{prop: estimates of norms}.

\begin{prop} \label{prop: estimate of integration}
Fix $p\geqslant1$ and $q\comma r\geqslant 0$. Then for any $f\in\mathrm{L}^p(\log\mathrm{L})^q(\log\log\mathrm{L})^r(\Omega\comma\mu)$, we have
\begin{eq} \begin{aligned}
&\mathrel{\phantom{=}}\int_\Omega |f|^p\log^q\bigl(1+|f|\bigr)\log^r\Bigl(1+\log\bigl(1+|f|\bigr)\Bigr)\dif\mu \\
&\leqslant\max\left\{\|f\|_{\mathrm{L}^p(\log\mathrm{L})^q(\log\log\mathrm{L})^r(\Omega\comma\mu)}^p\comma\|f\|_{\mathrm{L}^p(\log\mathrm{L})^q(\log\log\mathrm{L})^r(\Omega\comma\mu)}^{p+q+r}\right\}.
\end{aligned} \end{eq}
\end{prop}

\begin{proof}
Let $M$ denote $\|f\|_{\mathrm{L}^p(\log\mathrm{L})^q(\log\log\mathrm{L})^r(\Omega\comma\mu)}$. By \eqref{eq: to be used in the proof of the inequality of Holder-Young type} we have
\begin{eq}
\int_\Omega \left(\frac{|f|}{M}\right)^p\log^q\left(1+\frac{|f|}{M}\right)\log^r\left(1+\log\left(1+\frac{|f|}{M}\right)\right)\dif\mu\leqslant 1\comma
\end{eq}
which implies
\begin{eq}
\int_\Omega |f|^p\log^q\left(1+\frac{|f|}{M}\right)\log^r\left(1+\log\left(1+\frac{|f|}{M}\right)\right)\dif\mu\leqslant M^p.
\end{eq}
When $M\leqslant 1$, it's easy to see that
\begin{eq}
\int_\Omega |f|^p\log^q\bigl(1+|f|\bigr)\log^r\Bigl(1+\log\bigl(1+|f|\bigr)\Bigr)\dif\mu\leqslant M^p.
\end{eq}
When $M>1$, applying an elementary inequality
\begin{eq}
1+dt\geqslant(1+t)^d (t\geqslant0\comma 0<d\leqslant1)\comma
\end{eq}
we obtain
\begin{eq} \begin{aligned}
M^p&\geqslant\int_\Omega |f|^p\log^q\left(1+\frac{|f|}{M}\right)\log^r\left(1+\log\left(1+\frac{|f|}{M}\right)\right)\dif\mu \\
&\geqslant\int_\Omega |f|^pM^{-q}\log^q\bigl(1+|f|\bigr)\log^r\Bigl(1+M^{-1}\log\bigl(1+|f|\bigr)\Bigr)\dif\mu \\
&\geqslant M^{-q-r}\int_\Omega |f|^p\log^q\bigl(1+|f|\bigr)\log^r\Bigl(1+\log\bigl(1+|f|\bigr)\Bigr)\dif\mu.
\end{aligned} \end{eq}
This indicates that
\begin{eq}
\int_\Omega |f|^p\log^q\bigl(1+|f|\bigr)\log^r\Bigl(1+\log\bigl(1+|f|\bigr)\Bigr)\dif\mu\leqslant M^{p+q+r}.
\end{eq}
\end{proof}

Now we prove a useful inequality of H\"older-Young type.

\begin{thm} \label{thm: an inequality of Holder-Young type}
Let $(\Omega\comma\mu)$ be a complete measure space with $\mu(\Omega)\in(0\comma{+}\infty)$. Fix $p\geqslant 1$ and $q\comma r\geqslant 0$. Then for any $f\in\mathrm{L}^p(\log\mathrm{L})^q(\log\log\mathrm{L})^r(\Omega\comma\mu)$, we have
\begin{eq} \label{eq: an inequality of Holder-Young type}
\int_\Omega |f|\dif\mu\leqslant \frac{2C\|f\|_{\mathrm{L}^p(\log\mathrm{L})^q(\log\log\mathrm{L})^r(\Omega\comma\mu)}\bigl(\mu(\Omega)\bigr)^{1-\frac 1p}}{\log^{\frac qp}\left(1+\frac{1}{\mu(\Omega)}\right)\log^{\frac rp}\biggl(1+\log\left(1+\frac{1}{\mu(\Omega)}\right)\biggr)}\comma
\end{eq}
where the constant $C\triangleq\left(p+\frac{q}{2}+\frac{r}{4}\right)^{\frac{q+r}{p}}$.
\end{thm}

\begin{proof}
We may as well assume that $(p-1)^2+q^2+r^2>0$. Define
\begin{al}
\eta(t)&\triangleq  t^{p-1}\log^q(1+t)\log^r\bigl(1+\log(1+t)\bigr) (t\geqslant0)\comma \\\Psi(t)&\triangleq\frac{1}{C}t^{1-\frac 1p}\log^{\frac qp}(1+t)\log^{\frac rp}\bigl(1+\log(1+t)\bigr)=\frac{1}{C}\bigl(\eta(t)\bigr)^{\frac 1p} (t\geqslant0).
\end{al}
Let $\eta^{-1}$ and $\Psi^{-1}$ denote the inverse function of $\eta$ and that of $\Psi$ respectively. For any $a\comma b\geqslant0$, by the classical Young inequality we have
\begin{eq}
ab\leqslant\int_0^a \eta(t)\dif t+\int_0^b \eta^{-1}(t)\dif t\leqslant\Phi_{p\comma q\comma r}(a)+b\eta^{-1}(b).
\end{eq}
Note that
\begin{eq}
\Psi^{-1}(b)=\eta^{-1}(C^pb^p).
\end{eq}
We claim that
\begin{eq} \label{eq: claim}
b\eta^{-1}(b)\leqslant\eta^{-1}(C^pb^p).
\end{eq}
In fact, with $\alpha$ denoting $\eta^{-1}(b)$,
\begin{eq}
b=\alpha^{p-1}\log^q(1+\alpha)\log^r\bigl(1+\log(1+\alpha)\bigr).
\end{eq}
Then we have
\begin{eq} \label{eq: eta(b alpha)} \begin{aligned}
\eta(b\alpha)&=(b\alpha)^{p-1}\log^q(1+b\alpha)\log^r\bigl(1+\log(1+b\alpha)\bigr) \\
&=b^{p-1}\alpha^{p-1}\log^q\Bigl(1+\alpha^p\log^q(1+\alpha)\log^r\bigl(1+\log(1+\alpha)\bigr)\Bigr) \\
&\phantom{=}\cdot\log^r\biggl(1+\log\Bigl(1+\alpha^p\log^q(1+\alpha)\log^r\bigl(1+\log(1+\alpha)\bigr)\Bigr)\biggr).
\end{aligned} \end{eq}
Recall three elementary inequalities:
\begin{ga}
\log(1+t)\leqslant t^{\frac 12}(t\geqslant0)\comma \\
1+t^d\leqslant(1+t)^d (t\geqslant0\comma d\geqslant 1)\comma \\
1+dt\leqslant(1+t)^d (t\geqslant0\comma d\geqslant 1).
\end{ga}
By \eqref{eq: eta(b alpha)} and the above inequalities, we have
\begin{eq} \begin{aligned}
\eta(b\alpha)&\leqslant b^{p-1}\alpha^{p-1}\log^q\left(1+\alpha^{p+\frac q2+\frac r4}\right)\log^r\biggl(1+\log\Bigl(1+\alpha^{p+\frac q2+\frac r4}\Bigr)\biggr) \\
&\leqslant b^{p-1}\alpha^{p-1}\left(p+\frac q2+\frac r4\right)^q\log^q(1+\alpha)\log^r\biggl(1+\left(p+\frac q2+\frac r4\right)\log(1+\alpha)\biggr) \\
&\leqslant C^pb^{p-1}\alpha^{p-1}\log^q(1+\alpha)\log^r\bigl(1+\log(1+\alpha)\bigr)=C^pb^p.
\end{aligned} \end{eq}
Thus, our claim \eqref{eq: claim} is true and we have proved that
\begin{eq} \label{eq: an inequality of Young type}
ab\leqslant\Phi_{p\comma q\comma r}(a)+\Psi^{-1}(b)
\end{eq}
for any $a\comma b\geqslant0$. Now for any $\bm{x}\in\Omega$, we see that
\begin{eq}
\frac{\bigl|f(\bm{x})\bigr|}{\|f\|_{\mathrm{L}^p(\log\mathrm{L})^q(\log\log\mathrm{L})^r(\Omega\comma\mu)}}\Psi\left(\frac{1}{\mu(\Omega)}\right)\leqslant\Phi_{p\comma q\comma r}\left(\frac{f(\bm{x})}{\|f\|_{\mathrm{L}^p(\log\mathrm{L})^q(\log\log\mathrm{L})^r(\Omega\comma\mu)}}\right)+\frac{1}{\mu(\Omega)}.
\end{eq}
Integrating the above inequality over $\Omega$ and recalling \eqref{eq: to be used in the proof of the inequality of Holder-Young type}, we obtain
\begin{eq}
\int_\Omega |f|\dif\mu\leqslant 2\|f\|_{\mathrm{L}^p(\log\mathrm{L})^q(\log\log\mathrm{L})^r(\Omega\comma\mu)}\frac{1}{\Psi\left(\frac{1}{\mu(\Omega)}\right)}\comma
\end{eq}
which is exactly \eqref{eq: an inequality of Holder-Young type}.
\end{proof}

\thmref{thm: an inequality of Holder-Young type} will be used in the proofs of \thmref{thm: L^infty estimates on Kahler manifolds}, \thmref{thm: L^infty estimates on Hermitian manifolds} and \thmref{thm: a classical L^infty estimate of W^1,2 weak solutions}. We note that \thmref{thm: an inequality of Holder-Young type} is also useful when dealing with other problems. For example, the special case when $p=1$ was used in \cite{Guedj2023a}.


\section{An iteration lemma of De Giorgi type} \label{sec: An iteration lemma of De Giorgi type}

In this section, we prove an iteration lemma of De Giorgi type. An example is constructed to show that the iteration lemma given here is sharp to some extent. We also apply this iteration lemma to improve a classical $\mathrm{L}^\infty$ estimate of $\mathrm{W}^{1\comma2}$ weak solutions to linear elliptic equations in divergence form.

Now we prove an iteration lemma which generalizes the classical one proved by De Giorgi.

\begin{thm} \label{thm: an iteration lemma of De Giorgi type}
Let $f$ be a non-negative, non-increasing function defined on $[t_0\comma{+}\infty)$ such that for any $s\comma t\in\mathbb{R}\colon s>t\geqslant t_0$, there holds
\begin{eq} \label{eq: assumption of f}
f(s)\leqslant\frac{C}{(s-t)^\alpha}f(t)\log^{-\beta}\left(1+\frac{1}{f(t)}\right)\comma
\end{eq}
where $C$, $\alpha$ and $\beta$ are positive constants and $\beta>\alpha$. Then for any $\gamma\in\left(1\comma\frac{\beta}{\alpha}\right]$, we have $f(t_0+T_\gamma)=0$, where
\begin{eq} \label{eq: T_gamma} \begin{aligned}
T_\gamma&\triangleq
\begin{dcases}
\frac{(C\mathrm{e})^{\frac{1}{\alpha}}\left(\frac{2}{\log2}\right)^\gamma}{(\gamma-1)\log^{\frac{\beta}{\alpha}-\gamma}\left(1+\frac{1}{f(t_0)}\right)}\comma & 0\leqslant f(t_0)\leqslant 1 \\
\frac{(C\mathrm{e})^{\frac{1}{\alpha}}2^\gamma}{(\gamma-1)\log^{\frac{\beta}{\alpha}}\left(1+\frac{1}{f(t_0)}\right)}\comma & f(t_0)>1
\end{dcases} \\
&=\frac{(C\mathrm{e})^{\frac{1}{\alpha}}\left(\frac{2}{\log2}\right)^\gamma}{\gamma-1} \max\left\{\frac{1}{\log^{\frac{\beta}{\alpha}-\gamma}\left(1+\frac{1}{f(t_0)}\right)}\comma\frac{\log^\gamma 2}{\log^{\frac{\beta}{\alpha}}\left(1+\frac{1}{f(t_0)}\right)}\right\}.
\end{aligned} \end{eq}
\end{thm}

\begin{proof}
We may as well assume that $f(t_0)>0$. We claim that for any $n\in\mathbb{N}$,
\begin{eq} \label{eq: induction claim}
f\bigl(t_0+(1-n^{1-\gamma})T_{\gamma}\bigr)\leqslant f(t_0)\mathrm{e}^{1-n}.
\end{eq}
When $n=1$, \eqref{eq: induction claim} is obviously true. Next we assume that \eqref{eq: induction claim} holds for $n=m\in\mathbb{N}$ and manage to prove that it also holds for $n=m+1$. In fact, by the assumption \eqref{eq: assumption of f} we have
\begin{eq} \label{eq: induction} \begin{aligned}
&\mathrel{\phantom{=}}f\Bigl(t_0+\bigl(1-(m+1)^{1-\gamma}\bigr)T_{\gamma}\Bigr) \\
&\leqslant\frac{Cf\bigl(t_0+(1-m^{1-\gamma})T_\gamma\bigr)}{\Bigl(\bigl(m^{1-\gamma}-(m+1)^{1-\gamma}\bigr)T_{\gamma}\Bigr)^\alpha} \log^{-\beta}\left(1+\frac{1}{f\bigl(t_0+(1-m^{1-\gamma})T_\gamma\bigr)}\right) \\
&\leqslant\frac{Cf(t_0)\mathrm{e}^{1-m}}{\bigl((\gamma-1)(m+1)^{-\gamma}T_{\gamma}\bigr)^\alpha\log^{\beta}\left(1+\frac{\mathrm{e}^{m-1}}{f(t_0)}\right)} \\
&=\frac{Cf(t_0)\mathrm{e}^{1-m}}{\bigl((\gamma-1)T_{\gamma}\bigr)^\alpha\log^{\beta-\gamma\alpha}\left(1+\frac{\mathrm{e}^{m-1}}{f(t_0)}\right)}\Biggl(\frac{1}{m+1}\log\left(1+\frac{\mathrm{e}^{m-1}}{f(t_0)}\right)\Biggr)^{-\gamma\alpha}\comma
\end{aligned} \end{eq}
where the second ``$\leqslant$'' is due to the hypothesis of induction and the following elementary inequality
\begin{eq}
a^{1-\mu}-b^{1-\mu}\geqslant(\mu-1)(b-a)b^{-\mu} (b\geqslant a>0\comma\mu\geqslant 1).
\end{eq}
Define
\begin{eq}
g(t)\triangleq\frac{1}{t+1}\log\left(1+\frac{\mathrm{e}^{t-1}}{f(t_0)}\right) (t\geqslant 1).
\end{eq}
Note that
\begin{ga}
g'(t)=-\frac{1}{(t+1)^2}\log\left(1+\frac{\mathrm{e}^{t-1}}{f(t_0)}\right)+\frac{1}{t+1}\frac{\mathrm{e}^{t-1}}{f(t_0)+\mathrm{e}^{t-1}}\comma \\
\bigl((t+1)^2g'(t)\bigr)'=(t+1)\frac{f(t_0)\mathrm{e}^{t-1}}{\bigl(f(t_0)+\mathrm{e}^{t-1}\bigr)^2}>0\comma \\
4g'(1)=-\log\left(1+\frac{1}{f(t_0)}\right)+\frac{2}{f(t_0)+1}\geqslant\frac{f(t_0)-1}{f(t_0)\bigl(f(t_0)+1\bigr)}.
\end{ga}
When $f(t_0)>1$, we have
\begin{eq}
(t+1)^2g'(t)\geqslant 4g'(1)>0 (t\geqslant 1)
\end{eq}
and therefore
\begin{eq}
g(m)\geqslant g(1)=\frac12\log\left(1+\frac{1}{f(t_0)}\right).
\end{eq}
Recalling \eqref{eq: induction} and the definition \eqref{eq: T_gamma} of $T_\gamma$, we obtain
\begin{eq}
f\Bigl(t_0+\bigl(1-(m+1)^{1-\gamma}\bigr)T_{\gamma}\Bigr)\leqslant\frac{Cf(t_0)\mathrm{e}^{1-m}2^{\gamma\alpha}}{\bigl((\gamma-1)T_\gamma\bigr)^\alpha\log^\beta\left(1+\frac{1}{f(t_0)}\right)}=f(t_0)\mathrm{e}^{-m}.
\end{eq}
On the other hand, when $0<f(t_0)\leqslant 1$, we have
\begin{eq}
f(m)\geqslant\frac{1}{m+1}\log(1+\mathrm{e}^{m-1})\geqslant\frac{\log2}{2},
\end{eq}
which indicates that
\begin{eq}
f\Bigl(t_0+\bigl(1-(m+1)^{1-\gamma}\bigr)T_{\gamma}\Bigr)\leqslant\frac{Cf(t_0)\mathrm{e}^{1-m}\left(\frac{2}{\log2}\right)^{\gamma\alpha}}{\bigl((\gamma-1)T_\gamma\bigr)^\alpha\log^{\beta-\gamma\alpha}\left(1+\frac{1}{f(t_0)}\right)}=f(t_0)\mathrm{e}^{-m}.
\end{eq}
In conclusion, our claim \eqref{eq: induction claim} holds for $n=m+1$ and therefore holds for any $n\in\mathbb{N}$ by induction. Since $f$ is non-negative and non-increasing, we have
\begin{eq}
0\leqslant f(t_0+T_\gamma)\leqslant f\bigl(t_0+(1-n^{1-\gamma})T_{\gamma}\bigr)\leqslant f(t_0)\mathrm{e}^{1-n} (\forall n\in\mathbb{N})\comma
\end{eq}
which indicates that $f(t_0+T_\gamma)=0$.
\end{proof}

The condition $\beta>\alpha$ in \thmref{thm: an iteration lemma of De Giorgi type} is necessary for the existence of a positive constant $T$ such that $f(t_0+T)=0$. See the following example.

\begin{ex}
Define
\begin{eq}
f(t)\triangleq\mathrm{e}^{-\mathrm{e}^t} (t\geqslant 0).
\end{eq}
Fix $\alpha>0$. Then for any $s\comma t\in\mathbb{R}\colon s>t\geqslant0$,
\begin{eq} \begin{aligned}
\frac{f(s)(s-t)^\alpha\log^\alpha\left(1+\frac{1}{f(t)}\right)}{f(t)}&=\mathrm{e}^{-(\mathrm{e}^s-\mathrm{e}^t)}(s-t)^\alpha\log^\alpha\left(1+\mathrm{e}^{\mathrm{e}^t}\right) \\
&\leqslant\mathrm{e}^{-\mathrm{e}^t(s-t)}(s-t)^\alpha\log^\alpha\left(\mathrm{e}^{2\mathrm{e}^t}\right) \\
&=2^\alpha\bigl((s-t)\mathrm{e}^t\bigr)^\alpha\mathrm{e}^{-(s-t)\mathrm{e}^t} \\
&\leqslant 2^\alpha\alpha^\alpha\mathrm{e}^{-\alpha}=\left(\frac{2\alpha}{\mathrm{e}}\right)^\alpha\comma
\end{aligned} \end{eq}
where the second ``$\leqslant$'' is due to the fact that the maximum of the function $x^\alpha\mathrm{e}^{-x}(x\geqslant 0)$ is attained at $x=\alpha$.
\end{ex}

A straightforward corollary of \thmref{thm: an iteration lemma of De Giorgi type} is as follows, which can be viewed as another version of the iteration lemma of De Giorgi type.

\begin{cor} \label{cor: another version of the iteration lemma of De Giorgi type}
Let $f$ be a non-negative, non-increasing function defined on $[t_0\comma{+}\infty)$ such that for any $s\comma t\in\mathbb{R}\colon s>t\geqslant t_0$, there holds
\begin{eq}
f(s)\leqslant\frac{C}{(s-t)^\alpha}f(t)\log^{-\beta}\left(1+\frac{1}{f(t)}\right)\comma
\end{eq}
where $C$, $\alpha$ and $\beta$ are positive constants and $\beta>\alpha$. Suppose that $f(t_0+T)>0$ for some positive constant $T$. Then for any $\gamma\in\left(1\comma\frac{\beta}{\alpha}\right)$, we have
\begin{eq}
f(t_0)>\frac{1}{\mathrm{e}^{L_\gamma}-1}\comma
\end{eq}
where
\begin{eq} \label{eq: L_gamma}
L_\gamma\triangleq\max\left\{\left(\frac{(C\mathrm{e})^{\frac{1}{\alpha}}\left(\frac{2}{\log2}\right)^\gamma}{(\gamma-1)T}\right)^{\frac{\alpha}{\beta-\gamma\alpha}}\comma \left(\frac{(C\mathrm{e})^{\frac{1}{\alpha}}2^\gamma}{(\gamma-1)T}\right)^{\frac{\alpha}{\beta}}\right\}.
\end{eq}
\end{cor}

\begin{proof}
It suffices to note that the constant $T_\gamma$ defined in \thmref{thm: an iteration lemma of De Giorgi type} satisfies $T_\gamma>T$.
\end{proof}

\thmref{thm: an iteration lemma of De Giorgi type} and \corref{cor: another version of the iteration lemma of De Giorgi type} will be used in \secref{sec: L^infty estimates on Kahler manifolds} and \secref{sec: L^infty estimates on Hermitian manifolds} respectively. We note that the iteration lemma of De Giorgi type \thmref{thm: an iteration lemma of De Giorgi type} is of independent interest since one can use it, together with the inequality of H\"older-Young type \thmref{thm: an inequality of Holder-Young type} maybe, to improve many other known conclusions whose proofs rely on De Giorgi iteration. For example, we apply this iteration lemma to the following improvement of a classical $\mathrm{L}^\infty$ estimate of $\mathrm{W}^{1\comma2}$ weak solutions to linear elliptic equations in divergence form. See e.g. \cite{Gilbarg2001} for the classical estimate.

\begin{thm} \label{thm: a classical L^infty estimate of W^1,2 weak solutions}
Fix $n\in\mathbb{N}\colon n\geqslant 3$ and $q>\frac n2$. Let $\Omega$ be a bounded domain in $\mathbb{R}^n$ with $\partial\Omega\in\mathrm{C}^1$, and $u\in\mathrm{W}^{1\comma2}(\Omega)$ be a $\mathrm{W}^{1\comma 2}$ weak solution to the equation
\begin{eq}
-\sum_{i\comma j=1}^n\mathrm{D}_j\bigl(a_{ij}(\bm{x})\mathrm{D}_iu\bigr)+c(\bm{x})u=\varphi(\bm{x})
\end{eq}
in $\Omega$, where $a_{ij}\in\mathrm{L}^\infty(\Omega)$, $c\in\mathrm{L}^{\frac n2}(\Omega)$ and $\varphi\in\mathrm{L}^{\frac n2}(\log\mathrm{L})^q(\Omega)$ are given functions such that $c\geqslant0$ and $(a_{ij})\geqslant\varepsilon\mathbf{I}_n$ for a positive constant $\varepsilon$. Then we have
\begin{eq} \label{eq: a classical L^infty estimate of W^1,2 weak solutions}
\|u\|_{\mathrm{L}^\infty(\Omega)}\leqslant\inf\Bigl\{s\geqslant0\Bigm|\bigl(|u|-s\bigr)_+\in\mathrm{W}_0^{1\comma2}(\Omega)\Bigr\}+C\|\varphi\|_{\mathrm{L}^{\frac n2}(\log\mathrm{L})^q(\Omega)}\comma
\end{eq}
where $\inf\varnothing\triangleq{+}\infty$ and $C$ is a positive constant depending only on $\Omega$, $n$, $q$ and $\frac{1}{\varepsilon}$.
\end{thm}

\begin{proof}
Let $B$ denote $\inf\Bigl\{s\geqslant0\Bigm|\bigl(|u|-s\bigr)_+\in\mathrm{W}_0^{1\comma2}(\Omega)\Bigr\}$. We may as well assume that $B\in[0\comma{+}\infty)$ and only need to prove that
\begin{eq} \label{eq: a classical maximum principle of W^1,2 weak solutions}
\esssup_\Omega u\leqslant B+C\|\varphi\|_{\mathrm{L}^{\frac n2}(\log\mathrm{L})^q(\Omega)}.
\end{eq}
Fix $t\geqslant B$ and define
\begin{eq}
v_t\triangleq(u-t)_+.
\end{eq}
Note that $v_t\in\mathrm{W}_0^{1\comma2}(\Omega)$. Thus we have
\begin{eq} \label{eq: weak solution}
\int_\Omega \left(\sum_{i\comma j=1}^n a_{ij}(\bm{x})\mathrm{D}_iu\mathrm{D}_jv_t+c(\bm{x})uv_t\right)\dif\bm{x}=\int_\Omega \varphi(\bm{x})v_t(\bm{x})\dif\bm{x}.
\end{eq}
Let $\Omega_t$ denote the set
\begin{eq}
\bigl\{\bm{x}\in\Omega\bigm|u(\bm{x})>t\bigr\}
\end{eq}
and $|\Omega_t|$ denote the Lebesgue measure of $\Omega_t$. The left-hand side of \eqref{eq: weak solution} reads
\begin{eq} \label{eq: weak solution, left-hand}
\int_{\Omega_t} \left(\sum_{i\comma j=1}^n a_{ij}(\bm{x})\mathrm{D}_iu\mathrm{D}_ju+c(\bm{x})u(u-t)\right)\dif\bm{x}.
\end{eq}
By the assumptions about $a_{ij}$, $c$ and the Poincar\'e inequality, we have
\begin{eq} \label{eq: weak solution, left-hand, further}
\eqref{eq: weak solution, left-hand} \geqslant \varepsilon\int_{\Omega_t} |\mathrm{D}u|^2\dif\bm{x}=\varepsilon\|\mathrm{D}v_t\|_{\mathrm{L}^2(\Omega)}^2\geqslant\frac{\varepsilon}{C_1}\|v_t\|_{\mathrm{L}^{\frac{2n}{n-2}}(\Omega)}^2.
\end{eq}
The right-hand side of \eqref{eq: weak solution} reads
\begin{eq} \label{eq: weak solution, right-hand}
\int_{\Omega_t} \varphi(\bm{x})v_t\dif\bm{x}.
\end{eq}
By the H\"older inequality and \thmref{thm: an inequality of Holder-Young type}, we have
\begin{eq} \label{eq: weak solution, right-hand, further} \begin{aligned}
\eqref{eq: weak solution, right-hand} &\leqslant\|v_t\|_{\mathrm{L}^{\frac{2n}{n-2}}(\Omega_t)}\left(\int_{\Omega_t} |\varphi|^{\frac{2n}{n+2}}\dif\bm{x}\right)^{\frac{n+2}{2n}} \\
&\leqslant\|v_t\|_{\mathrm{L}^{\frac{2n}{n-2}}(\Omega)}\left(\frac{C_2\left\||\varphi|^{\frac{2n}{n+2}}\right\|_{\mathrm{L}^{\frac{n+2}{4}}(\log\mathrm{L})^q(\Omega)}|\Omega_t|^{\frac{n-2}{n+2}}}{\log^{\frac{4q}{n+2}}\left(1+\frac{1}{|\Omega_t|}\right)}\right)^{\frac{n+2}{2n}} \\
&=C_2^{\frac{n+2}{2n}}\|v_t\|_{\mathrm{L}^{\frac{2n}{n-2}}(\Omega)}\frac{\left\||\varphi|^{\frac{2n}{n+2}}\right\|_{\mathrm{L}^{\frac{n+2}{4}}(\log\mathrm{L})^q(\Omega)}^{\frac{n+2}{2n}}|\Omega_t|^{\frac{n-2}{2n}}}{\log^{\frac{2q}{n}}\left(1+\frac{1}{|\Omega_t|}\right)}.
\end{aligned} \end{eq}
Combining \eqref{eq: weak solution}, \eqref{eq: weak solution, left-hand, further} and \eqref{eq: weak solution, right-hand, further}, we obtain
\begin{eq} \label{eq: v_t, norm estimate}
\|v_t\|_{\mathrm{L}^{\frac{2n}{n-2}}(\Omega)}^{\frac{2n}{n-2}}\leqslant\left(\frac{C_1C_2^{\frac{n+2}{2n}}}{\varepsilon}\right)^{\frac{2n}{n-2}}\frac{\left\||\varphi|^{\frac{2n}{n+2}}\right\|_{\mathrm{L}^{\frac{n+2}{4}}(\log\mathrm{L})^q(\Omega)}^{\frac{n+2}{n-2}}|\Omega_t|}{\log^{\frac{4q}{n-2}}\left(1+\frac{1}{|\Omega_t|}\right)}.
\end{eq}
To estimate $\left\||\varphi|^{\frac{2n}{n+2}}\right\|_{\mathrm{L}^{\frac{n+2}{4}}(\log\mathrm{L})^q(\Omega)}$, we note that
\begin{eq} \begin{aligned}
&\mathrel{\phantom{=}}\int_\Omega \left(\frac{|\varphi|^{\frac{2n}{n+2}}}{\|\varphi\|_{\mathrm{L}^{\frac n2}(\log\mathrm{L})^q(\Omega)}^{\frac{2n}{n+2}}}\right)^{\frac{n+2}{4}}\log^q\left(1+\frac{|\varphi|^{\frac{2n}{n+2}}}{\|\varphi\|_{\mathrm{L}^{\frac n2}(\log\mathrm{L})^q(\Omega)}^{\frac{2n}{n+2}}}\right)\dif\bm{x} \\
&\leqslant\left(\frac{2n}{n+2}\right)^q\int_\Omega \left(\frac{|\varphi|}{\|\varphi\|_{\mathrm{L}^{\frac n2}(\log\mathrm{L})^q}(\Omega)}\right)^{\frac{n}{2}}\log^q\left(1+\frac{|\varphi|}{\|\varphi\|_{\mathrm{L}^{\frac n2}(\log\mathrm{L})^q(\Omega)}}\right)\dif\bm{x}.
\end{aligned} \end{eq}
This inequality together with \eqref{eq: to be used in the proof of the inequality of Holder-Young type} and \propref{prop: estimates of norms} indicates that
\begin{eq} \label{eq: phi, norm estimate}
\left\||\varphi|^{\frac{2n}{n+2}}\right\|_{\mathrm{L}^{\frac{n+2}{4}}(\log\mathrm{L})^q(\Omega)}\leqslant\left(\frac{2n}{n+2}\right)^{\frac{4q}{n+2}}\|\varphi\|_{\mathrm{L}^{\frac n2}(\log\mathrm{L})^q(\Omega)}^{\frac{2n}{n+2}}.
\end{eq}
Thus, by \eqref{eq: v_t, norm estimate} and \eqref{eq: phi, norm estimate} we obtain
\begin{eq}
\|v_t\|_{\mathrm{L}^{\frac{2n}{n-2}}(\Omega)}^{\frac{2n}{n-2}}\leqslant C_3\|\varphi\|_{\mathrm{L}^{\frac n2}(\log\mathrm{L})^q(\Omega)}^{\frac{2n}{n-2}}\frac{|\Omega_t|}{\log^{\frac{4q}{n-2}}\left(1+\frac{1}{|\Omega_t|}\right)}.
\end{eq}
On the other hand, for any $s>t$, there holds
\begin{eq}
\|v_t\|_{\mathrm{L}^{\frac{2n}{n-2}}(\Omega)}^{\frac{2n}{n-2}}=\int_{\Omega_t} (u-t)^{\frac{2n}{n-2}}\dif\bm{x}\geqslant(s-t)^{\frac{2n}{n-2}}|\Omega_s|.
\end{eq}
As a result, we have
\begin{eq} \label{eq: Omega_s, estimate}
|\Omega_s|\leqslant \frac{C_3\|\varphi\|_{\mathrm{L}^{\frac n2}(\log\mathrm{L})^q(\Omega)}^{\frac{2n}{n-2}}}{(s-t)^{\frac{2n}{n-2}}}\frac{|\Omega_t|}{\log^{\frac{4q}{n-2}}\left(1+\frac{1}{|\Omega_t|}\right)}
\end{eq}
for any $s\comma t\in\mathbb{R}\colon s>t\geqslant B$. Since $q>\frac n2$ and $|\Omega_B|\leqslant|\Omega|$, \eqref{eq: Omega_s, estimate} implies 
\begin{eq}
\left|\Omega_{B+T_{\frac{2q}{n}}}\right|=0
\end{eq}
by \thmref{thm: an iteration lemma of De Giorgi type}, where
\begin{eq}
T_{\frac{2q}{n}}\triangleq\frac{(C_3\mathrm{e})^{\frac{n-2}{2n}}\|\varphi\|_{\mathrm{L}^{\frac n2}(\log\mathrm{L})^q(\Omega)}\left(\frac{2}{\log2}\right)^{\frac{2q}{n}}}{\frac{2q}{n}-1} \max\left\{1\comma\frac{\log^{\frac{2q}{n}} 2}{\log^{\frac{2q}{n}}\left(1+\frac{1}{|\Omega_B|}\right)}\right\}\leqslant C\|\varphi\|_{\mathrm{L}^{\frac n2}(\log\mathrm{L})^q(\Omega)}.
\end{eq}
It follows that \eqref{eq: a classical maximum principle of W^1,2 weak solutions} holds.
\end{proof}

\begin{rmk}
After completing this paper, we learned that \citeauthor{Cruz-Uribe2021} \cite{Cruz-Uribe2021} had proved a similar version of
\thmref{thm: a classical L^infty estimate of W^1,2 weak solutions}. Their proof was based on an iteration argument similar to but slightly different from \thmref{thm: an iteration lemma of De Giorgi type}. On the other hand, a more general version of
\thmref{thm: a classical L^infty estimate of W^1,2 weak solutions} was proved by \citeauthor{Cianchi1999} using totally different methods, see \cite[p. 200]{Cianchi1999}.
\end{rmk}


\section[Preliminaries to L∞ estimates of fully non-linear elliptic equations on complex manifolds]{Preliminaries to $\mathrm{L}^\infty$ estimates of fully non-linear elliptic equations on complex manifolds} \label{sec: Preliminaries to L^infty estimates of fully non-linear elliptic equations on complex manifolds}

In this section, we give some preliminaries to $\mathrm{L}^\infty$ estimates of fully non-linear elliptic equations on complex manifolds. Let $(\mathcal{M}\comma\bm{\omega}_{\bm g})$ be a compact Hermitian manifold of dimension $n$ with the Hermitian metric
\begin{eq}
\bm{g}\triangleq g_{j\overline{k}}\dif z^j\otimes\dif\overline{z}^k+\overline{g_{j\overline{k}}}\dif\overline{z}^j\otimes\dif z^k
\end{eq}
and the corresponding K\"ahler form (not necessarily closed)
\begin{eq}
\bm{\omega}_{\bm g}\triangleq\mathrm{i}g_{j\overline{k}}\dif z^j\wedge\dif\overline{z}^k.
\end{eq}
in the local holomorphic coordinate system $(U\mbox{; }z^j)$. The volume form of $(\mathcal{M}\comma\bm{\omega}_{\bm g})$ is
\begin{eq}
\bm{\omega}_{\bm g}^n=\mathrm{i}^nn!\det(g_{j\overline{k}})\dif z^1\wedge\dif\overline{z}^1\wedge\cdots\wedge\dif z^n\wedge\dif\overline{z}^n.
\end{eq}
Let $\left(g^{j\overline{k}}\right)$ denote
\begin{eq}
\left(\left(g_{j\overline{k}}\right)^{-1}\right)^{\mathrm T}.
\end{eq}
Then there hold
\begin{eq}
g^{j\overline l}g_{k\overline l}=g^{l\overline j}g_{l\overline k}=\delta_{jk}
\end{eq}
and
\begin{eq}
\frac{\partial g^{j\overline k}}{\partial z^l}=-g^{j\overline q}g^{p\overline k}\frac{\partial g_{p\overline q}}{\partial z^l}.
\end{eq}
Let
\begin{eq}
\bm\chi\triangleq\mathrm{i}\chi_{j\overline k}\dif z^j\wedge\dif\overline{z}^k.
\end{eq}
be a positive $(1\comma1)$-form on $\mathcal{M}$. For any $u\in\mathrm{C}^2(\mathcal M)$, let $\bm{\omega}_{\bm g}^{-1}(\bm\chi+\mathrm{i}\partial\overline\partial u)$ denote the matrix-valued function
\begin{eq}
\left(g^{j\overline k}\right)\left(\chi_{j\overline k}+\mathrm{D}_{j\overline k}u\right)^{\mathrm T}
\end{eq}
well-defined on $\mathcal M$, whose eigenvalues are all real. Fix $k\in\{1\comma2\comma\cdots\comma n\}$. One can prove that
\begin{eq}
(\bm\chi+\mathrm{i}\partial\overline\partial u)^k\wedge\bm{\omega}_{\bm g}^{n-k}=\frac{(n-k)!\,k!}{n!}\sigma_k\Bigl(\bm\lambda\left(\bm{\omega}_{\bm g}^{-1}(\bm\chi+\mathrm{i}\partial\overline\partial u)\right)\Bigr)\bm{\omega}_{\bm g}^n\comma
\end{eq}
where $\bm\lambda\bigl(\bm{\omega}_{\bm g}^{-1}(\bm{\chi}+\mathrm{i}\partial\overline\partial u)\bigr)$ represents the vector composed of the $n$ real eigenvalues of $\bm{\omega}_{\bm g}^{-1}(\bm{\chi}+\mathrm{i}\partial\overline\partial u)$ and $\sigma_k$ is the $k$-th elementary symmetric polynomial defined as
\begin{eq}
\sigma_k(\bm\lambda)\triangleq\sum_{1\leqslant i_1<i_2<\cdots<i_k\leqslant n}\lambda_{i_1}\lambda_{i_2}\cdots\lambda_{i_k}.
\end{eq}
The $k$-th G\r{a}rding cone $\Gamma_k$ is defined as
\begin{eq}
\bigl\{\bm\lambda\in\mathbb{R}^n\bigm|\sigma_l(\bm\lambda)>0\comma\forall l\in\{1\comma2\comma\cdots\comma k\}\bigr\}.
\end{eq}
Note that
\begin{al}
\Gamma_1&=\left\{\bm\lambda\in\mathbb{R}^n\middle|\lambda_1+\lambda_2+\cdots+\lambda_n>0\right\}\comma \\
\Gamma_n&=\bigl\{\bm\lambda\in\mathbb{R}^n\bigm|\lambda_i>0\comma\forall i\in\{1\comma2\comma\cdots\comma n\}\bigr\}\comma
\end{al}
and $\Gamma_k$ is a symmetric open convex cone in $\mathbb{R}^n$ satisfying $\Gamma_n\subset\Gamma\subset\Gamma_1$. $\Gamma_k$ is closely related to $\sigma_k$ especially when $k\geqslant 2$. In fact, the function $\sigma_k^{\frac 1k}$ is elliptic, in the sense of the condition \eqref{item: elliptic} on page \pageref{item: elliptic}, and concave in $\Gamma_k$. For any $l\in\{1\comma2\comma\cdots\comma k-1\}$, the function
\begin{eq}
\left(\frac{\sigma_k}{\sigma_l}\right)^{\frac{1}{k-l}}
\end{eq}
is also elliptic and concave in $\Gamma_k$. See e.g. \cite[pp.~401--407]{Lieberman1996} for the proofs of these statements. Moreover, $\sigma_k^{\frac 1k}$ satisfies the condition \eqref{item: structural} on page \pageref{item: structural}, see \cite{Guo2023}.

For the rest of this section, we always assume that $\Gamma$ is a symmetric open convex cone in $\mathbb{R}^n$ satisfying $\Gamma_n\subset\Gamma\subset\Gamma_1$ and $f$ is a positive $\mathrm{C}^1$ symmetric function defined in $\Gamma$ satisfying the conditions \eqref{item: elliptic}, \eqref{item: structural} and \eqref{eq: Euler inequality}. The $\Gamma$ is called the admissible cone with respect to the $f$. Consider the following fully non-linear elliptic equation
\begin{eq} \label{eq: equation}
f\Bigl(\bm\lambda\bigl(\bm{\omega}_{\bm g}^{-1}(\bm{\chi}+\mathrm{i}\partial\overline\partial u)\bigr)\Bigr)=\varphi
\end{eq}
on $(\mathcal{M}\comma\bm{\omega}_{\bm g})$, where $\varphi$ is a positive smooth function. When $f=\sigma_n^{\frac 1n} (\Gamma=\Gamma_n)$, \eqref{eq: equation} is equivalent to the Monge-Amp\`ere equation. When $f=\sigma_k^{\frac 1k} (\Gamma=\Gamma_k)$, \eqref{eq: equation} is equivalent to the $k$-Hessian equation
\begin{eq}
(\bm\chi+\mathrm{i}\partial\overline\partial u)^k\wedge\bm{\omega}_{\bm g}^{n-k}=\frac{(n-k)!\,k!}{n!}\varphi^k\bm{\omega}_{\bm g}^n.
\end{eq}
We call a $\mathrm{C}^2$ solution $u$ to \eqref{eq: equation} an admissible solution if it satisfies
\begin{eq} \label{eq: admissible}
\bm\lambda\bigl(\bm{\omega}_{\bm g}^{-1}(\bm{\chi}+\mathrm{i}\partial\overline\partial u)\bigr)\in\Gamma
\end{eq}
on $\mathcal M$. We are mainly interested in the admissible solutions of \eqref{eq: equation} because by the following proposition, \eqref{eq: equation} is an elliptic equation with respect to an admissible solution $u$ in the sense that the matrix
\begin{eq} \label{eq: (F^jk)}
\left(F_{\bm\chi+\mathrm{i}\partial\overline\partial u}^{j\overline k}\right)
\end{eq}
defined by
\begin{eq} \label{eq: F^jk} \begin{aligned}
F_{\bm\chi+\mathrm{i}\partial\overline\partial u}^{j\overline k}&\triangleq\left.\frac{\partial\log f\Bigl(\bm\lambda\bigl(\bm{\omega}_{\bm g}^{-1}\bm{\omega}\bigr)\Bigr)}{\partial\omega_{j\overline k}}\right|_{\bm\omega=\bm\chi+\mathrm{i}\partial\overline\partial u} \\
&=\frac{1}{f\Bigl(\bm\lambda\bigl(\bm{\omega}_{\bm g}^{-1}(\bm{\chi}+\mathrm{i}\partial\overline\partial u)\bigr)\Bigr)} \left.\frac{\partial f\Bigl(\bm\lambda\bigl(\bm{\omega}_{\bm g}^{-1}\bm{\omega}\bigr)\Bigr)}{\partial\omega_{j\overline k}}\right|_{\bm\omega=\bm\chi+\mathrm{i}\partial\overline\partial u}
\end{aligned} \end{eq}
is positive definite.

\begin{prop} \label{prop: properties of F^jk}
Let $u\in\mathrm{C}^2(\mathcal M)$ be an admissible solution to \eqref{eq: equation}.
\begin{enumerate}[(1)]
\item Assume that $g_{j\overline k}(\bm{z}_0)=\delta_{jk}$ and $\chi_{j\overline k}(\bm{z}_0)+\mathrm{D}_{j\overline k}u(\bm{z}_0)=\mu_j\delta_{jk}$ for a point $\bm{z}_0\in\mathcal M$, a local holomorphic coordinate system $(U\mbox{; }z^j)$ containing $\bm{z}_0$ and a vector $\bm\mu\in\mathbb{R}^n$. Then we have
\begin{eq}
F_{\bm\chi+\mathrm{i}\partial\overline\partial u}^{j\overline k}(\bm{z}_0)=\frac{1}{f(\bm\mu)}\left.\frac{\partial f(\bm\lambda)}{\partial \lambda_j}\right|_{\bm\lambda=\bm\mu}\delta_{jk}.
\end{eq}
\item The matrix \eqref{eq: (F^jk)} is positive definite and $F_{\bm\chi+\mathrm{i}\partial\overline\partial u}^{j\overline k}\chi_{j\overline k}\geqslant0$. 
\item $F_{\bm\chi+\mathrm{i}\partial\overline\partial u}^{j\overline k}\left(\chi_{j\overline k}+\mathrm{D}_{j\overline k}u\right)\leqslant\Lambda$. 
\item Let $\bm\omega$ be a positive $(1\comma1)$-form on $\mathcal{M}$. Then we have
\begin{eq}
F_{\bm\chi+\mathrm{i}\partial\overline\partial u}^{j\overline k}\omega_{j\overline k}\geqslant n\left(\frac{\delta\det\left(\bm{\omega}_{\bm g}^{-1}\bm{\omega}\right)}{\varphi^n}\right)^{\frac 1n}.
\end{eq}
\end{enumerate}
\end{prop}

\begin{proof}
See \cite[pp.~101--102]{Guo2023a}.
\end{proof}

At the end of this section, we state a $\mathrm{L}^1$ estimate of admissible functions, i.e. $\mathrm{C}^2$ functions satisfying \eqref{eq: admissible}, on compact Hermitian manifolds. This $\mathrm{L}^1$ estimate will be used in \secref{sec: L^infty estimates on Kahler manifolds} and \secref{sec: L^infty estimates on Hermitian manifolds}.

\begin{prop} \label{prop: L^1 estimate}
Let $u\in\mathrm{C}^2(\mathcal M)$ satisfy \eqref{eq: admissible}. Assume that $\bm\chi$ satisfies $\bm\chi\leqslant M\bm{\omega}_{\bm g}$ for some positive constant $M$. Then we have
\begin{eq}
\int_{\mathcal M}\left(\sup_{\mathcal M}u-u\right)\bm{\omega}_{\bm g}^n\leqslant C\comma
\end{eq}
where $C$ is a positive constant depending only on $\mathcal M$, $\bm\omega_{\bm g}$, $n$ and $M$.
\end{prop}

\begin{proof}
See \cite[p.~18]{Guo2023b} for the proof when $\bm\chi=\bm\omega_{\bm g}$. The proof for the general case is almost the same.
\end{proof}


\section[L∞ estimates on K\"ahler manifolds]{$\mathrm{L}^\infty$ estimates on K\"ahler manifolds} \label{sec: L^infty estimates on Kahler manifolds}

In this section, we prove \thmref{thm: L^infty estimates on Kahler manifolds}. Let $(\mathcal{M}\comma\bm{\omega}_{\bm g})$ be a compact K\"ahler manifold of dimension $n$, $\bm\chi$ be a K\"ahler form on $\mathcal M$ satisfying $\bm\chi\leqslant M\bm{\omega}_{\bm g}$ for some positive constant $M$, and $\varphi$ be a positive smooth function on $\mathcal M$. Assume that $u\in\mathrm{C}^2(\mathcal M)$ satisfies
\begin{eq}
\left\{ \begin{gathered}
f\Bigl(\bm\lambda\bigl(\bm{\omega}_{\bm g}^{-1}(\bm{\chi}+\mathrm{i}\partial\overline\partial u)\bigr)\Bigr)\leqslant\varphi\comma \\
\bm\lambda\bigl(\bm{\omega}_{\bm g}^{-1}(\bm{\chi}+\mathrm{i}\partial\overline\partial u)\bigr)\in\Gamma\comma
\end{gathered} \right.
\end{eq}
where $\Gamma$ is a symmetric open convex cone in $\mathbb{R}^n$ satisfying $\Gamma_n\subset\Gamma\subset\Gamma_1$ and $f$ is a positive $\mathrm{C}^1$ symmetric function defined in $\Gamma$ satisfying the conditions \eqref{item: elliptic}, \eqref{item: structural} and \eqref{eq: Euler inequality} on page \pageref{item: elliptic}. Let $V_{\bm g}$, $V_{\bm\chi}$ and $\Ent_{q\comma r}\left(\frac{\varphi^n}{V_{\bm\chi}}\right)$ denote $\int_{\mathcal M} \bm\omega_{\bm g}^n$, $\int_{\mathcal M} \bm\chi^n$ and
\begin{eq}
\left\|\frac{\varphi^n}{V_{\bm\chi}}\right\|_{\mathrm{L}^1(\log\mathrm{L})^q(\log\log\mathrm{L})^r(\mathcal M\comma\bm{\omega}_{\bm g}^n)}
\end{eq}
respectively (see \definitionref{definition: L^p(logL)^q(loglogL)^r}). First, we prove an energy estimate which improves that in \cite{Guo2022}. 

\begin{prop} \label{prop: energy estimate}
Fix $q>0$ and
\begin{eq}
a\in\begin{dcases}
\left(0\comma\frac{nq}{n-q}\right]\comma & q<n \\
(0\comma{+}\infty)\comma & q\geqslant n
\end{dcases}.
\end{eq}
Then we have
\begin{eq} \label{eq: energy estimate}
\int_{\mathcal M} \left(\sup_{\mathcal M}u-u\right)^a\frac{\varphi^n}{V_{\bm\chi}}\bm\omega_{\bm g}^n\leqslant C\left(\mathrm{e}^{C\left(\Ent_{q\comma0}\left(\frac{\varphi^n}{V_{\bm\chi}}\right)\right)^{\frac 1q}}-1\right)^a\Biggl(1+\Ent_{q\comma0}\left(\frac{\varphi^n}{V_{\bm\chi}}\right)\Biggr)\comma
\end{eq}
where $C$ is a positive constant depending only on $\mathcal M$, $\bm\omega_{\bm g}$, $n$, $q$, $a$, $\frac{1}{\delta}$, $\Lambda$, and $M$.
\end{prop}

\begin{proof}
We follow the strategy of \cite{Guo2022}. For $l\in\mathbb{N}$, let $\tau_l$ be a smooth non-decreasing function defined on $\mathbb{R}$ such that
\begin{eq} \label{eq: tau_l}
\tau_l(t)=\begin{dcases}
\frac{1}{2l}\comma & t\leqslant -\frac1l \\
t+\frac1l\comma & t\geqslant 0
\end{dcases}.
\end{eq}
See e.g. \propref{prop: glue two smooth convex functions defined on disjointed intervals} for the existence of such a function $\tau_l$. We may as well assume that $\osc\limits_{\mathcal M} u>0$. Fix $s\in\left[0\comma\osc\limits_{\mathcal M} u\right)$ and define
\begin{eq} \label{eq: tilde u_s}
\tilde{u}_s(\bm z)\triangleq \sup_{\mathcal M} u-u(\bm z)-s\comma\bm z\in\mathcal M.
\end{eq}
By the celebrated theorem of \citeauthor{Yau1978} \cite{Yau1978}, we can let $v_{s\comma l}$ be the unique smooth solution to the auxiliary complex Monge-Amp\`ere equation
\begin{eq} \label{eq: auxiliary complex Monge-Ampere}
\left\{ \begin{gathered}
\det\Bigl(\bm\omega_{\bm g}^{-1}\left(\bm\chi+\mathrm{i}\partial\overline\partial v_{s\comma l}\right)\Bigr)=\frac{\bigl(\tau_l\left(\tilde{u}_s\right)\bigr)^a}{A_{s\comma l}}\varphi^n\comma \\
\bm\chi+\mathrm{i}\partial\overline\partial v_{s\comma l}>\bm{0}\comma\sup_{\mathcal M} v_{s\comma l}=0.
\end{gathered} \right.
\end{eq}
on $(\mathcal M\comma\bm\omega_{\bm g})$, where the normalized constant $A_{s\comma l}$ is defined as
\begin{eq} \label{eq: A_sl}
A_{s\comma l}\triangleq\int_{\mathcal M} \bigl(\tau_l\left(\tilde{u}_s\right)\bigr)^a\frac{\varphi^n}{V_{\bm\chi}}\bm\omega_{\bm g}^n
\end{eq}
to make \eqref{eq: auxiliary complex Monge-Ampere} compatible. Let $\Omega_s$ denote the set
\begin{eq} \label{eq: Omega_s, energy estimate}
\bigl\{\bm{z}\in\mathcal M\bigm|\tilde{u}_s(\bm z)>0\bigr\}.
\end{eq}
Note that
\begin{eq} \label{eq: A_s}
\lim_{l\to\infty}A_{s\comma l}=A_s\triangleq\int_{\Omega_s} \tilde{u}_s^a\frac{\varphi^n}{V_{\bm\chi}}\bm\omega_{\bm g}^n>0. 
\end{eq}
Now we define the comparison function
\begin{eq}
w\triangleq-\nu(-v_{s\comma l}+B)^b+\tilde{u}_s\comma
\end{eq}
where $\nu$, $B$, $b$ are positive constants to be chosen later so that $w\leqslant 0$ on $\mathcal M$. Assume that $\max\limits_{\mathcal M} w=w(\bm{z}_0)$ for some point $\bm{z}_0\in\mathcal M$. If $\bm{z}_0\in\mathcal M\setminus\Omega_s$, it's easy to see that $w(\bm{z}_0)\leqslant 0$. Henceforth we  assume that $\bm{z}_0\in\Omega_s$ and choose a local holomorphic coordinate system $(U\mbox{; }z^j)$ containing $\bm{z}_0$. Straightforward calculations show that
\begin{eq}
\mathrm{D}_{j\overline k}w=\nu b(1-b)(-v_{s\comma l}+B)^{b-2}\mathrm{D}_jv_{s\comma l}\mathrm{D}_{\overline k}v_{s\comma l}+\nu b(-v_{s\comma l}+B)^{b-1}\mathrm{D}_{j\overline k}v_{s\comma l}-\mathrm{D}_{j\overline k}u.
\end{eq}
Note that the matrix $\left(\mathrm{D}_{j\overline k}w(\bm{z}_0)\right)$ is negative semi-definite since $\bm{z}_0$ is the maximum point of $w$. Recalling the definition \eqref{eq: F^jk} of $F_{\bm\chi+\mathrm{i}\partial\overline\partial u}^{j\overline k}$ and the property (2) in \propref{prop: properties of F^jk}, we assume that $b<1$ and calculate at $\bm{z}_0$ as follows:
\begin{eq} \begin{aligned}
0&\geqslant F_{\bm\chi+\mathrm{i}\partial\overline\partial u}^{j\overline k}\mathrm{D}_{j\overline k}w \\
&\geqslant\nu b(-v_{s\comma l}+B)^{b-1}F_{\bm\chi+\mathrm{i}\partial\overline\partial u}^{j\overline k}\mathrm{D}_{j\overline k}v_{s\comma l}-F_{\bm\chi+\mathrm{i}\partial\overline\partial u}^{j\overline k}\mathrm{D}_{j\overline k}u \\
&\geqslant\nu b(-v_{s\comma l}+B)^{b-1}F_{\bm\chi+\mathrm{i}\partial\overline\partial u}^{j\overline k}\left(\chi_{j\overline k}+\mathrm{D}_{j\overline k}v_{s\comma l}\right)-F_{\bm\chi+\mathrm{i}\partial\overline\partial u}^{j\overline k}\left(\chi_{j\overline k}+\mathrm{D}_{j\overline k}u\right) \\
&\phantom{=}+(1-\nu bB^{b-1})F_{\bm\chi+\mathrm{i}\partial\overline\partial u}^{j\overline k}\chi_{j\overline k}.
\end{aligned} \end{eq}
Assume further that
\begin{eq} \label{eq: B}
\text{$1-\nu bB^{b-1}=0$, i.e. $B=(\nu b)^{\frac{1}{1-b}}$}.
\end{eq}
Then we have at $\bm{z}_0$,
\begin{eq} \begin{aligned}
0&\geqslant\nu b(-v_{s\comma l}+B)^{b-1}F_{\bm\chi+\mathrm{i}\partial\overline\partial u}^{j\overline k}\left(\chi_{j\overline k}+\mathrm{D}_{j\overline k}v_{s\comma l}\right)-F_{\bm\chi+\mathrm{i}\partial\overline\partial u}^{j\overline k}\left(\chi_{j\overline k}+\mathrm{D}_{j\overline k}u\right) \\
&\geqslant\nu b(-v_{s\comma l}+B)^{b-1}n\left(\frac{\delta\tilde{u}_s^a}{A_{s\comma l}}\right)^{\frac 1n}\frac{\varphi}{f\Bigl(\bm\lambda\bigl(\bm{\omega}_{\bm g}^{-1}(\bm{\chi}+\mathrm{i}\partial\overline\partial u)\bigr)\Bigr)}-\Lambda \\
&\geqslant\nu b(-v_{s\comma l}+B)^{b-1}n\delta^{\frac 1n}A_{s\comma l}^{-\frac 1n}\tilde{u}_s^{\frac an}-\Lambda\comma
\end{aligned} \end{eq}
where the second ``$\geqslant$'' is due to the properties (3), (4) in \propref{prop: properties of F^jk} and the fact that at $\bm{z}_0$,
\begin{eq}
\tau_l\left(\tilde{u}_s\right)=\tilde{u}_s+\frac 1l.
\end{eq}
It follows that at $\bm{z}_0$,
\begin{eq}
\tilde{u}_s\leqslant\left(\frac{\Lambda}{bn\delta^{\frac 1n}}\right)^{\frac na}\nu^{-\frac na}A_{s\comma l}^{\frac 1a}(-v_{s\comma l}+B)^{\frac{n(1-b)}{a}}
\end{eq}
and
\begin{eq}
w\leqslant\nu(-v_{s\comma l}+B)^b\left(-1+\left(\frac{\Lambda}{bn\delta^{\frac 1n}}\right)^{\frac na}\nu^{-\frac na-1}A_{s\comma l}^{\frac 1a}(-v_{s\comma l}+B)^{\frac{n(1-b)}{a}-b}\right).
\end{eq}
Now we choose the constants $b$, $\nu$ and $B$ as follows:
\begin{ga} \label{eq: b, nu and B}
\text{$\frac{n(1-b)}{a}-b=0$, i.e. $b=\frac{n}{n+a}\in(0\comma1)$}\comma \\
\text{$-1+\left(\frac{\Lambda}{bn\delta^{\frac 1n}}\right)^{\frac na}\nu^{-\frac na-1}A_{s\comma l}^{\frac 1a}=0$, i.e. $\nu=\left(\frac{(n+a)\Lambda}{n^2\delta^{\frac 1n}}\right)^{\frac{n}{n+a}}A_{s\comma l}^{\frac{1}{n+a}}$}\comma \\
B=(\nu b)^{\frac{1}{1-b}}=\left(\frac{n}{n+a}\left(\frac{(n+a)\Lambda}{n^2\delta^{\frac 1n}}\right)^{\frac{n}{n+a}}\right)^{\frac{n+a}{a}}A_{s\comma l}^{\frac 1a}.
\end{ga}
Then we obtain $\max\limits_{\mathcal M}w=w(\bm{z}_0)\leqslant 0$, i.e.
\begin{eq} \label{eq: comparison inequality}
\tilde{u}_s\leqslant\nu(-v_{s\comma l}+B)^b=C_1A_{s\comma l}^{\frac{1}{n+a}}\left(-v_{s\comma l}+C_2A_{s\comma l}^{\frac 1a}\right)^{\frac{n}{n+a}}
\end{eq}
on $\mathcal M$. Thus, on $\Omega_s$ we have 
\begin{eq} \label{eq: key estimate}
\frac{\tilde{u}_s^{\frac{n+a}{n}q}}{A_{s\comma l}^{\frac qn}}\leqslant C_1^{\frac{n+a}{n}q}\left(-v_{s\comma l}+C_2A_{s\comma l}^{\frac 1a}\right)^q\leqslant C_1^{\frac{n+a}{n}q}\max\{1\comma 2^{q-1}\}\left((-v_{s\comma l})^q+C_2^qA_{s\comma l}^{\frac qa}\right).
\end{eq}
Note that $M\bm\omega_{\bm g}+\mathrm{i}\partial\overline\partial v_{s\comma l}>\bm{0}$. In view of the well-known lemma related to the $\alpha$-invariants in \cite[pp.~228--229]{Tian1987} , there exists a positive constant $\alpha$ depending only on $\mathcal M$, $\bm\omega_{\bm g}$, $n$ and $M$ such that
\begin{eq} \label{eq: alpha-invariant}
\int_{\mathcal M} \mathrm{e}^{-\alpha v_{s\comma l}}\bm\omega_{\bm g}^n\leqslant C_3,
\end{eq}
where $C_3$ is a positive constant depending only on $\mathcal M$, $\bm\omega_{\bm g}$ and $n$. By the classical Young inequality we have
\begin{eq} \begin{aligned}
(-v_{s\comma l})^q\psi&\leqslant\int_0^{(-v_{s\comma l})^q} \left(\mathrm{e}^{\frac{\alpha}{2}t^{\frac 1q}}-1\right)\dif t+\int_0^\psi \left(\frac{2}{\alpha}\log(1+t)\right)^q\dif t \\
&\leqslant(-v_{s\comma l})^q\mathrm{e}^{-\frac{\alpha}{2}v_{s\comma l}}+\left(\frac{2}{\alpha}\right)^q\psi\log^q(1+\psi) \\
&\leqslant\left(\frac{2q}{\mathrm{e}\alpha}\right)^q\mathrm{e}^{-\alpha v_{s\comma l}}+\left(\frac{2}{\alpha}\right)^q\psi\log^q(1+\psi)
\end{aligned} \end{eq}
for any positive smooth function $\psi$ on $\mathcal M$. Letting 
\begin{eq} \label{eq: psi}
\psi=\frac{\frac{\varphi^{n}}{V_{\bm\chi}}}{\Ent_{q\comma 0}\left(\frac{\varphi^{n}}{V_{\bm\chi}}\right)}\comma
\end{eq}
and integrating both sides of the acquired inequality over $\mathcal M$, we obtain
\begin{eq}
\int_{\mathcal M} (-v_{s\comma l})^q\frac{\varphi^{n}}{V_{\bm\chi}}\bm\omega_{\bm g}^n\leqslant\Biggl(\left(\frac{2q}{\mathrm{e}\alpha}\right)^qC_3+\left(\frac{2}{\alpha}\right)^q\Biggr) \Ent_{q\comma 0}\left(\frac{\varphi^{n}}{V_{\bm\chi}}\right)
\end{eq}
by \eqref{eq: alpha-invariant} and \eqref{eq: to be used in the proof of the inequality of Holder-Young type}. This inequality together with \eqref{eq: key estimate} implies
\begin{eq}
\int_{\Omega_s} \frac{\tilde{u}_s^{\frac{n+a}{n}q}}{A_{s\comma l}^{\frac qn}}\frac{\varphi^n}{V_{\bm\chi}}\bm\omega_{\bm g}^n\leqslant C_4\Ent_{q\comma 0}\left(\frac{\varphi^{n}}{V_{\bm\chi}}\right)+C_5A_{s\comma l}^{\frac qa}\int_{\Omega_s} \frac{\varphi^n}{V_{\bm\chi}}\bm\omega_{\bm g}^n.
\end{eq}
Note that the above inequality holds for any $l\in\mathbb{N}$. Letting $l\to\infty$ and recalling \eqref{eq: A_s}, we obtain
\begin{eq} \label{eq: key estimate, further}
\int_{\Omega_s} \tilde{u}_s^{\frac{n+a}{n}q}\frac{\varphi^n}{V_{\bm\chi}}\bm\omega_{\bm g}^n\leqslant C_4\Ent_{q\comma 0}\left(\frac{\varphi^{n}}{V_{\bm\chi}}\right)A_s^{\frac qn}+C_5A_s^{\frac{n+a}{na}q}h(s)\comma
\end{eq}
where $h(s)$ denotes
\begin{eq} \label{eq: h(s)}
\int_{\Omega_s} \frac{\varphi^n}{V_{\bm\chi}}\bm\omega_{\bm g}^n.
\end{eq}
Since $a\leqslant\frac{nq}{n-q}$ when $q<n$, by the H\"older inequality and \eqref{eq: key estimate, further} we have
\begin{eq} \label{eq: A_s, estimate} \begin{aligned}
A_s&=\int_{\Omega_s} \tilde{u}_s^a\frac{\varphi^n}{V_{\bm\chi}}\bm\omega_{\bm g}^n \\
&\leqslant\left(\int_{\Omega_s} \tilde{u}_s^{\frac{n+a}{n}q}\frac{\varphi^n}{V_{\bm\chi}}\bm\omega_{\bm g}^n\right)^{\frac{na}{(n+a)q}}\left(\int_{\Omega_s} \frac{\varphi^n}{V_{\bm\chi}}\bm\omega_{\bm g}^n\right)^{\frac{(n+a)q-na}{(n+a)q}} \\
&\leqslant\Biggl(C_4\Ent_{q\comma 0}\left(\frac{\varphi^{n}}{V_{\bm\chi}}\right)A_s^{\frac qn}+C_5A_s^{\frac{n+a}{na}q}h(s)\Biggr)^{\frac{na}{(n+a)q}} \bigl(h(s)\bigr)^{\frac{(n+a)q-na}{(n+a)q}} \\
&\leqslant C_6\Biggl(\Ent_{q\comma 0}\left(\frac{\varphi^{n}}{V_{\bm\chi}}\right)\Biggr)^{\frac{na}{(n+a)q}}A_s^{\frac{a}{n+a}}\bigl(h(s)\bigr)^{\frac{(n+a)q-na}{(n+a)q}}+C_7A_sh(s)
\end{aligned} \end{eq}
for any $s\in\left[0\comma\osc\limits_{\mathcal M} u\right)$. On the other hand, by the classical Young inequality and \propref{prop: L^1 estimate} we have
\begin{eq} \begin{aligned}
&\mathrel{\phantom{=}}\int_{\Omega_s} 2^{-q}\log^q\left(\sup\limits_{\mathcal M}u-u+2\right)\psi\bm\omega_{\bm g}^n \\
&\leqslant\int_{\Omega_s} \left(\int_0^{2^{-q}\log^q\left(\sup\limits_{\mathcal M}u-u+2\right)} \left(\mathrm{e}^{t^{\frac 1q}}-1\right)\dif t+\int_0^\psi \log^q(1+t)\dif t\right)\bm\omega_{\bm g}^n \\
&\leqslant\int_{\Omega_s} \left(2^{-q}\log^q\left(\sup\limits_{\mathcal M}u-u+2\right)\left(\sup\limits_{\mathcal M}u-u+2\right)^{\frac 12}+\psi\log^q(1+\psi)\right)\bm\omega_{\bm g}^n \\
&\leqslant\left(\frac{q}{\mathrm e}\right)^q\int_{\Omega_s} \left(\sup\limits_{\mathcal M}u-u+2\right)\bm\omega_{\bm g}^n+\int_{\Omega_s} \psi\log^q(1+\psi)\bm\omega_{\bm g}^n \\
&\leqslant\left(\frac{q}{\mathrm e}\right)^qC_8+\int_{\mathcal M} \psi\log^q(1+\psi)\bm\omega_{\bm g}^n
\end{aligned} \end{eq}
for any $s\geqslant 0$ and any positive smooth function $\psi$ on $\mathcal M$. Defining $\psi$ as that in \eqref{eq: psi} and recalling the definition \eqref{eq: h(s)} of $h(s)$, we obtain
\begin{eq}
h(s)\leqslant C_9\log^{-q}(s+2)\Ent_{q\comma0}\left(\frac{\varphi^n}{V_{\bm\chi}}\right)
\end{eq}
for any $s\geqslant0$ by \eqref{eq: to be used in the proof of the inequality of Holder-Young type}. In particular, for
\begin{eq}
s_0\triangleq\mathrm{e}^{\left(2C_7C_9\Ent_{q\comma0}\left(\frac{\varphi^n}{V_{\bm\chi}}\right)\right)^{\frac 1q}}-1
\end{eq}
there holds
\begin{eq}
h(s_0)\leqslant\frac{1}{2C_7}.
\end{eq}
If $s_0\geqslant\osc\limits_{\mathcal M}u$, we are done. Henceforth we assume that $s_0<\osc\limits_{\mathcal M}u$ and therefore $A_{s_0}>0$. In view of \eqref{eq: A_s, estimate}, we have
\begin{eq}
A_{s_0}\leqslant C_{10}\Biggl(\Ent_{q\comma 0}\left(\frac{\varphi^n}{V_{\bm\chi}}\right)\Biggr)^{\frac aq}.
\end{eq}
Recalling the definition \eqref{eq: A_s} of $A_{s_0}$, we see that
\begin{eq}
\int_{\Omega_{s_0}} \left(\sup\limits_{\mathcal M}u-u-s_0\right)^a\frac{\varphi^n}{V_{\bm\chi}}\bm\omega_{\bm g}^n\leqslant C_{10}\Biggl(\Ent_{q\comma 0}\left(\frac{\varphi^n}{V_{\bm\chi}}\right)\Biggr)^{\frac aq}.
\end{eq}
It follows that 
\begin{eq} \begin{aligned}
&\mathrel{\phantom{=}}\int_{\mathcal M} \left(\sup_{\mathcal M}u-u\right)^a\frac{\varphi^n}{V_{\bm\chi}}\bm\omega_{\bm g}^n \\
&\leqslant \int_{\Omega_{s_0}} \left(\sup_{\mathcal M}u-u-s_0+s_0\right)^a\frac{\varphi^n}{V_{\bm\chi}}\bm\omega_{\bm g}^n+s_0^a\int_{\mathcal M\setminus\Omega_{s_0}} \frac{\varphi^n}{V_{\bm\chi}}\bm\omega_{\bm g}^n \\
&\leqslant\max\{1\comma2^{a-1}\}\Biggl(\int_{\Omega_{s_0}} \left(\sup_{\mathcal M}u-u-s_0\right)^a\frac{\varphi^n}{V_{\bm\chi}}\bm\omega_{\bm g}^n+s_0^a\int_{\mathcal M} \frac{\varphi^n}{V_{\bm\chi}}\bm\omega_{\bm g}^n\Biggr) \\
&\leqslant\max\{1\comma2^{a-1}\}C_{10}\Biggl(\Ent_{q\comma 0}\left(\frac{\varphi^n}{V_{\bm\chi}}\right)\Biggr)^{\frac aq} \\
&\phantom{=}+\max\{1\comma2^{a-1}\}\left(\mathrm{e}^{\left(2C_7C_9\Ent_{q\comma0}\left(\frac{\varphi^n}{V_{\bm\chi}}\right)\right)^{\frac 1q}}-1\right)^a\int_{\mathcal M} \frac{\varphi^n}{V_{\bm\chi}}\bm\omega_{\bm g}^n.\end{aligned} \end{eq}
Finally, by the classical Young inequality and \eqref{eq: to be used in the proof of the inequality of Holder-Young type} it's easy to see that
\begin{eq}
\frac{1}{\Ent_{q\comma0}\left(\frac{\varphi^n}{V_{\bm\chi}}\right)} \int_{\mathcal M} \frac{\varphi^n}{V_{\bm\chi}}\bm\omega_{\bm g}^n\leqslant(\mathrm{e}-1)V_{\bm g}+1\comma
\end{eq}
and therefore we have proved \eqref{eq: energy estimate}.
\end{proof}

Next we prove the sharp $\mathrm{L}^\infty$ estimate in \thmref{thm: L^infty estimates on Kahler manifolds}. Note that for any positive smooth function $\psi$ on $\mathcal M$, $\Ent_{n\comma 0}(\psi)$ can be controlled by $\Ent_{n\comma r}(\psi) (r>0)$ since we have
\begin{eq} \begin{aligned}
&\mathrel{\phantom{=}}\int_{\mathcal M} \frac{\psi}{\Ent_{n\comma r}(\psi)}\log^n\left(1+\frac{\psi}{\Ent_{n\comma r}(\psi)}\right)\bm\omega_{\bm g}^n \\
&\leqslant(\mathrm{e}-1)V_{\bm g}+\frac{1}{\log^r2}\int_{\left\{\frac{\psi}{\Ent_{n\comma r}(\psi)}\geqslant\mathrm{e}-1\right\}} \frac{\psi}{\Ent_{n\comma r}(\psi)}\log^n\left(1+\frac{\psi}{\Ent_{n\comma r}(\psi)}\right) \\
&\phantom{\leqslant(\mathrm{e}-1)V_{\bm g}+\frac{1}{\log^r2}\int_{\left\{\frac{\psi}{\Ent_{n\comma r}(\psi)}\geqslant\mathrm{e}-1\right\}}} \cdot\log^r\biggl(1+\log\left(1+\frac{\psi}{\Ent_{n\comma r}(\psi)}\right)\biggr)\bm\omega_{\bm g}^n \\
&\leqslant \frac{1}{\log^r2}+(\mathrm{e}-1)V_{\bm g}
\end{aligned} \end{eq}
by \eqref{eq: to be used in the proof of the inequality of Holder-Young type} and therefore
\begin{eq} \label{eq: Ent_n0 controlled by Ent_nr}
\Ent_{n\comma 0}(\psi)\leqslant\max\left\{1\comma\frac{1}{\log^r2}+(\mathrm{e}-1)V_{\bm g}\right\}\Ent_{n\comma r}(\psi)
\end{eq}
by \propref{prop: estimates of norms}.

\begin{thm}
Fix $r>n$. Then we have
\begin{eq} \label{eq: L^infty estimate}
\osc\limits_{\mathcal M} u\leqslant C\comma
\end{eq}
where $C$ is a positive constant depending only on $\mathcal{M}$, $\bm{\omega}_{\bm g}$, $n$, $r$, $\frac{1}{\delta}$, $\Lambda$, $M$ and $\Ent_{n\comma r}\left(\frac{\varphi^n}{V_{\bm\chi}}\right)$.
\end{thm}

\begin{proof}
The first half of the proof follows the strategy of \cite{Guo2023} and is similar to that of \propref{prop: energy estimate}. For $l\in\mathbb{N}$, let $\tau_l$ be a smooth non-decreasing function defined on $\mathbb{R}$ satisfying \eqref{eq: tau_l}. We may as well assume that $\osc\limits_{\mathcal M} u>0$. Fix $s\in\left[0\comma\osc\limits_{\mathcal M} u\right)$ and define $\tilde{u}_s$ as that in \eqref{eq: tilde u_s}. By the celebrated theorem of \citeauthor{Yau1978} \cite{Yau1978}, we can let $v_{s\comma l}$ be the unique smooth solution to the auxiliary complex Monge-Amp\`ere equation
\begin{eq} \label{eq: auxiliary complex Monge-Ampere, L^infty estimate}
\left\{ \begin{gathered}
\det\Bigl(\bm\omega_{\bm g}^{-1}\left(\bm\chi+\mathrm{i}\partial\overline\partial v_{s\comma l}\right)\Bigr)=\frac{\tau_l(\tilde{u}_s)}{A_{s\comma l}}\varphi^n\comma \\
\bm\chi+\mathrm{i}\partial\overline\partial v_{s\comma l}>\bm{0}\comma\sup_{\mathcal M} v_{s\comma l}=0.
\end{gathered} \right.
\end{eq}
on $(\mathcal M\comma\bm\omega_{\bm g})$, where the normalized constant $A_{s\comma l}$ is defined as
\begin{eq}
A_{s\comma l}\triangleq\int_{\mathcal M} \tau_l(\tilde{u}_s)\frac{\varphi^n}{V_{\bm\chi}}\bm\omega_{\bm g}^n
\end{eq}
to make \eqref{eq: auxiliary complex Monge-Ampere, L^infty estimate} compatible. Let $\Omega_s$ denote the set \eqref{eq: Omega_s, energy estimate}. Note that
\begin{eq}
\lim_{l\to\infty}A_{s\comma l}=A_s\triangleq\int_{\Omega_s} \tilde{u}_s\frac{\varphi^n}{V_{\bm\chi}}\bm\omega_{\bm g}^n>0. 
\end{eq}
By arguments almost the same as those in the proof of \propref{prop: energy estimate}, we obtain the following comparison inequality
\begin{eq}
\tilde{u}_s\leqslant C_1A_{s\comma l}^{\frac{1}{n+1}}\left(-v_{s\comma l}+C_2A_{s\comma l}\right)^{\frac{n}{n+1}}
\end{eq}
on $\mathcal M$ similar to \eqref{eq: comparison inequality}. Thus, on $\Omega_s$ we have the following key estimate
\begin{eq}
C_1^{-\frac{n+1}{n}}A_{s\comma l}^{-\frac 1n}\tilde{u}_s^{\frac{n+1}{n}}\leqslant -v_{s\comma l}+C_2A_{s\comma l}.
\end{eq}
Recalling \eqref{eq: alpha-invariant}, we see that
\begin{eq}
\int_{\Omega_s} \mathrm{e}^{\alpha C_1^{-\frac{n+1}{n}}A_{s\comma l}^{-\frac 1n}\tilde{u}_s^{\frac{n+1}{n}}}\bm\omega_{\bm g}^n\leqslant C_3\mathrm{e}^{\alpha C_2A_{s\comma l}}.
\end{eq}
Note that the above inequality holds for any $l\in\mathbb{N}$. Letting $l\to\infty$, we obtain
\begin{eq} \label{eq: key inequality in L^infty estimates}
\int_{\Omega_s} \mathrm{e}^{\alpha C_1^{-\frac{n+1}{n}}A_s^{-\frac 1n}\tilde{u}_s^{\frac{n+1}{n}}}\bm\omega_{\bm g}^n\leqslant C_3\mathrm{e}^{\alpha C_2A_s}.
\end{eq}
Define
\begin{eq}
\eta(t)\triangleq t^n\log^r\left(1+t^{\frac{n}{n+1}}\right) (t\geqslant0).
\end{eq}
Let $\eta^{-1}$ denote the inverse function of $\eta$. By the classical Young inequality, for any $x\comma y\geqslant0$ there holds
\begin{eq} \begin{aligned}
\eta(x)y&\leqslant\int_0^{\eta(x)} \left(\mathrm{e}^{\eta^{-1}(t)}-1\right)\dif t+\int_0^y \eta\bigl(\log(1+t)\bigr)\dif t \\
&\leqslant\eta(x)(\mathrm{e}^x-1)+y\eta\bigl(\log(1+y)\bigr) \\
&\leqslant C_4\mathrm{e}^{2x}+y\log^n(1+y)\log^r\left(1+\log^{\frac{n}{n+1}}(1+y)\right).
\end{aligned} \end{eq}
Letting 
\begin{eq}
x=\left(\frac{\tilde{u}_s(\bm{z})}{\left(\frac{2}{\alpha}\right)^{\frac{n}{n+1}}C_1A_s^{\frac{1}{n+1}}}\right)^{\frac{n+1}{n}}\comma\quad y=\frac{\frac{\bigl(\varphi(\bm{z})\bigr)^{n}}{V_{\bm\chi}}}{\Ent_{n\comma r}\left(\frac{\varphi^{n}}{V_{\bm\chi}}\right)}\comma
\end{eq}
and integrating both sides of the acquired inequality over $\Omega_s$, we obtain
\begin{eq} \begin{aligned}
&\mathrel{\phantom{=}}\int_{\Omega_s} \left(\frac{\tilde{u}_s}{\left(\frac{2}{\alpha}\right)^{\frac{n}{n+1}}C_1A_s^{\frac{1}{n+1}}}\right)^{n+1}\log^r\left(1+\frac{\tilde{u}_s}{\left(\frac{2}{\alpha}\right)^{\frac{n}{n+1}}C_1A_s^{\frac{1}{n+1}}}\right) \frac{\varphi^{n}}{V_{\bm\chi}}\bm\omega_{\bm g}^n \\
&\leqslant \biggl(C_4C_3\mathrm{e}^{\alpha C_2A_s}+\int_{\Omega_s\cap\{y>\mathrm{e}-1\}} y\log^n(1+y)\log^r\bigl(1+\log(1+y)\bigr)\bm\omega_{\bm g}^n \\
&\phantom{\leqslant \biggl(}+(\mathrm{e}-1)\log^r2\cdot V_{\bm g}\biggr)\Ent_{n\comma r}\left(\frac{\varphi^{n}}{V_{\bm\chi}}\right) \\
&\leqslant \left(C_4C_3\mathrm{e}^{\alpha C_2A_s}+C_5\right)\Ent_{n\comma r}\left(\frac{\varphi^{n}}{V_{\bm\chi}}\right)
\end{aligned} \end{eq}
by \eqref{eq: key inequality in L^infty estimates} and \eqref{eq: to be used in the proof of the inequality of Holder-Young type}. By \propref{prop: energy estimate} and \eqref{eq: Ent_n0 controlled by Ent_nr} we note that
\begin{eq} \begin{aligned}
A_s&=\int_{\Omega_s} \tilde{u}_s\frac{\varphi^n}{V_{\bm\chi}}\bm\omega_{\bm g}^n \\
&\leqslant\int_{\mathcal M} \left(\sup_{\mathcal M}u-u\right)\frac{\varphi^n}{V_{\bm\chi}}\bm\omega_{\bm g}^n \\
&\leqslant C_6\left(\mathrm{e}^{C_6\left(\Ent_{n\comma 0}\left(\frac{\varphi^n}{V_{\bm\chi}}\right)\right)^{\frac 1n}}-1\right)\Biggl(1+\Ent_{n\comma 0}\left(\frac{\varphi^n}{V_{\bm\chi}}\right)\Biggr) \\
&\leqslant C_6\left(\mathrm{e}^{C_6\left(C_7\Ent_{n\comma r}\left(\frac{\varphi^n}{V_{\bm\chi}}\right)\right)^{\frac 1n}}-1\right)\Biggl(1+C_7\Ent_{n\comma r}\left(\frac{\varphi^n}{V_{\bm\chi}}\right)\Biggr).
\end{aligned} \end{eq}
Thus, we have
\begin{eq}
\int_{\Omega_s} \left(\frac{\tilde{u}_s}{\left(\frac{2}{\alpha}\right)^{\frac{n}{n+1}}C_1A_s^{\frac{1}{n+1}}}\right)^{n+1}\log^r\left(1+\frac{\tilde{u}_s}{\left(\frac{2}{\alpha}\right)^{\frac{n}{n+1}}C_1A_s^{\frac{1}{n+1}}}\right) \frac{\varphi^{n}}{V_{\bm\chi}}\bm\omega_{\bm g}^n\leqslant C_8
\end{eq}
for a positive constant $C_8$ depending on $\Ent_{n\comma r}\left(\frac{\varphi^{n}}{V_{\bm\chi}}\right)$ in addition to other trivial quantities. By \propref{prop: estimates of norms} the inequality above indicates that
\begin{eq}
\left\|\tilde{u}_s\right\|_{\mathrm{L}^{n+1}(\log\mathrm{L})^r\left(\Omega_s\comma\frac{\varphi^n}{V_{\bm\chi}}\bm\omega_{\bm g}^n\right)}\leqslant C_8^{\frac{1}{n+1}}\left(\frac{2}{\alpha}\right)^{\frac{n}{n+1}}C_1A_s^{\frac{1}{n+1}}.
\end{eq}
Let $h(s)$ denote
\begin{eq}
\int_{\Omega_s} \frac{\varphi^n}{V_{\bm\chi}}\bm\omega_{\bm g}^n.
\end{eq}
Then by \thmref{thm: an inequality of Holder-Young type} we have
\begin{eq} \begin{aligned}
A_s&=\int_{\Omega_s} \tilde{u}_s\frac{\varphi^n}{V_{\bm\chi}}\bm\omega_{\bm g}^n \\
&\leqslant\frac{2\left(n+1+\frac{r}{2}\right)^{\frac{r}{n+1}}\left\|\tilde{u}_s\right\|_{\mathrm{L}^{n+1}(\log\mathrm{L})^r\left(\Omega_s\comma\frac{\varphi^n}{V_{\bm\chi}}\bm\omega_{\bm g}^n\right)}\bigl(h(s)\bigr)^{\frac{n}{n+1}}}{\log^{\frac{r}{n+1}}\left(1+\frac{1}{h(s)}\right)} \\
&\leqslant \frac{2\left(n+1+\frac{r}{2}\right)^{\frac{r}{n+1}}C_8^{\frac{1}{n+1}}\left(\frac{2}{\alpha}\right)^{\frac{n}{n+1}}C_1A_s^{\frac{1}{n+1}}\bigl(h(s)\bigr)^{\frac{n}{n+1}}}{\log^{\frac{r}{n+1}}\left(1+\frac{1}{h(s)}\right)}\comma
\end{aligned} \end{eq}
and therefore
\begin{eq}
A_s\leqslant \frac{C_9h(s)}{\log^{\frac{r}{n}}\left(1+\frac{1}{h(s)}\right)}
\end{eq}
for any $s\in\left[0\comma\osc\limits_{\mathcal M}u\right)$, where $C_9$ is a positive constant depending on $\Ent_{n\comma r}\left(\frac{\varphi^{n}}{V_{\bm\chi}}\right)$ and other trivial quantities. Note that the above inequality holds for any $s\geqslant0$ in fact. On the other hand, for any $t>s$ there holds
\begin{eq}
A_s\geqslant\int_{\Omega_t} \left(\sup_{\mathcal M}u-u-s\right)\frac{\varphi^n}{V_{\bm\chi}}\bm\omega_{\bm g}^n\geqslant(t-s)h(t).
\end{eq}
As a result, we have
\begin{eq}
h(t)\leqslant\frac{C_9}{t-s}\frac{h(s)}{\log^{\frac{r}{n}}\left(1+\frac{1}{h(s)}\right)}
\end{eq}
for any $t\comma s\in\mathbb{R}\colon t>s\geqslant0$. Since $r>n$, this indicates by \thmref{thm: an iteration lemma of De Giorgi type} that $h\left(T_{\frac{n+r}{2n}}\right)=0$, i.e.
\begin{eq}
\sup_{\mathcal M}u-u\leqslant T_{\frac{n+r}{2n}}
\end{eq}
on $\mathcal M$, where
\begin{eq}
T_{\frac{n+r}{2n}}\triangleq\frac{C_9\mathrm{e}\left(\frac{2}{\log2}\right)^{\frac{n+r}{2n}}}{\frac{r-n}{2n}} \max\left\{\frac{1}{\log^{\frac{r-n}{2n}}\left(1+\frac{1}{h(0)}\right)}\comma\frac{\log^{\frac{n+r}{2n}} 2}{\log^{\frac{r}{n}}\left(1+\frac{1}{h(0)}\right)}\right\}.
\end{eq}
Thus, we have $\osc\limits_{\mathcal M} u\leqslant T_{\frac{n+r}{2n}}$. Finally, by \thmref{thm: an inequality of Holder-Young type} it's easy to see that
\begin{eq}
h(0)\leqslant\int_{\mathcal M} \frac{\varphi^n}{V_{\bm\chi}}\bm\omega_{\bm g}^n\leqslant\frac{2\left(1+\frac n2+\frac r4\right)^{n+r}\Ent_{n\comma r}\left(\frac{\varphi^n}{V_{\bm\chi}}\right)}{\log^n\left(1+\frac{1}{V_{\bm g}}\right)\log^r\biggl(1+\log\left(1+\frac{1}{V_{\bm g}}\right)\biggr)}\comma
\end{eq}
and therefore we have proved \eqref{eq: L^infty estimate}.
\end{proof}


\section[L∞ estimates on Hermitian manifolds]{$\mathrm{L}^\infty$ estimates on Hermitian manifolds} \label{sec: L^infty estimates on Hermitian manifolds}

In this section, we prove \thmref{thm: L^infty estimates on Hermitian manifolds}. Let $(\mathcal{M}\comma\bm{\omega}_{\bm g})$ be a compact Hermitian manifold of dimension $n$ (note that $\bm\omega_{\bm g}$ is not necessarily closed), $\bm\chi$ be a positive $(1\comma 1)$-form on $\mathcal M$ satisfying
\begin{eq} \label{eq: chi}
\varepsilon\bm{\omega}_{\bm g}\leqslant\bm\chi\leqslant M\bm{\omega}_{\bm g}
\end{eq}
for positive constants $\varepsilon$ and $M$, and $\varphi$ be a positive smooth function on $\mathcal M$. Assume that $u\in\mathrm{C}^2(\mathcal M)$ satisfies
\begin{eq}
\left\{ \begin{gathered}
f\Bigl(\bm\lambda\bigl(\bm{\omega}_{\bm g}^{-1}(\bm{\chi}+\mathrm{i}\partial\overline\partial u)\bigr)\Bigr)\leqslant\varphi\comma \\
\bm\lambda\bigl(\bm{\omega}_{\bm g}^{-1}(\bm{\chi}+\mathrm{i}\partial\overline\partial u)\bigr)\in\Gamma\comma
\end{gathered} \right.
\end{eq}
where $\Gamma$ is a symmetric open convex cone in $\mathbb{R}^n$ satisfying $\Gamma_n\subset\Gamma\subset\Gamma_1$ and $f$ is a positive $\mathrm{C}^1$ symmetric function defined in $\Gamma$ satisfying the conditions \eqref{item: elliptic}, \eqref{item: structural} and \eqref{eq: Euler inequality} on page \pageref{item: elliptic}. Let $V_{\bm g}$ and $\Ent_{q\comma r}(\varphi^n)$ denote $\int_{\mathcal M} \bm\omega_{\bm g}^n$ and
\begin{eq}
\left\|\varphi^n\right\|_{\mathrm{L}^1(\log\mathrm{L})^q(\log\log\mathrm{L})^r(\mathcal M\comma\bm{\omega}_{\bm g}^n)}
\end{eq}
respectively (see \definitionref{definition: L^p(logL)^q(loglogL)^r}). The following lemma in \cite{Wang2020} will be used as an analogue of \eqref{eq: alpha-invariant} in the Hermitian case.

\begin{lem} \label{lem: exponential integral estimate}
Let $\Omega$ be a bounded, smooth, pseudo-convex domain in $\mathbb{C}^n$ and $v\in\mathrm{C}^\infty\left(\overline\Omega\right)$ be a plurisubharmonic function such that $\left.v\right|_{\partial\Omega}=0$ and
\begin{eq}
\int_\Omega \det\left(\mathrm{D}^2_{\mathbb{C}}v\right)\dif V_{\mathbb{C}^n}\leqslant 1\comma
\end{eq}
where $\mathrm{D}^2_{\mathbb{C}}v$ denotes the complex Hessian matrix of $v$ and $\dif V_{\mathbb{C}^n}$ denotes the standard volume form in $\mathbb{C}^n$. Then we have
\begin{eq} \label{eq: exponential integral estimate}
\int_\Omega \mathrm{e}^{-\alpha v}\dif V_{\mathbb{C}^n}\leqslant C
\end{eq}
for positive constants $\alpha$ and $C$ depending only on $n$ and the diameter of $\Omega$.
\end{lem}

\begin{proof}
See \cite[p.~13]{Wang2020}.
\end{proof}

Now we prove the sharp $\mathrm{L}^\infty$ estimate in \thmref{thm: L^infty estimates on Hermitian manifolds}.

\begin{thm}
Fix $r>n$. Then we have
\begin{eq} \label{eq: L^infty estimate, Hermitian}
\osc\limits_{\mathcal M} u\leqslant C\comma
\end{eq}
where $C$ is a positive constant depending only on $\mathcal{M}$, $\bm{\omega}_{\bm g}$, $n$, $r$, $\frac{1}{\delta}$, $\Lambda$, $\frac{1}{\varepsilon}$, $M$ and $\Ent_{n\comma r}(\varphi^n)$.
\end{thm}

\begin{proof}
The first half of the proof follows the strategy of \cite{Guo2023b}. Since $(\mathcal M\comma\bm\omega_{\bm g})$ is a compact Hermitian manifold, it's well-known that there exists a positive constant $d_0<\frac 12$ depending only on $\mathcal M$ and $\bm\omega_{\bm g}$ such that for any point $\bm{z}\in\mathcal M$, one can choose a local holomorphic coordinate system $(U\comma\bm{\rho}\mbox{; }z^j)$ which contains $\bm{z}$ and satisfies
\begin{eq}
{\mathcal B}_{2d_0}(\bm z)\triangleq\bm{\rho}^{-1}\Bigl(\mathrm{B}_{2d_0}\bigl(\bm{\rho}(\bm z)\bigr)\Bigr)\subset U
\end{eq}
and
\begin{eq} \label{eq: minimal and maximal eigenvalue of g_jk}
\frac12\mathbf{I}_n\leqslant\bigl(g_{j\overline k}(\bm{w})\bigr)\leqslant2\mathbf{I}_n\comma\forall\bm{w}\in{\mathcal B}_{2d_0}(\bm z)\comma
\end{eq}
where $\mathrm{B}_{2d_0}\bigl(\bm{\rho}(\bm z)\bigr)$ denotes the Euclidean ball of radius $2d_0$ centered at $\bm{\rho}(\bm z)$. Henceforth we assume that $\min\limits_{\mathcal M}u=u(\bm{z}_0)$ for some point $\bm{z}_0\in\mathcal{M}$ and choose as above the local holomorphic coordinate system $(U\comma\bm{\rho}\mbox{; }z^j)$ containing $\bm{z}_0$. For $l\in\mathbb{N}$, let $\tau_l$ be a smooth non-decreasing function defined on $\mathbb{R}$ satisfying \eqref{eq: tau_l}. Fix $s\in\left[0\comma\frac{\varepsilon}{2}d_0^2\right)$ and define
\begin{eq}
\tilde{u}_s(\bm z)\triangleq -u(\bm z)+u({\bm z}_0)+\frac{\varepsilon}{2}d_0^2-\frac{\varepsilon}{2}\bigl|\bm{\rho}(\bm z)-\bm{\rho}({\bm z}_0)\bigr|^2-s\comma \bm z\in\overline{\mathcal{B}_{d_0}({\bm z}_0)}.
\end{eq}
Note that $\tilde{u}_s(\bm z_0)>0$ and
\begin{eq}
\tilde{u}_s(\bm z)\leqslant0\comma\forall\bm z\in\partial\mathcal{B}_{d_0}({\bm z}_0).
\end{eq}
Let $\Omega_s$ denote the set
\begin{eq}
\Bigl\{\bm{z}\in\mathcal{B}_{d_0}({\bm z}_0)\Bigm|\tilde{u}_s(\bm{z})>0\Bigr\}.
\end{eq}
By the celebrated theorem of \citeauthor{Caffarelli1985} \cite{Caffarelli1985}, we can let $v_{s\comma l}$ be the unique smooth solution to the Dirichlet problem of the complex Monge-Amp\`ere equation
\begin{eq}
\left\{ \begin{gathered}
\det\left(\mathrm{D}_{\mathbb{C}}^2 v_{s\comma l}\right)=\frac{\tau_l\left(\tilde{u}_s\circ\bm{\rho}^{-1}\right)}{A_{s\comma l}}\left(\varphi\circ\bm{\rho}^{-1}\right)^n\det\left(g_{j\overline k}\circ\bm{\rho}^{-1}\right) \mbox{ in } \mathrm{B}_{d_0}\bigl(\bm{\rho}({\bm z}_0)\bigr)\comma \\
\mathrm{D}_{\mathbb{C}}^2 v_{s\comma l}>\bm{0} \mbox{ in } \mathrm{B}_{d_0}\bigl(\bm{\rho}({\bm z}_0)\bigr)\comma \\
v_{s\comma l}=0 \mbox{ on } \partial\mathrm{B}_{d_0}\bigl(\bm{\rho}({\bm z}_0)\bigr)\comma
\end{gathered} \right.
\end{eq}
where the normalized constant $A_{s\comma l}$ is defined as
\begin{eq}
A_{s\comma l}\triangleq\int_{\mathrm{B}_{d_0} \bigl(\bm{\rho}({\bm z}_0)\bigr)} \tau_l\left(\tilde{u}_s\circ\bm{\rho}^{-1}\right)\left(\varphi\circ\bm{\rho}^{-1}\right)^n\det\left(g_{j\overline k}\circ\bm{\rho}^{-1}\right)\dif V_{\mathbb{C}^n}=\int_{\mathcal{B}_{d_0}(\bm{z}_0)} \tau_l\left(\tilde{u}_s\right)\varphi^n\bm{\omega}_{\bm g}^n
\end{eq}
so that
\begin{eq}
\int_{\mathrm{B}_{d_0}\bigl(\bm{\rho}({\bm z}_0)\bigr)} \det\left(\mathrm{D}_{\mathbb{C}}^2 v_{s\comma l}\right)\dif V_{\mathbb{C}^n}=1.
\end{eq}
Note that
\begin{eq}
v_{s\comma l}<0 \mbox{ in } \mathrm{B}_{d_0}\bigl(\bm{\rho}(\bm{z}_0)\bigr)
\end{eq}
and
\begin{eq}
\lim_{l\to\infty}A_{s\comma l}=A_s\triangleq\int_{\Omega_s} \tilde{u}_s\varphi^n\bm{\omega}_{\bm g}^n>0. 
\end{eq}
Now we define the comparison function
\begin{eq}
w\triangleq-\nu(-v_{s\comma l}\circ\bm{\rho})^b+\tilde{u}_s\comma
\end{eq}
where $\nu$, $b$ are positive constants to be chosen later so that $w\leqslant 0$ in $\overline{\mathcal{B}_{d_0}({\bm z}_0)}$. Assume that $\max\limits_{\overline{\mathcal{B}_{d_0}({\bm z}_0)}} w=w(\bm{z}_1)$ for some point $\bm{z}_1\in\overline{\mathcal{B}_{d_0}({\bm z}_0)}$. If $\bm{z}_1\in\overline{\mathcal{B}_{d_0}({\bm z}_0)}\setminus\Omega_s$, it's easy to see that $w(\bm{z}_1)\leqslant 0$. Henceforth we assume that $\bm{z}_1\in\Omega_s$. Straightforward calculations show that
\begin{eq} \begin{aligned}
\mathrm{D}_{j\overline k}w&=\nu b(1-b)(-v_{s\comma l}\circ\bm{\rho})^{b-2}\mathrm{D}_j\left(v_{s\comma l}\circ\bm{\rho}\right)\mathrm{D}_{\overline k}\left(v_{s\comma l}\circ\bm{\rho}\right) \\
&\phantom{=}+\nu b(-v_{s\comma l}\circ\bm{\rho})^{b-1}\mathrm{D}_{j\overline k}\left(v_{s\comma l}\circ\bm{\rho}\right)-\mathrm{D}_{j\overline k}u-\frac{\varepsilon}{2}\delta_{jk}.
\end{aligned} \end{eq}
Note that the matrix $\left(\mathrm{D}_{j\overline k}w(\bm{z}_1)\right)$ is negative semi-definite since $\bm{z}_1$ is the maximum point of $w$. Recalling the definition \eqref{eq: F^jk} of $F_{\bm\chi+\mathrm{i}\partial\overline\partial u}^{j\overline k}$ and the property (2) in \propref{prop: properties of F^jk}, we assume that $b<1$ and calculate at $\bm{z}_1$ as follows:
\begin{eq} \begin{aligned}
0\geqslant F_{\bm\chi+\mathrm{i}\partial\overline\partial u}^{j\overline k}\mathrm{D}_{j\overline k}w&\geqslant\nu b(-v_{s\comma l}\circ\bm{\rho})^{b-1}F_{\bm\chi+\mathrm{i}\partial\overline\partial u}^{j\overline k}\mathrm{D}_{j\overline k}(v_{s\comma l}\circ\bm{\rho})-F_{\bm\chi+\mathrm{i}\partial\overline\partial u}^{j\overline k}\left(\mathrm{D}_{j\overline k}u+\frac{\varepsilon}{2}\delta_{jk}\right) \\
&\geqslant\nu b(-v_{s\comma l}\circ\bm{\rho})^{b-1}F_{\bm\chi+\mathrm{i}\partial\overline\partial u}^{j\overline k}\mathrm{D}_{j\overline k}(v_{s\comma l}\circ\bm{\rho})-F_{\bm\chi+\mathrm{i}\partial\overline\partial u}^{j\overline k}\left(\chi_{j\overline k}+\mathrm{D}_{j\overline k}u\right)\comma
\end{aligned} \end{eq}
where the last ``$\geqslant$'' is due to \eqref{eq: chi} and \eqref{eq: minimal and maximal eigenvalue of g_jk}. Further calculations show that at $\bm{z}_1$,
\begin{eq} \begin{aligned}
0&\geqslant\nu b(-v_{s\comma l}\circ\bm{\rho})^{b-1}F_{\bm\chi+\mathrm{i}\partial\overline\partial u}^{j\overline k}\mathrm{D}_{j\overline k}(v_{s\comma l}\circ\bm{\rho})-F_{\bm\chi+\mathrm{i}\partial\overline\partial u}^{j\overline k}\left(\chi_{j\overline k}+\mathrm{D}_{j\overline k}u\right) \\
&\geqslant\nu b(-v_{s\comma l}\circ\bm{\rho})^{b-1}n\left(\frac{\delta\tilde{u}_s}{A_{s\comma l}}\right)^{\frac 1n}\frac{\varphi}{f\Bigl(\bm\lambda\bigl(\bm{\omega}_{\bm g}^{-1}(\bm{\chi}+\mathrm{i}\partial\overline\partial u)\bigr)\Bigr)}-\Lambda \\
&\geqslant\nu b(-v_{s\comma l}\circ\bm{\rho})^{b-1}n\delta^{\frac 1n}A_{s\comma l}^{-\frac 1n}\left(\tilde{u}_s\right)^{\frac 1n}-\Lambda\comma
\end{aligned} \end{eq}
where the second ``$\geqslant$'' is due to the properties (3), (4) in \propref{prop: properties of F^jk} and the fact that at~$\bm{z}_1$,
\begin{eq}
\tau_l\left(\tilde{u}_s\right)=\tilde{u}_s+\frac 1l.
\end{eq}
It follows that at $\bm{z}_1$,
\begin{eq}
\tilde{u}_s\leqslant\left(\frac{\Lambda}{bn\delta^{\frac 1n}}\right)^n\nu^{-n}A_{s\comma l}(-v_{s\comma l}\circ\bm{\rho})^{n(1-b)}
\end{eq}
and
\begin{eq}
w\leqslant\nu(-v_{s\comma l}\circ\bm{\rho})^b\Biggl(-1+\left(\frac{\Lambda}{bn\delta^{\frac 1n}}\right)^n\nu^{-n-1}A_{s\comma l}(-v_{s\comma l}\circ\bm{\rho})^{n(1-b)-b}\Biggr).
\end{eq}
Now we choose the constants $b$ and $\nu$ as follows:
\begin{ga} \label{eq: b and nu, Hermitian}
\text{$n(1-b)-b=0$, i.e. $b=\frac{n}{n+1}\in(0\comma1)$}\comma \\
\text{$-1+\left(\frac{\Lambda}{bn\delta^{\frac 1n}}\right)^n\nu^{-n-1}A_{s\comma l}=0$, i.e. $\nu=\left(\frac{(n+1)\Lambda}{n^2\delta^{\frac 1n}}\right)^{\frac{n}{n+1}}A_{s\comma l}^{\frac{1}{n+1}}$}.
\end{ga}
Then we obtain $\max\limits_{\overline{\mathcal{B}_{d_0}({\bm z}_0)}}w=w(\bm{z}_1)\leqslant 0$, i.e.
\begin{eq} \label{eq: comparison inequality, Hermitian}
\tilde{u}_s\leqslant\nu(-v_{s\comma l}\circ\bm{\rho})^b=C_1A_{s\comma l}^{\frac{1}{n+1}}\left(-v_{s\comma l}\circ\bm{\rho}\right)^{\frac{n}{n+1}}
\end{eq}
in $\overline{\mathcal{B}_{d_0}({\bm z}_0)}$. Thus, on $\Omega_s$ we obtain the following key estimate
\begin{eq} \label{eq: key estimate, Hermitian} 
C_1^{-\frac{n+1}{n}}A_{s\comma l}^{-\frac 1n}\left(\tilde{u}_s\right)^{\frac{n+1}{n}}\leqslant -v_{s\comma l}\circ\bm{\rho}.
\end{eq}
By \lemref{lem: exponential integral estimate} and \eqref{eq: minimal and maximal eigenvalue of g_jk}, we have
\begin{eq} \begin{aligned}
\int_{\Omega_s} \mathrm{e}^{\alpha C_1^{-\frac{n+1}{n}}A_{s\comma l}^{-\frac 1n}\left(\tilde{u}_s\right)^{\frac{n+1}{n}}}\bm\omega_{\bm g}^n&\leqslant\int_{\mathcal{B}_{d_0}(\bm{z}_0)} \mathrm{e}^{-\alpha v_{s\comma l}\circ\bm{\rho}}\bm\omega_{\bm g}^n \\
&=\int_{\mathrm{B}_{d_0}\bigl(\bm{\rho}(\bm{z}_0)\bigr)} \mathrm{e}^{-\alpha v_{s\comma l}}\det\left(g_{j\overline k}\circ\bm{\rho}^{-1}\right)\dif V_{\mathbb{C}^n}\leqslant C_2
\end{aligned} \end{eq}
for positive constants $\alpha$ and $C_2$ depending only on $n$ and $d_0$. Note that the above inequality holds for any $l\in\mathbb{N}$. Letting $l\to\infty$, we obtain
\begin{eq} \label{eq: key inequality in L^infty estimates, Hermitian}
\int_{\Omega_s} \mathrm{e}^{\alpha C_1^{-\frac{n+1}{n}}A_s^{-\frac 1n}\left(\tilde{u}_s\right)^{\frac{n+1}{n}}}\bm\omega_{\bm g}^n\leqslant C_2.
\end{eq}
Define
\begin{eq}
\eta(t)\triangleq t^n\log^r\left(1+t^{\frac{n}{n+1}}\right) (t\geqslant0).
\end{eq}
Let $\eta^{-1}$ denote the inverse function of $\eta$. By the classical Young inequality, for any $x\comma y\geqslant0$ there holds
\begin{eq} \begin{aligned}
\eta(x)y&\leqslant\int_0^{\eta(x)} \left(\mathrm{e}^{\eta^{-1}(t)}-1\right)\dif t+\int_0^y \eta\bigl(\log(1+t)\bigr)\dif t \\
&\leqslant\eta(x)(\mathrm{e}^x-1)+y\eta\bigl(\log(1+y)\bigr) \\
&\leqslant C_3\mathrm{e}^{2x}+y\log^n(1+y)\log^r\left(1+\log^{\frac{n}{n+1}}(1+y)\right).
\end{aligned} \end{eq}
Letting 
\begin{eq}
x=\left(\frac{\tilde{u}_s(\bm{z})}{\left(\frac{2}{\alpha}\right)^{\frac{n}{n+1}}C_1A_s^{\frac{1}{n+1}}}\right)^{\frac{n+1}{n}}\comma\quad y=\frac{\bigl(\varphi(\bm{z})\bigr)^n}{\Ent_{n\comma r}\left(\varphi^n\right)}\comma
\end{eq}
and integrating both sides of the acquired inequality over $\Omega_s$, we obtain
\begin{eq} \begin{aligned}
&\mathrel{\phantom{=}}\int_{\Omega_s} \left(\frac{\tilde{u}_s}{\left(\frac{2}{\alpha}\right)^{\frac{n}{n+1}}C_1A_s^{\frac{1}{n+1}}}\right)^{n+1}\log^r\left(1+\frac{\tilde{u}_s}{\left(\frac{2}{\alpha}\right)^{\frac{n}{n+1}}C_1A_s^{\frac{1}{n+1}}}\right) \varphi^n\bm\omega_{\bm g}^n \\
&\leqslant \biggl(C_3C_2+\int_{\Omega_s\cap\{y>\mathrm{e}-1\}} y\log^n(1+y)\log^r\bigl(1+\log(1+y)\bigr)\bm\omega_{\bm g}^n \\
&\phantom{\leqslant \biggl(}+(\mathrm{e}-1)\log^r2\cdot V_{\bm g}\biggr)\Ent_{n\comma r}\left(\varphi^n\right) \\
&\leqslant C_4\Ent_{n\comma r}\left(\varphi^n\right)
\end{aligned} \end{eq}
by \eqref{eq: key inequality in L^infty estimates, Hermitian} and \eqref{eq: to be used in the proof of the inequality of Holder-Young type}. This indicates that
\begin{eq}
\left\|\tilde{u}_s\right\|_{\mathrm{L}^{n+1}(\log\mathrm{L})^r\left(\Omega_s\comma\varphi^n\bm\omega_{\bm g}^n\right)}\leqslant \left(\frac{2}{\alpha}\right)^{\frac{n}{n+1}}C_1A_s^{\frac{1}{n+1}}\max\left\{1\comma \bigl(C_4\Ent_{n\comma r}(\varphi^n)\bigr)^{\frac{1}{n+1}}\right\}
\end{eq}
by \propref{prop: estimates of norms}. Let $h(s)$ denote
\begin{eq}
\int_{\Omega_s} \varphi^n\bm\omega_{\bm g}^n.
\end{eq}
Then by \thmref{thm: an inequality of Holder-Young type} we have
\begin{eq} \begin{aligned}
A_s&=\int_{\Omega_s} \tilde{u}_s\varphi^n\bm\omega_{\bm g}^n \\
&\leqslant\frac{2\left(n+1+\frac{r}{2}\right)^{\frac{r}{n+1}}\left\|\tilde{u}_s\right\|_{\mathrm{L}^{n+1}(\log\mathrm{L})^r\left(\Omega_s\comma\varphi^n\bm\omega_{\bm g}^n\right)}\bigl(h(s)\bigr)^{\frac{n}{n+1}}}{\log^{\frac{r}{n+1}}\left(1+\frac{1}{h(s)}\right)} \\
&\leqslant \frac{2\left(n+1+\frac{r}{2}\right)^{\frac{r}{n+1}}\left(\frac{2}{\alpha}\right)^{\frac{n}{n+1}}C_1A_s^{\frac{1}{n+1}}\max\left\{1\comma \bigl(C_4\Ent_{n\comma r}(\varphi^n)\bigr)^{\frac{1}{n+1}}\right\}\bigl(h(s)\bigr)^{\frac{n}{n+1}}}{\log^{\frac{r}{n+1}}\left(1+\frac{1}{h(s)}\right)}\comma
\end{aligned} \end{eq}
and therefore
\begin{eq}
A_s\leqslant \frac{C_5h(s)}{\log^{\frac{r}{n}}\left(1+\frac{1}{h(s)}\right)}
\end{eq}
for any $s\in\left[0\comma\frac{\varepsilon}{2}d_0^2\right)$, where $C_5$ is a positive constant depending on $\Ent_{n\comma r}\left(\varphi^n\right)$ and other trivial quantities. Note that the above inequality holds for any $s\geqslant0$ in fact. On the other hand, for any $t>s$ there holds
\begin{eq}
A_s\geqslant\int_{\Omega_t} \left(-u(\bm z)+u({\bm z}_0)+\frac{\varepsilon}{2}d_0^2-\frac{\varepsilon}{2}\bigl|\bm{\rho}(\bm z)-\bm{\rho}({\bm z}_0)\bigr|^2-s\right)\varphi^n\bm\omega_{\bm g}^n\geqslant(t-s)h(t).
\end{eq}
As a result, we have
\begin{eq} \label{eq: iteration, Hermitian}
h(t)\leqslant\frac{C_5}{t-s}\frac{h(s)}{\log^{\frac{r}{n}}\left(1+\frac{1}{h(s)}\right)}
\end{eq}
for any $t\comma s\in\mathbb{R}\colon t>s\geqslant0$. Since $r>n$ and
\begin{eq}
h(s)>0\comma\forall s\in\left[0\comma\frac{\varepsilon}{2}d_0^2\right)\comma
\end{eq}
\eqref{eq: iteration, Hermitian} indicates by \corref{cor: another version of the iteration lemma of De Giorgi type} that
\begin{eq} \label{eq: h(0), Hermitian}
h(0)\geqslant\frac{1}{\mathrm{e}^{L_{\frac{n+r}{2n}}}-1}\comma
\end{eq}
where
\begin{eq}
L_{\frac{n+r}{2n}}\triangleq\max\left\{\left(\frac{C_5\mathrm{e}\left(\frac{2}{\log2}\right)^{\frac{n+r}{2n}}}{\frac{r-n}{2n}\cdot\frac{\varepsilon}{2}d_0^2}\right)^{\frac{2n}{r-n}}\comma \left(\frac{C_5\mathrm{e}2^{\frac{n+r}{2n}}}{\frac{r-n}{2n}\cdot\frac{\varepsilon}{2}d_0^2}\right)^{\frac{n}{r}}\right\}.
\end{eq}
On the other hand, since
\begin{eq}
\Omega_0=\left\{\bm{z}\in\mathcal{B}_{d_0}({\bm z}_0)\middle|-u(\bm z)+u(\bm{z}_0)+\frac{\varepsilon}{2}d_0^2-\frac{\varepsilon}{2}\bigl|\bm{\rho}(\bm z)-\bm{\rho}({\bm z}_0)\bigr|^2>0\right\}\comma
\end{eq}
there holds
\begin{eq} \label{eq: estimate of osc u, Hermitian} \begin{aligned}
&\mathrel{\phantom{=}}\frac{1}{\Ent_{n\comma 0}(\varphi^n)}h(0)\log^n\left(1+\osc_{\mathcal M}u\right) \\
&=\int_{\Omega_0} \log^n\left(1+\sup_{\mathcal M}u-u(\bm{z}_0)\right)\frac{\varphi^n}{\Ent_{n\comma 0}(\varphi^n)}\bm\omega_{\bm g}^n \\
&\leqslant\int_{\Omega_0} \log^n\left(1+\sup_{\mathcal M}u-u+\frac{\varepsilon}{2}d_0^2\right)\frac{\varphi^n}{\Ent_{n\comma 0}(\varphi^n)}\bm\omega_{\bm g}^n.
\end{aligned} \end{eq}
By the classical Young inequality, for any $x\comma y\geqslant 0$ we have
\begin{eq} \begin{aligned}
xy&\leqslant\int_0^x \left(\mathrm{e}^{\frac12 t^{\frac 1n}}-1\right)\dif t+\int_0^y 2^n\log^n(1+t)\dif t \\
&\leqslant C_6\mathrm{e}^{x^{\frac 1n}}+2^ny\log^n(1+y).
\end{aligned} \end{eq}
Letting 
\begin{eq}
x=\log^n\left(1+\sup_{\mathcal M}u-u(\bm{z})+\frac{\varepsilon}{2}d_0^2\right)\comma\quad y=\frac{\bigl(\varphi(\bm{z})\bigr)^n}{\Ent_{n\comma 0}\left(\varphi^n\right)}\comma
\end{eq}
and integrating both sides of the acquired inequality over $\Omega_0$, we see that
\begin{eq} \label{eq: estimate of osc u, Hermitian, continued} \begin{aligned}
&\mathrel{\phantom{=}}\int_{\Omega_0} \log^n\left(1+\sup_{\mathcal M}u-u+\frac{\varepsilon}{2}d_0^2\right)\frac{\varphi^n}{\Ent_{n\comma 0}(\varphi^n)}\bm\omega_{\bm g}^n \\
&\leqslant C_6\int_{\Omega_0} \left(1+\sup_{\mathcal M}u-u+\frac{\varepsilon}{2}d_0^2\right)\bm\omega_{\bm g}^n+2^n \\
&\leqslant C_6\biggl(C_7+\left(1+\frac{\varepsilon}{2}d_0^2\right)V_{\bm g}\biggr)+2^n
\end{aligned} \end{eq}
by \eqref{eq: to be used in the proof of the inequality of Holder-Young type} and \propref{prop: L^1 estimate}. Finally, combining \eqref{eq: estimate of osc u, Hermitian}, \eqref{eq: estimate of osc u, Hermitian, continued}, \eqref{eq: h(0), Hermitian} and \eqref{eq: Ent_n0 controlled by Ent_nr}, we obtain \eqref{eq: L^infty estimate, Hermitian}.
\end{proof}


\section{Gluing lemma of smooth radial plurisubharmonic functions and an example} \label{sec: Gluing lemma of smooth radial plurisubharmonic functions and an example}

In this section, we prove a gluing lemma of smooth, radial, strictly plurisubharmonic functions and use this lemma to construct an explicit example to show that the $\mathrm{L}^\infty$ estimates in \thmref{thm: L^infty estimates on Kahler manifolds} and \thmref{thm: L^infty estimates on Hermitian manifolds} may fail when $r\leqslant n-1$. The following lemma is well-known to experts.

\begin{lem} \label{lem: regularization of |t|}
For any $\varepsilon>0$, there exists a $\rho_{\varepsilon}\in\mathrm{C}^\infty(\mathbb{R})$ which satisfies:
\begin{enumerate}[(1)]
\item For any $t\in\mathbb{R}\colon |t|\geqslant\varepsilon$, there holds $\rho_{\varepsilon}(t)=|t|$;
\item For any $t\in\mathbb{R}$, there hold $\rho_{\varepsilon}(t)\geqslant|t|$, $\rho_{\varepsilon}(-t)=\rho_{\varepsilon}(t)$, $\bigl|\rho_{\varepsilon}'(t)\bigr|\leqslant 1$ and $0\leqslant\rho_{\varepsilon}''(t)\leqslant\frac{M}{\varepsilon}$, where $M$ is a positive constant independent of $\varepsilon$, $\rho_{\varepsilon}$, $t$.
\end{enumerate}
\end{lem}

\begin{proof}
See e.g. \cite[pp.~207--208]{Guan2002}. The constant $M$ can be taken as $3$.
\end{proof}

Next we generalize \lemref{lem: regularization of |t|} to glue two smooth, strictly convex functions defined on disjointed intervals.

\begin{prop} \label{prop: glue two smooth, strictly convex functions defined on disjointed intervals}
Let $a_1$, $b_1$, $a_2$, $b_2$ be real constants such that $a_1<b_1<a_2<b_2$, $k\geqslant 2$ be a positive integer or $\infty$, $f\in\mathrm{C}^k\bigl([a_1\comma b_1]\bigr)$ and $g\in\mathrm{C}^k\bigl([a_2\comma b_2]\bigr)$. Assume that
\begin{eq} \label{eq: f and g strictly convex}
\alpha_1\triangleq\min_{[a_1\comma b_1]}f''>0\comma\quad\alpha_2\triangleq\min_{[a_2\comma b_2]}g''>0.
\end{eq}
Then there exists a function $h\in\mathrm{C}^k(\mathbb{R})$ satisfying
\begin{eq}
\left.h\vphantom{h_{[a_1\comma b_1]}}\right|_{[a_1\comma b_1]}=f\comma\left.h\vphantom{h_{[a_2\comma b_2]}}\right|_{[a_2\comma b_2]}=g\comma \mbox{and } \inf_{\mathbb{R}}h''>0
\end{eq}
if and only if
\begin{eq} \label{eq: compatibility condition}
f_-'(b_1)<\frac{g(a_2)-f(b_1)}{a_2-b_1}<g_+'(a_2).
\end{eq}
Moreover, if \eqref{eq: compatibility condition} holds, the function $h$ can be chosen to satisfy
\begin{eq} \label{eq: infimum of h''}
\inf_{\mathbb{R}}h''\geqslant\min\left\{\frac{\alpha_1}{2}\comma\frac{\alpha_2}{2}\comma\frac{\frac{g(a_2)-f(b_1)}{a_2-b_1}-f_-'(b_1)}{a_2-b_1}\comma\frac{g_+'(a_2)-\frac{g(a_2)-f(b_1)}{a_2-b_1}}{a_2-b_1}\right\}
\end{eq}
and
\begin{eq} \label{eq: supremum of h''} \begin{aligned}
\sup_{\mathbb{R}}h''&\leqslant\frac{16M\bigl(g_+'(a_2)-f_-'(b_1)\bigr)^2}{(a_2-b_1)\min\left\{\frac{g(a_2)-f(b_1)}{a_2-b_1}-f_-'(b_1)\comma g_+'(a_2)-\frac{g(a_2)-f(b_1)}{a_2-b_1}\right\}} \\
&\phantom{=}+1+\max\left\{\sup_{[a_1\comma b_1]}f''\comma\sup_{[a_2\comma b_2]}g''\right\}\comma
\end{aligned} \end{eq}
where $M$ is the positive constant in \lemref{lem: regularization of |t|}.
\end{prop}

\begin{proof}
Without loss of generality, we assume that $k=\infty$ and $f\comma g\in\mathrm{C}^\infty(\mathbb{R})$. The necessity is obvious and we will focus on the proof of the sufficiency, i.e. we assume that \eqref{eq: compatibility condition} holds and manage to prove the existence of $h$. Fix a positive constant
\begin{eq} \label{eq: c}
c\triangleq\min\left\{\frac{\alpha_1}{2}\comma\frac{\alpha_2}{2}\comma\frac{\frac{g(a_2)-f(b_1)}{a_2-b_1}-f'(b_1)}{a_2-b_1}\comma\frac{g'(a_2)-\frac{g(a_2)-f(b_1)}{a_2-b_1}}{a_2-b_1}\right\}.
\end{eq}
Then it's easy to see that there exists a positive constant $\delta\in\left(0\comma\frac{a_2-b_1}{2}\right)$ satisfying:
\begin{eq} \label{eq: delta}
\left\{ \begin{lgathered}
\inf_{[a_1-\delta\comma b_1+\delta]} f''\geqslant c\comma\quad\inf_{[a_2-\delta\comma b_2+\delta]} g''\geqslant c\comma \\
\sup_{[a_1-\delta\comma b_1+\delta]} f''\leqslant \sup_{[a_1\comma b_1]} f''+1\comma\quad\sup_{[a_2-\delta\comma b_2+\delta]} g''\leqslant \sup_{[a_2\comma b_2]} g''+1\comma \\
\frac{g'(a_2-\delta)-f'(b_1+\delta)}{a_2-b_1-2\delta}>c\comma \\
\frac{2\left(\frac{g(a_2)-f(b_1+\delta)}{a_2-(b_1+\delta)}-f'(b_1+\delta)\right)}{a_2-(b_1+\delta)}\geqslant\frac c2+\frac{\frac{g(a_2)-f(b_1)}{a_2-b_1}-f'(b_1)}{a_2-b_1}>c\comma \\
\frac{2\left(g'(a_2-\delta)-\frac{g(a_2-\delta)-f(b_1)}{(a_2-\delta)-b_1}\right)}{(a_2-\delta)-b_1}\geqslant\frac c2+\frac{g'(a_2)-\frac{g(a_2)-f(b_1)}{a_2-b_1}}{a_2-b_1}>c.
\end{lgathered} \right. \end{eq}
Choose cut-off functions $\xi\comma\eta\in\mathrm{C}_{\mathrm{c}}^\infty(\mathbb{R})$ such that
\begin{eq} \label{eq: xi, eta}
\left\{ \begin{lgathered}
\xi\bigr|_{[a_1\comma b_1]}\equiv1\comma0\leqslant\xi\leqslant1\comma\supp\xi\subset(a_1-\delta\comma b_1+\delta)\text{;} \\
\eta\bigr|_{[a_2\comma b_2]}\equiv1\comma0\leqslant\eta\leqslant1\comma\supp\eta\subset(a_2-\delta\comma b_2+\delta).
\end{lgathered} \right.
\end{eq}
Define
\begin{al}
\tilde{f}(t)&\triangleq\int_{b_1}^t \int_{b_1}^y \xi(x)\bigl(f''(x)-c\bigr)\dif x\dif y+f(b_1)+f'(b_1)(t-b_1)+\frac c2 (t-b_1)^2\comma \\
\tilde{g}(t)&\triangleq\int_{a_2}^t \int_{a_2}^y \eta(x)\bigl(g''(x)-c\bigr)\dif x\dif y+g(a_2)+g'(a_2)(t-a_2)+\frac c2 (t-a_2)^2.
\end{al}
By \eqref{eq: xi, eta} we have
\begin{eq} \label{eq: tilde f(a_2)<tilde g(a_2)} \begin{aligned}
\tilde{f}(a_2)&=\int_{b_1}^{b_1+\delta} \int_{b_1}^y \xi(x)\bigl(f''(x)-c\bigr)\dif x\dif y+\int_{b_1+\delta}^{a_2} \int_{b_1}^{b_1+\delta} \xi(x)\bigl(f''(x)-c\bigr)\dif x\dif y \\
&\phantom{=}+f(b_1)+f'(b_1)(a_2-b_1)+\frac c2 (a_2-b_1)^2 \\
&\leqslant\int_{b_1}^{b_1+\delta} \int_{b_1}^y \bigl(f''(x)-c\bigr)\dif x\dif y+\int_{b_1+\delta}^{a_2} \int_{b_1}^{b_1+\delta} \bigl(f''(x)-c\bigr)\dif x\dif y \\
&\phantom{=}+f(b_1)+f'(b_1)(a_2-b_1)+\frac c2 (a_2-b_1)^2 \\
&=f(b_1+\delta)+f'(b_1+\delta)\bigl(a_2-(b_1+\delta)\bigr)+\frac c2 \bigl(a_2-(b_1+\delta)\bigr)^2 \\
&<g(a_2)=\tilde{g}(a_2)\comma
\end{aligned} \end{eq}
where the ``$<$'' is due to the fourth line of \eqref{eq: delta}. Similarly, by \eqref{eq: xi, eta} and the fifth line of \eqref{eq: delta}, we have
\begin{eq} \label{eq: tilde g(b_1)<tilde f(b_1)} \begin{aligned}
\tilde{g}(b_1)&\leqslant\int_{a_2-\delta}^{a_2} \int_y^{a_2} \bigl(g''(x)-c\bigr)\dif x\dif y+\int_{b_1}^{a_2-\delta} \int_{a_2-\delta}^{a_2} \bigl(g''(x)-c\bigr)\dif x\dif y \\
&\phantom{=}+g(a_2)+g'(a_2)(b_1-a_2)+\frac c2 (b_1-a_2)^2 \\
&=g(a_2-\delta)-g'(a_2-\delta)(a_2-\delta-b_1)+\frac c2 (a_2-\delta-b_1)^2<f(b_1)=\tilde{f}(b_1).
\end{aligned} \end{eq}
Moreover, it's easy to see that $\tilde{f}\comma\tilde{g}\in\mathrm{C}^\infty(\mathbb{R})$, 
\begin{eq} \label{eq: tilde f=f, tilde g=g}
\tilde{f}\Bigr|_{[a_1\comma b_1]}=f\bigr|_{[a_1\comma b_1]}\comma\quad\tilde{g}\Bigr|_{[a_2\comma b_2]}=g\bigr|_{[a_2\comma b_2]}\comma
\end{eq}
and
\begin{eq} \label{eq: tilde f'', tilde g''}
\tilde{f}''(t)=\xi(t)\bigl(f''(t)-c\bigr)+c\geqslant c\comma\quad\tilde{g}''(t)=\eta(t)\bigl(g''(t)-c\bigr)+c\geqslant c
\end{eq}
for any $t\in\mathbb{R}$ by the first line of \eqref{eq: delta}. As a result,
\begin{eq} \label{eq: varepsilon}
\varepsilon\triangleq\min\left\{\frac{\tilde{f}(b_1)-\tilde{g}(b_1)}{2}\comma\frac{\tilde{g}(a_2)-\tilde{f}(a_2)}{2}\right\}
\end{eq}
is a positive constant and we can let $\rho_{\varepsilon}$ be the smooth function in \lemref{lem: regularization of |t|}. We claim that the function
\begin{eq} \label{eq: h(t)}
h(t)\triangleq\begin{dcases}
\tilde{f}(t)\comma & t\leqslant b_1 \\
\frac{\tilde{f}(t)+\tilde{g}(t)+\rho_{\varepsilon}\bigl(\tilde{f}(t)-\tilde{g}(t)\bigr)}{2}\comma & b_1<t<a_2 \\
\tilde{g}(t)\comma & t\geqslant a_2
\end{dcases} \end{eq}
satisfies all our requirements. In fact, when $t$ is sufficiently close to $b_1+0$, we have $\tilde{f}(t)-\tilde{g}(t)>\varepsilon$ and therefore $h(t)=\tilde{f}(t)$; when $t$ is sufficiently close to $a_2-0$, we have $\tilde{f}(t)-\tilde{g}(t)<-\varepsilon$ and therefore $h(t)=\tilde{g}(t)$. Thus, we have $h\in\mathrm{C}^\infty(\mathbb{R})$. By \eqref{eq: tilde f=f, tilde g=g} and \eqref{eq: tilde f'', tilde g''} we see that
\begin{eq}
h\bigr|_{[a_1\comma b_1]}=\tilde{f}\Bigr|_{[a_1\comma b_1]}=f\bigr|_{[a_1\comma b_1]}\comma\quad h\bigr|_{[a_2\comma b_2]}=\tilde{g}\Bigr|_{[a_2\comma b_2]}=g\bigr|_{[a_2\comma b_2]}\comma
\end{eq}
and
\begin{eq}
h''(t)=\tilde{f}''(t)\geqslant c\comma\forall t\leqslant b_1\text{;} \quad h''(t)=\tilde{g}''(t)\geqslant c\comma\forall t\geqslant a_2.
\end{eq}
Moreover, for any $t\in(b_1\comma a_2)$ there holds
\begin{eq} \begin{aligned}
h''(t)&=\frac12 \Bigl(\tilde{f}''(t)+\tilde{g}''(t)+\rho_{\varepsilon}''\bigl(\tilde{f}(t)-\tilde{g}(t)\bigr)\bigl(\tilde{f}'(t)-\tilde{g}'(t)\bigr)^2 \\
&\hphantom{=\frac12 \Bigl(}+\rho_{\varepsilon}'\bigl(\tilde{f}(t)-\tilde{g}(t)\bigr)\bigl(\tilde{f}''(t)-\tilde{g}''(t)\bigr)\Bigr) \\
&\geqslant \frac{1+\rho_{\varepsilon}'\bigl(\tilde{f}(t)-\tilde{g}(t)\bigr)}{2}\tilde{f}''(t)+\frac{1-\rho_{\varepsilon}'\bigl(\tilde{f}(t)-\tilde{g}(t)\bigr)}{2}\tilde{g}''(t)\geqslant c.
\end{aligned} \end{eq}
As a result, we have $\inf\limits_{\mathbb{R}} h''\geqslant c>0$. Recalling the definition \eqref{eq: c} of the positive constant $c$, we have justified \eqref{eq: infimum of h''}, while \eqref{eq: supremum of h''} remains to be proved. For any $t\in(b_1\comma a_2)$, by \eqref{eq: xi, eta} and the third line of \eqref{eq: delta} there hold
\begin{eq} \begin{aligned}
\tilde{g}'(t)-\tilde{f}'(t)&=\int_{a_2}^t \eta(x)\bigl(g''(x)-c\bigr)\dif x+g'(a_2)+c(t-a_2) \\
&\phantom{=}-\int_{b_1}^t \xi(x)\bigl(f''(x)-c\bigr)\dif x-f'(b_1)-c(t-b_1) \\
&\geqslant -\int_{a_2-\delta}^{a_2} \bigl(g''(x)-c\bigr)\dif x+g'(a_2)-c(a_2-b_1) \\
&\phantom{=}-\int_{b_1}^{b_1+\delta} \bigl(f''(x)-c\bigr)\dif x-f'(b_1) \\
&=g'(a_2-\delta)-f'(b_1+\delta)-c(a_2-b_1-2\delta)>0 \\
\end{aligned} \end{eq}
and
\begin{eq}
\tilde{g}'(t)-\tilde{f}'(t)\leqslant g'(a_2)+c(t-a_2)-f'(b_1)-c(t-b_1)\leqslant g'(a_2)-f'(b_1).
\end{eq}
The two inequalities above imply
\begin{eq} \label{eq: (tilde g'-tilde f')^2}
\left(\tilde{g}'(t)-\tilde{f}'(t)\right)^2\leqslant\bigl(g'(a_2)-f'(b_1)\bigr)^2
\end{eq}
for any $t\in(b_1\comma a_2)$. On the other hand, by \eqref{eq: varepsilon}, \eqref{eq: tilde f(a_2)<tilde g(a_2)}, \eqref{eq: tilde g(b_1)<tilde f(b_1)}, \eqref{eq: delta} and \eqref{eq: c} we have
\begin{eq} \label{eq: varepsilon, estimate} \begin{aligned}
\varepsilon&\geqslant\frac12\min\Bigl\{f(b_1)-g(a_2-\delta)+g'(a_2-\delta)(a_2-\delta-b_1)-\frac c2 (a_2-\delta-b_1)^2 \comma \\
&\mathrel{\hphantom{\geqslant}}\hphantom{\frac12\min\Bigl\{}g(a_2)-f(b_1+\delta)-f'(b_1+\delta)(a_2-b_1-\delta)-\frac c2 (a_2-b_1-\delta)^2\Bigr\} \\
&\geqslant\frac{(a_2-b_1-\delta)^2}{4}\min\left\{\frac{g'(a_2)-\frac{g(a_2)-f(b_1)}{a_2-b_1}}{a_2-b_1}-\frac c2\comma\frac{\frac{g(a_2)-f(b_1)}{a_2-b_1}-f'(b_1)}{a_2-b_1}-\frac c2\right\} \\
&\geqslant\frac{(a_2-b_1-\delta)^2}{8(a_2-b_1)}\min\left\{g'(a_2)-\frac{g(a_2)-f(b_1)}{a_2-b_1}\comma\frac{g(a_2)-f(b_1)}{a_2-b_1}-f'(b_1)\right\} \\
&\geqslant\frac{a_2-b_1}{32}\min\left\{g'(a_2)-\frac{g(a_2)-f(b_1)}{a_2-b_1}\comma\frac{g(a_2)-f(b_1)}{a_2-b_1}-f'(b_1)\right\}.
\end{aligned} \end{eq}
Combining \eqref{eq: (tilde g'-tilde f')^2}, \eqref{eq: varepsilon, estimate} and \lemref{lem: regularization of |t|}, we obtain
\begin{eq} \begin{aligned}
&\mathrel{\phantom{=}}\rho_{\varepsilon}''\bigl(\tilde{f}(t)-\tilde{g}(t)\bigr)\bigl(\tilde{f}'(t)-\tilde{g}'(t)\bigr)^2 \\
&\leqslant\frac{32M\bigl(g'(a_2)-f'(b_1)\bigr)^2}{(a_2-b_1)\min\left\{\frac{g(a_2)-f(b_1)}{a_2-b_1}-f'(b_1)\comma g'(a_2)-\frac{g(a_2)-f(b_1)}{a_2-b_1}\right\}}
\end{aligned} \end{eq}
for any $t\in(b_1\comma a_2)$. Furthermore, by \eqref{eq: c} and the second line of \eqref{eq: delta} there hold
\begin{al}
\tilde{f}''(t)&=\xi(t)\bigl(f''(t)-c\bigr)+c\leqslant\max\left\{c\comma\sup_{[a_1-\delta\comma b_1+\delta]} f''\right\}\leqslant\sup_{[a_1\comma b_1]} f''+1\comma \\
\tilde{g}''(t)&=\eta(t)\bigl(g''(t)-c\bigr)+c\leqslant\max\left\{c\comma\sup_{[a_2-\delta\comma b_2+\delta]} g''\right\}\leqslant\sup_{[a_2\comma b_2]} g''+1
\end{al}
for any $t\in\mathbb{R}$. As a result, we have
\begin{eq}
h''(t)=\tilde{f}''(t)\leqslant\sup_{[a_1\comma b_1]} f''+1\comma\forall t\leqslant b_1\text{;}\quad h''(t)=\tilde{g}''(t)\leqslant\sup_{[a_2\comma b_2]} g''+1\comma\forall t\geqslant a_2\text{;}
\end{eq}
and
\begin{eq} \begin{aligned}
h''(t)&=\frac{1+\rho_{\varepsilon}'\bigl(\tilde{f}(t)-\tilde{g}(t)\bigr)}{2}\tilde{f}''(t)+\frac{1-\rho_{\varepsilon}'\bigl(\tilde{f}(t)-\tilde{g}(t)\bigr)}{2}\tilde{g}''(t) \\
&\hphantom{=}+\frac12\rho_{\varepsilon}''\bigl(\tilde{f}(t)-\tilde{g}(t)\bigr)\bigl(\tilde{f}'(t)-\tilde{g}'(t)\bigr)^2 \\
&\leqslant\frac{16M\bigl(g'(a_2)-f'(b_1)\bigr)^2}{(a_2-b_1)\min\left\{\frac{g(a_2)-f(b_1)}{a_2-b_1}-f'(b_1)\comma g'(a_2)-\frac{g(a_2)-f(b_1)}{a_2-b_1}\right\}} \\
&\phantom{=}+1+\max\left\{\sup_{[a_1\comma b_1]}f''\comma\sup_{[a_2\comma b_2]}g''\right\}\comma\forall t\in(b_1\comma a_2).
\end{aligned} \end{eq}
We have justified \eqref{eq: supremum of h''} and therefore our claim that the function $h$ satisfies all our requirements.
\end{proof}

The following proposition is an analogue of \propref{prop: glue two smooth, strictly convex functions defined on disjointed intervals} for the case of non-strictly convex functions. Its proof is similar to that of \propref{prop: glue two smooth, strictly convex functions defined on disjointed intervals} and therefore put into \appref{app: Proof of the gluing lemma of smooth convex functions defined on disjointed intervals}.

\begin{prop} \label{prop: glue two smooth convex functions defined on disjointed intervals}
Let $a_1$, $b_1$, $a_2$, $b_2$ be real constants such that $a_1<b_1<a_2<b_2$, $k\geqslant 2$ be a positive integer or $\infty$, $f\in\mathrm{C}^k\bigl([a_1\comma b_1]\bigr)$ and $g\in\mathrm{C}^k\bigl([a_2\comma b_2]\bigr)$ be convex functions. Assume that both $f$ and $g$ can extend to be $\mathrm{C}^k$ convex functions defined on $\mathbb{R}$. If
\begin{eq} \label{eq: compatibility condition, convex case}
f_-'(b_1)<\frac{g(a_2)-f(b_1)}{a_2-b_1}<g_+'(a_2)\comma
\end{eq}
then there exists a convex function $h\in\mathrm{C}^k(\mathbb{R})$ satisfying
\begin{eq}
\left.h\vphantom{h_{[a_1\comma b_1]}}\right|_{[a_1\comma b_1]}=f\comma\quad \left.h\vphantom{h_{[a_2\comma b_2]}}\right|_{[a_2\comma b_2]}=g\comma
\end{eq}
and
\begin{eq} \label{eq: supremum of h'', convex case} \begin{aligned}
\sup_{\mathbb{R}}h''&\leqslant\frac{4M\bigl(g_+'(a_2)-f_-'(b_1)\bigr)^2}{(a_2-b_1)\min\left\{\frac{g(a_2)-f(b_1)}{a_2-b_1}-f_-'(b_1)\comma g_+'(a_2)-\frac{g(a_2)-f(b_1)}{a_2-b_1}\right\}} \\
&\phantom{=}+1+\max\left\{\sup_{[a_1\comma b_1]}f''\comma\sup_{[a_2\comma b_2]}g''\right\}\comma
\end{aligned} \end{eq}
where $M$ is the positive constant in \lemref{lem: regularization of |t|}.
\end{prop}

\begin{rmk}
See e.g. \cite[p.~6]{Azagra2019} for when the assumption ``both $f$ and $g$ can extend to be $\mathrm{C}^k$ convex functions defined on $\mathbb{R}$'' in \propref{prop: glue two smooth convex functions defined on disjointed intervals} holds.
\end{rmk}

We call \propref{prop: glue two smooth, strictly convex functions defined on disjointed intervals} the gluing lemma of smooth, strictly convex functions defined on disjointed intervals, and \propref{prop: glue two smooth convex functions defined on disjointed intervals} the gluing lemma of smooth convex functions defined on disjointed intervals. Now we prove a gluing lemma of smooth, radial, strictly plurisubharmonic functions.

\begin{thm} \label{thm: gluing lemma of smooth, radial, strictly plurisubharmonic functions}
Let $a_1$, $b_1$, $a_2$, $b_2$ be positive constants such that $a_1<b_1<a_2<b_2$, $n\in\mathbb{N}$, $k\geqslant 2$ be a positive integer or $\infty$, $f\in\mathrm{C}^k\bigl([a_1\comma b_1]\bigr)$ and $g\in\mathrm{C}^k\bigl([a_2\comma b_2]\bigr)$. Assume that $f\bigl(|\bm z|^2\bigr)$ and $g\bigl(|\bm z|^2\bigr)$ are strictly plurisubharmonic functions defined on $\bigl\{\bm z\in\mathbb{C}^n\bigm|a_1\leqslant\left|\bm z\right|^2\leqslant b_1\bigr\}$ and $\bigl\{\bm z\in\mathbb{C}^n\bigm|a_2\leqslant\left|\bm z\right|^2\leqslant b_2\bigr\}$ respectively. Then there exists a function $h\in\mathrm{C}^k\bigl([a_1\comma{+}\infty)\bigr)$ such that
\begin{eq} \label{eq: h=f, h=g}
\left.h\vphantom{h_{[a_1\comma b_1]}}\right|_{[a_1\comma b_1]}=f\comma\left.h\vphantom{h_{[a_2\comma b_2]}}\right|_{[a_2\comma b_2]}=g\comma
\end{eq}
and $h\bigl(|\bm z|^2\bigr)$ is a strictly plurisubharmonic function on $\bigl\{\bm z\in\mathbb{C}^n\bigm|\left|\bm z\right|^2\geqslant a_1\bigr\}$ if and only if
\begin{eq} \label{eq: compatibility condition, strictly plurisubharmonic case}
b_1f_-'(b_1)<\frac{g(a_2)-f(b_1)}{\log a_2-\log b_1}<a_2g_+'(a_2).
\end{eq}
Moreover, if \eqref{eq: compatibility condition, strictly plurisubharmonic case} holds, the function $h$ can be chosen so that for any $\bm z\in\mathbb{C}^n\colon b_1<|\bm z|^2<a_2$,
\begin{eq} \label{eq: supremum of h'', strictly plurisubharmonic case} \begin{aligned}
&\mathrel{\phantom{=}}\det\biggl(\mathrm{D}_{\mathbb{C}}^2\Bigl(h\bigl(|\bm z|^2\bigr)\Bigr)\biggr) \\
&\leqslant \frac{a_2^{n-1}\bigl(g_+'(a_2)\bigr)^{n-1}}{b_1^{2n}} \\
&\phantom{=}\cdot\Biggl(\frac{16M\bigl(a_2g_+'(a_2)-b_1f_-'(b_1)\bigr)^2}{(\log a_2-\log b_1)\min\left\{\frac{g(a_2)-f(b_1)}{\log a_2-\log b_1}-b_1f_-'(b_1)\comma a_2g_+'(a_2)-\frac{g(a_2)-f(b_1)}{\log a_2-\log b_1}\right\}} \\
&\phantom{=\cdot\Biggl(}+1+\max\left\{\sup_{[\log a_1\comma\log b_1]}F''\comma\sup_{[\log a_2\comma\log b_2]}G''\right\}\Biggr)\comma
\end{aligned} \end{eq}
where $M$ is the positive constant in \lemref{lem: regularization of |t|} and the functions $F$ and $G$ are defined in \eqref{eq: F(t)} and \eqref{eq: G(t)} below.
\end{thm}

\begin{proof}
For simplicity we only prove the case when $n\geqslant 2$. We may as well assume that $f\comma g\in\mathrm{C}^k(\mathbb{R})$. Straightforward calculations show that
\begin{eq}
\mathrm{D}_{j\overline k}\Bigl(f\bigl(|\bm z|^2\bigr)\Bigr)=f'\bigl(|\bm z|^2\bigr)\delta_{jk}+f''\bigl(|\bm z|^2\bigr)\overline{z}_jz_k
\end{eq}
and all eigenvalues of $\mathrm{D}_{\mathbb{C}}^2\Bigl(f\bigl(|\bm z|^2\bigr)\Bigr)$ are
\begin{eq} \label{eq: eigenvalues of D_C^2(f(|z|^2))}
f'\bigl(|\bm z|^2\bigr)\mbox{ (with multiplicity $n-1$) and }f'\bigl(|\bm z|^2\bigr)+f''\bigl(|\bm z|^2\bigr)|\bm z|^2.
\end{eq}
Define
\begin{eq} \label{eq: F(t)}
F(t)\triangleq f(\mathrm{e}^t)\comma t\in[\log a_1\comma\log b_1]
\end{eq}
and
\begin{eq} \label{eq: G(t)}
G(t)\triangleq g(\mathrm{e}^t)\comma t\in[\log a_2\comma\log b_2].
\end{eq}
It's easy to see that ``$f\bigl(|\bm z|^2\bigr)$ is a strictly plurisubharmonic function on $\bigl\{\bm z\in\mathbb{C}^n\bigm|a_1\leqslant\left|\bm z\right|^2\leqslant b_1\bigr\}$'' is equivalent to ``$F'(t)>0$ and $F''(t)>0$ for any $t\in[\log a_1\comma\log b_1]$'', and ``$g\bigl(|\bm z|^2\bigr)$ is a strictly plurisubharmonic function on $\bigl\{\bm z\in\mathbb{C}^n\bigm|a_2\leqslant\left|\bm z\right|^2\leqslant b_2\bigr\}$'' is equivalent to ``$G'(t)>0$ and $G''(t)>0$ for any $t\in[\log a_2\comma\log b_2]$''. Moreover, \eqref{eq: compatibility condition, strictly plurisubharmonic case} is equivalent to
\begin{eq} \label{eq: compatibility condition, strictly plurisubharmonic case, equivalent form}
F'(\log b_1)<\frac{G(\log a_2)-F(\log b_1)}{\log a_2-\log b_1}<G'(\log a_2).
\end{eq}
When \eqref{eq: compatibility condition, strictly plurisubharmonic case} holds, by \propref{prop: glue two smooth, strictly convex functions defined on disjointed intervals} there exists a function $H\in\mathrm{C}^k(\mathbb{R})$ satisfying
\begin{eq}
H\bigr|_{[\log a_1\comma\log b_1]}=F\comma H\bigr|_{[\log a_2\comma\log b_2]}=G\comma \inf_{\mathbb{R}}H''>0\comma
\end{eq}
and
\begin{eq} \begin{aligned}
\sup_{\mathbb{R}}H''&\leqslant\frac{16M\bigl(a_2g'(a_2)-b_1f'(b_1)\bigr)^2}{(\log a_2-\log b_1)\min\left\{\frac{g(a_2)-f(b_1)}{\log a_2-\log b_1}-b_1f'(b_1)\comma a_2g'(a_2)-\frac{g(a_2)-f(b_1)}{\log a_2-\log b_1}\right\}} \\
&\phantom{=}+1+\max\left\{\sup_{[\log a_1\comma\log b_1]}F''\comma\sup_{[\log a_2\comma\log b_2]}G''\right\}\comma
\end{aligned} \end{eq}
where $M$ is the positive constant in \lemref{lem: regularization of |t|}. Since $F'(\log a_1)>0$, we have
\begin{eq}
\inf_{[\log a_1\comma{+}\infty)}H'=H'(\log a_1)>0.
\end{eq}
As a result, the function
\begin{eq}
h(t)\triangleq H(\log t)\comma t\in[a_1\comma{+}\infty)
\end{eq}
satisfies all our requirements including \eqref{eq: supremum of h'', strictly plurisubharmonic case} since
\begin{eq} \begin{aligned}
\det\biggl(\mathrm{D}_{\mathbb{C}}^2\Bigl(h\bigl(|\bm z|^2\bigr)\Bigr)\biggr)&=\Bigl(h'\bigl(|\bm z|^2\bigr)\Bigr)^{n-1}\Bigl(h'\bigl(|\bm z|^2\bigr)+h''\bigl(|\bm z|^2\bigr)|\bm z|^2\Bigr) \\
&=\frac{1}{|\bm z|^{2n}}\Bigl(H'\bigl(\log|\bm z|^2\bigr)\Bigr)^{n-1}H''\bigl(\log|\bm z|^2\bigr).
\end{aligned} \end{eq}
On the other hand, when there exists a function $h\in\mathrm{C}^k\bigl([a_1\comma{+}\infty)\bigr)$ satisfying \eqref{eq: h=f, h=g} and that $h\bigl(|\bm z|^2\bigr)$ is a strictly plurisubharmonic function on $\bigl\{\bm z\in\mathbb{C}^n\bigm|\left|\bm z\right|^2\geqslant a_1\bigr\}$, we see that the function
\begin{eq}
H(t)\triangleq h(\mathrm{e}^t)\comma t\in[\log a_1\comma{+}\infty)
\end{eq}
satisfies
\begin{eq}
H\bigr|_{[\log a_1\comma\log b_1]}=F\comma H\bigr|_{[\log a_2\comma\log b_2]}=G\comma \mbox{and } \inf_{[\log a_1\comma\log b_2]}H''>0.
\end{eq}
Then it's easy to prove \eqref{eq: compatibility condition, strictly plurisubharmonic case, equivalent form}, or equivalently \eqref{eq: compatibility condition, strictly plurisubharmonic case}.
\end{proof}

\begin{rmk}
One can prove an analogue of \thmref{thm: gluing lemma of smooth, radial, strictly plurisubharmonic functions} for the case of non-strictly plurisubharmonic functions. Its proof is similar to that of \thmref{thm: gluing lemma of smooth, radial, strictly plurisubharmonic functions}, with ``\propref{prop: glue two smooth, strictly convex functions defined on disjointed intervals}'' replaced by ``\propref{prop: glue two smooth convex functions defined on disjointed intervals}''.
\end{rmk}

At the end of this section, we use \thmref{thm: gluing lemma of smooth, radial, strictly plurisubharmonic functions} to construct an explicit example to show that the $\mathrm{L}^\infty$ estimates in \thmref{thm: L^infty estimates on Kahler manifolds} and \thmref{thm: L^infty estimates on Hermitian manifolds} may fail when $r\leqslant n-1$. The idea of our construction originates from \cite{Guo2022a} and \cite{Guedj2024a}. We note that the right-hand function $\varphi_\varepsilon$ of the equation \eqref{eq: varphi_varepsilon} and the solution $u_\varepsilon$ to \eqref{eq: varphi_varepsilon} in our example are smooth without singularities.

\begin{ex} \label{ex: L^infty estimates fail when r leqslant n-1}
Let $\bm\omega_{\bm g}$ denote the K\"ahler form of the Fubini-Study metric on $\mathbb{C}\mathrm{P}^n$ ($n\geqslant2$). Fix a point
\begin{eq}
\bm z_0=\bigl[(1\comma0\comma\cdots\comma0)\bigr]\in\mathbb{C}\mathrm{P}^n
\end{eq}
and choose the standard local holomorphic coordinate system $(U\comma\bm\rho\mbox{; }z^j)$ containing $\bm z_0$. Note that $\bm\rho(\bm z_0)=\bm0$, $\bm\rho(U)=\mathbb{C}^n$ and
\begin{eq}
(\bm\rho^{-1})^{*}\bigl(\bm\omega_{\bm g}\bigr|_{U}\bigr)=\mathrm{i}\partial\overline\partial\log\bigl(1+|\bm z|^2\bigr)=\frac{\mathrm{i}}{1+|\bm z|^2}\left(\delta_{jk}-\frac{\overline{z}_jz_k}{1+|\bm z|^2}\right)\dif z^j\wedge\dif\overline{z}^k.
\end{eq}
For any $\varepsilon\in\left(0\comma\frac{1}{16}\right)$, we want to construct a smooth function $u_{\varepsilon}$ on $\mathbb{C}\mathrm{P}^n$ such that $u_{\varepsilon}\in\mathrm{C}_{\mathrm c}^\infty(U)$, $\bm\omega_{\bm g}+\mathrm{i}\partial\overline\partial u_{\varepsilon}>\bm0$, $\lim\limits_{\varepsilon\to0+0} u_{\varepsilon}(\bm z_0)={-}\infty$ and
\begin{eq}
\left\|\det\left(\bm{\omega}_{\bm g}^{-1}(\bm\omega_{\bm g}+\mathrm{i}\partial\overline\partial u_{\varepsilon})\right)\right\|_{\mathrm{L}^1(\log\mathrm{L})^n(\log\log\mathrm{L})^{n-1}\left(\mathbb{C}\mathrm{P}^n\comma\bm\omega_{\bm g}^n\right)}\leqslant C
\end{eq}
for a positive constant $C$ independent of $\varepsilon$. For this purpose, it suffices to construct a smooth, strictly plurisubharmonic function $v_{\varepsilon}$ on $\bm\rho(U)=\mathbb{C}^n$ such that $\lim\limits_{\varepsilon\to0+0} v_{\varepsilon}(\bm 0)={-}\infty$,
\begin{eq}
v_{\varepsilon}(\bm z)=\log\bigl(1+|\bm z|^2\bigr)\comma\forall\bm z\in\mathbb{C}^n\colon |\bm z|\geqslant 1\comma
\end{eq}
and
\begin{eq}
\left\|\det\left(\mathrm{D}_{\mathbb{C}}^2 v_{\varepsilon}\right)\right\|_{\mathrm{L}^1(\log\mathrm{L})^n(\log\log\mathrm{L})^{n-1}\left(\mathrm{B}_1(\bm 0)\comma\dif V_{\mathbb{C}^n}\right)}\leqslant C
\end{eq}
for a positive constant $C$ independent of $\varepsilon$. If we could construct such a function $v_{\varepsilon}$, then the function
\begin{eq} \label{eq: u_varepsilon}
u_{\varepsilon}(\bm z)\triangleq\begin{dcases}
v_{\varepsilon}\bigl(\bm\rho(\bm z)\bigr)-\log\bigl(1+|\bm\rho(\bm z)|^2\bigr)\comma & \bm z\in U \\
0\comma & \bm z\in\mathbb{C}\mathrm{P}^n\setminus U
\end{dcases} \end{eq}
would satisfy all our requirements. Now define
\begin{eq} \label{eq: f_varepsilon}
f_{\varepsilon}(t)\triangleq-\frac{\log2}{2}\log\Biggl(1+\log\biggl(1+\log\biggl(1+\log\left(1+\frac{1}{t+\varepsilon}\right)\biggr)\biggr)\Biggr)\comma t\in\left[0\comma\frac{1}{16}\right]\comma
\end{eq}
By \corref{cor: properties of f_varepsilon} we see that $f_{\varepsilon}\bigl(|\bm z|^2\bigr)$ is a smooth, strictly plurisubharmonic function on $\overline{\mathrm{B}_{\frac 14}(\bm 0)}\subset\mathbb{C}^n$ and
\begin{eq}
\left\|\det\biggl(\mathrm{D}_{\mathbb{C}}^2\Bigl(f_{\varepsilon}\left(|\bm z|^2\right)\Bigr)\biggr)\right\|_{\mathrm{L}^1(\log\mathrm{L})^n(\log\log\mathrm{L})^{n-1}\left(\mathrm{B}_{\frac14}(\bm 0)\comma\dif V_{\mathbb{C}^n}\right)}\leqslant C_1.
\end{eq}
On the other hand, since
\begin{eq} \begin{aligned}
\frac{1}{16}f_{\varepsilon}'\left(\frac{1}{16}\right)&=\frac{\frac{\log2}{2}\biggl(1+\log\Bigl(1+\log\Bigl(1+\log\left(1+\frac{1}{\frac{1}{16}+\varepsilon}\right)\Bigr)\Bigr)\biggr)^{-1}\left(\frac{1}{\frac{1}{16}+\varepsilon}\right)^2}{16\biggl(1+\log\Bigl(1+\log\left(1+\frac{1}{\frac{1}{16}+\varepsilon}\right)\Bigr)\biggr)\biggl(1+\log\left(1+\frac{1}{\frac{1}{16}+\varepsilon}\right)\biggr)\left(1+\frac{1}{\frac{1}{16}+\varepsilon}\right)} \\
&\leqslant\frac{\log2\cdot\Bigl(1+\log\bigl(1+\log(1+\log9)\bigr)\Bigr)^{-1}}{2\bigl(1+\log(1+\log9)\bigr)(1+\log9)} \\
&\leqslant\frac{1}{12}<\frac14\leqslant\frac{\log2-f_{\varepsilon}\left(\frac{1}{16}\right)}{\log1-\log\frac{1}{16}}
\end{aligned} \end{eq}
and
\begin{eq} \begin{aligned}
\frac{\log2-f_{\varepsilon}\left(\frac{1}{16}\right)}{\log1-\log\frac{1}{16}}&=\frac14+\frac{1}{8}\log\Biggl(1+\log\biggl(1+\log\biggl(1+\log\left(1+\frac{1}{\frac{1}{16}+\varepsilon}\right)\biggr)\biggr)\Biggr) \\
&\leqslant\frac14+\frac{1}{8}\log\Bigl(1+\log\bigl(1+\log(1+\log17)\bigr)\Bigr) \\
&\leqslant\frac14+\frac{1}{8}\log(1+\log3)<\frac38<\frac12\comma
\end{aligned} \end{eq}
by \thmref{thm: gluing lemma of smooth, radial, strictly plurisubharmonic functions} there exists a function $h_{\varepsilon}\in\mathrm{C}^\infty\Bigl(\bigl[\frac{1}{64}\comma{+}\infty\bigr)\Bigr)$ such that $h_{\varepsilon}\bigl(|\bm z|^2\bigr)$ is a smooth, strictly plurisubharmonic function on $\bigl\{\bm z\in\mathbb{C}^n\bigm|\left|\bm z\right|\geqslant \frac18\bigr\}$,
\begin{eq}
h_{\varepsilon}\bigl(|\bm z|^2\bigr)=\begin{dcases}
f_{\varepsilon}\bigl(|\bm z|^2\bigr)\comma & \frac18\leqslant|\bm z|\leqslant\frac14 \\
\log\bigl(1+|\bm z|^2\bigr)\comma & 1\leqslant|\bm z|\leqslant2
\end{dcases} \end{eq}
and for any $\bm z\in\mathbb{C}^n\colon \frac14<|\bm z|<1$, there holds
\begin{eq} \label{eq: supremum of h'', strictly plurisubharmonic case'} \begin{aligned}
&\mathrel{\phantom{=}}\det\biggl(\mathrm{D}_{\mathbb{C}}^2\Bigl(h_{\varepsilon}\bigl(|\bm z|^2\bigr)\Bigr)\biggr) \\
&\leqslant \frac{\left(\frac12\right)^{n-1}}{\left(\frac{1}{16}\right)^{2n}}\Biggl(\frac{16M\left(\frac12-\frac{1}{16}f_{\varepsilon}'\left(\frac{1}{16}\right)\right)^2}{4\log2\cdot\min\left\{\frac{\log2-f_{\varepsilon}\left(\frac{1}{16}\right)}{4\log2}-\frac{1}{16}f_{\varepsilon}'\left(\frac{1}{16}\right)\comma \frac12-\frac{\log2-f_{\varepsilon}\left(\frac{1}{16}\right)}{4\log2}\right\}} \\
&\hphantom{=\frac{\left(\frac12\right)^{n-1}}{\left(\frac{1}{16}\right)^{2n}}\Biggl(}+1+\max\Biggl\{\sup_{t\in\left[\log \frac{1}{64}\comma\log \frac{1}{16}\right]}\bigl(f_{\varepsilon}(\mathrm{e}^t)\bigr)''\comma\sup_{t\in[\log 1\comma\log 4]}\frac{\mathrm{e}^t}{(1+\mathrm{e}^t)^2}\Biggr\}\Biggr)\leqslant C_2.
\end{aligned} \end{eq}
Then it's straightforward to see that the function
\begin{eq}
v_{\varepsilon}(\bm z)\triangleq\begin{dcases}
f_{\varepsilon}\bigl(|\bm z|^2\bigr)\comma & 0\leqslant|\bm z|\leqslant\frac14 \\
h_{\varepsilon}\bigl(|\bm z|^2\bigr)\comma & \frac14<|\bm z|<1 \\
\log\bigl(1+|\bm z|^2\bigr)\comma & |\bm z|\geqslant1
\end{dcases} \end{eq}
is exactly what we want. Recall that for any $\varepsilon\in\left(0\comma\frac{1}{16}\right)$, the function $u_\varepsilon$ defined in \eqref{eq: u_varepsilon} is a smooth admissible solution to the complex Monge-Amp\`ere equation
\begin{eq} \label{eq: varphi_varepsilon}
\Bigl(\det\left(\bm{\omega}_{\bm g}^{-1}(\bm\omega_{\bm g}+\mathrm{i}\partial\overline\partial u_{\varepsilon})\right)\Bigr)^{\frac 1n}=\varphi_\varepsilon
\end{eq}
on $\left(\mathbb{C}\mathrm{P}^n\comma\bm\omega_{\bm g}\right)$, where $\varphi_\varepsilon$ is a smooth function on $\mathbb{C}\mathrm{P}^n$ satisfying
\begin{eq}
\left\|\varphi_\varepsilon^n\right\|_{\mathrm{L}^1(\log\mathrm{L})^n(\log\log\mathrm{L})^{n-1}\left(\mathbb{C}\mathrm{P}^n\comma\bm\omega_{\bm g}^n\right)}\leqslant C
\end{eq}
for a positive constant $C$ independent of $\varepsilon$. However,
\begin{eq}
\lim\limits_{\varepsilon\to0+0} \osc_{\mathbb{C}\mathrm{P}^n} u_{\varepsilon}\geqslant-\lim\limits_{\varepsilon\to0+0} u_{\varepsilon}(\bm z_0)={+}\infty.
\end{eq}
\end{ex}


\appendix

\section[Properties of the function Φ(t; p, q, r) in Section 2]{Properties of the function $\Phi_{p\comma q\comma r}$ in Section~2} \label{app: Properties of the function Phi_pqr in Section 2}

Recall that we have defined the function $\Phi_{p\comma q\comma r}$ in \eqref{eq: Phi_pqr} as
\begin{eq}
\Phi_{p\comma q\comma r}(t)\triangleq t^p\log^q(1+t)\log^r\bigl(1+\log(1+t)\bigr).
\end{eq}
In this appendix, we prove two properties of this function which have been used in \secref{sec: Orlicz spaces and an inequality of Holder-Young type}.

\begin{prop} \label{prop: strong convexity}
Fix $p\geqslant 1$ and $q\comma r\geqslant 0$ such that $q^2+r^2>0$. Then the function
\begin{eq}
\Phi_{p\comma q\comma r}\left(t^{\frac 1p}\right)=t\log^q\left(1+t^{\frac 1p}\right)\log^r\biggl(1+\log\left(1+t^{\frac 1p}\right)\biggr)
\end{eq}
is strictly convex on $(0\comma{+}\infty)$.
\end{prop}

\begin{proof}
By straightforward calculations we obtain
\begin{eq}
\frac{\dif{}^2}{\dif t^2} \Phi_{p\comma q\comma r}\left(t^{\frac 1p}\right)=\frac{\log ^{q-2}\left(1+t^{\frac{1}{p}}\right)\log^{r-2}\left(1+\log\left(1+t^{\frac{1}{p}}\right)\right)}{p^2t^{1-\frac{1}{p}}\left(1+t^{\frac{1}{p}}\right)^2\left(1+\log\left(1+t^{\frac{1}{p}}\right)\right)^2}f\left(t^{\frac 1p}\right)\comma
\end{eq}
where
\begin{eq} \label{eq: f(t)} \begin{aligned}
f(t)&\triangleq(pt+p+1)\log^3(1+t)\log\bigl(1+\log(1+t)\bigr)\Bigl(q\log\bigl(1+\log(1+t)\bigr)+r\Bigr) \\
&+\log^2(1+t)\Bigl(q\bigl((q-1)t+2p(1+t)+2\bigr)\log^2\bigl(1+\log(1+t)\bigr) \\
&\hphantom{+\log^2(1+t)\Bigl(} +r\bigl((2q-1)t+p(1+t)+1\bigr)\log\bigl(1+\log(1+t)\bigr)+r(r-1)t\Bigr) \\
&+q\log(1+t)\log\bigl(1+\log(1+t)\bigr) \\
&\phantom{+}\cdot\Bigl(\bigl(2(q-1)t+p(1+t)+1\bigr)\log\bigl(1+\log(1+t)\bigr)+2rt\Bigr) \\
&+q(q-1)t\log^2\bigl(1+\log(1+t)\bigr).
\end{aligned} \end{eq}
We only need to prove that $f(t)>0$ for any $t>0$. Since $p\geqslant 1$, we have
\begin{eq} \label{eq: f(t), estimate} \begin{aligned}
f(t)&\geqslant(t+2)\log^3(1+t)\log\bigl(1+\log(1+t)\bigr)\Bigl(q\log\bigl(1+\log(1+t)\bigr)+r\Bigr) \\
&+\log^2(1+t)\Bigl(q\bigl((q+1)t+4\bigr)\log^2\bigl(1+\log(1+t)\bigr) \\
&\hphantom{+\log^2(1+t)\Bigl(} +r(2qt+2)\log\bigl(1+\log(1+t)\bigr)+r(r-1)t\Bigr) \\
&+q\log(1+t)\log\bigl(1+\log(1+t)\bigr)\Bigl(\bigl((2q-1)t+2\bigr)\log\bigl(1+\log(1+t)\bigr)+2rt\Bigr) \\
&+q(q-1)t\log^2\bigl(1+\log(1+t)\bigr).
\end{aligned} \end{eq}
Removing the two non-negative items containing ``$qr$'' on the right-hand side of \eqref{eq: f(t), estimate}, we obtain
\begin{eq} \label{eq: f(t), further estimate} \begin{aligned}
f(t)&\geqslant(t+2)\log^3(1+t)\log\bigl(1+\log(1+t)\bigr)\Bigl(q\log\bigl(1+\log(1+t)\bigr)+r\Bigr) \\
&+\log^2(1+t)\Bigl(q\bigl((q+1)t+4\bigr)\log^2\bigl(1+\log(1+t)\bigr) \\
&\hphantom{+\log^2(1+t)\Bigl(} +2r\log\bigl(1+\log(1+t)\bigr)+r(r-1)t\Bigr) \\
&+q\log(1+t)\log^2\bigl(1+\log(1+t)\bigr)\bigl((2q-1)t+2\bigr) \\
&+q(q-1)t\log^2\bigl(1+\log(1+t)\bigr)=g(t)q+h(t)r\comma
\end{aligned} \end{eq}
where
\begin{al}
g(t)&\triangleq \bigl(1+\log(1+t)\bigr)^2\log^2\bigl(1+\log(1+t)\bigr)\bigl((q-1)t+(2+t)\log(1+t)\bigr)\comma \\
h(t)&\triangleq \log^2(1+t)\Bigl((r-1)t+\bigl(2+(2+t)\log(1+t)\bigr)\log\bigl(1+\log(1+t)\bigr)\Bigr).
\end{al}
Recall two elementary inequalities:
\begin{ga}
(1+t)\log(1+t)\geqslant t (t\geqslant0)\comma \\
\bigl(1+(1+t)\log(1+t)\bigr)\log\bigr(1+\log(1+t)\bigr)\geqslant t (t\geqslant0)\comma
\end{ga}
where the second inequality can be proved by differentiation and using the first inequality. Thus, for any $t>0$ we have
\begin{al}
g(t)&\geqslant \bigl(1+\log(1+t)\bigr)^2\log^2\bigl(1+\log(1+t)\bigr)\log(1+t)>0\comma \\
h(t)&\geqslant \log^2(1+t)\bigl(1+\log(1+t)\bigr)\log\bigl(1+\log(1+t)\bigr)>0.
\end{al} 
Since $q^2+r^2>0$, by \eqref{eq: f(t), further estimate} we see that $f(t)>0$ for any $t>0$.
\end{proof}

\begin{cor} \label{cor: convexity}
Fix $p\geqslant 1$ and $q\comma r\geqslant 0$ such that $(p-1)^2+q^2+r^2>0$. Then the function $\Phi_{p\comma q\comma r}$ is strictly convex on $(0\comma{+}\infty)$.
\end{cor}

\begin{proof}
We may as well assume that $q^2+r^2>0$. By \propref{prop: strong convexity},
\begin{eq}
\Psi(t)\triangleq\Phi_{p\comma q\comma r}\left(t^{\frac 1p}\right)
\end{eq}
is strictly convex on $(0\comma{+}\infty)$. Then we have
\begin{eq}
\Phi_{p\comma q\comma r}''(t)=p^2t^{2p-2}\Psi''(t^p)+p(p-1)t^{p-2}\Psi'(t^p)>0
\end{eq}
for any $t>0$ since $p\geqslant 1$ and $\Psi$ is increasing on $[0\comma{+}\infty)$.
\end{proof}


\section{Proof of the gluing lemma of smooth convex functions defined on disjointed intervals} \label{app: Proof of the gluing lemma of smooth convex functions defined on disjointed intervals}

In this appendix, we prove the gluing lemma of smooth convex functions defined on disjointed intervals, i.e. \propref{prop: glue two smooth convex functions defined on disjointed intervals}.

\begin{proof}
Without loss of generality, we assume that $k=\infty$ and $f\comma g$ are both smooth convex functions defined on $\mathbb{R}$. It's easy to see that there exists a positive constant $\delta\in\left(0\comma\frac{a_2-b_1}{2}\right)$ satisfying:
\begin{eq} \label{eq: delta, convex case}
\left\{ \begin{lgathered}
\sup_{[a_1-\delta\comma b_1+\delta]} f''\leqslant \sup_{[a_1\comma b_1]} f''+1\comma\quad\sup_{[a_2-\delta\comma b_2+\delta]} g''\leqslant \sup_{[a_2\comma b_2]} g''+1\comma \\
f'(b_1+\delta)<g'(a_2-\delta)\comma \\
2\left(\frac{g(a_2)-f(b_1+\delta)}{a_2-(b_1+\delta)}-f'(b_1+\delta)\right)\geqslant\frac{g(a_2)-f(b_1)}{a_2-b_1}-f'(b_1)>0\comma \\
2\left(g'(a_2-\delta)-\frac{g(a_2-\delta)-f(b_1)}{(a_2-\delta)-b_1}\right)\geqslant g'(a_2)-\frac{g(a_2)-f(b_1)}{a_2-b_1}>0.
\end{lgathered} \right. \end{eq}
Choose cut-off functions $\xi\comma\eta\in\mathrm{C}_{\mathrm{c}}^\infty(\mathbb{R})$ such that
\begin{eq} \label{eq: xi, eta, convex case}
\left\{ \begin{lgathered}
\xi\bigr|_{[a_1\comma b_1]}\equiv1\comma0\leqslant\xi\leqslant1\comma\supp\xi\subset(a_1-\delta\comma b_1+\delta)\text{;} \\
\eta\bigr|_{[a_2\comma b_2]}\equiv1\comma0\leqslant\eta\leqslant1\comma\supp\eta\subset(a_2-\delta\comma b_2+\delta).
\end{lgathered} \right.
\end{eq}
Define
\begin{al}
\tilde{f}(t)&\triangleq\int_{b_1}^t \int_{b_1}^y \xi(x)f''(x)\dif x\dif y+f(b_1)+f'(b_1)(t-b_1)\comma \\
\tilde{g}(t)&\triangleq\int_{a_2}^t \int_{a_2}^y \eta(x)g''(x)\dif x\dif y+g(a_2)+g'(a_2)(t-a_2).
\end{al}
By \eqref{eq: xi, eta, convex case} and the last two lines of \eqref{eq: delta, convex case}, we have
\begin{eq} \label{eq: tilde f(a_2)<tilde g(a_2), convex case} \begin{aligned}
\tilde{f}(a_2)&\leqslant\int_{b_1}^{b_1+\delta} \int_{b_1}^y f''(x)\dif x\dif y+\int_{b_1+\delta}^{a_2} \int_{b_1}^{b_1+\delta} f''(x)\dif x\dif y+f(b_1)+f'(b_1)(a_2-b_1) \\
&=f(b_1+\delta)+f'(b_1+\delta)\bigl(a_2-(b_1+\delta)\bigr)<g(a_2)=\tilde{g}(a_2)\comma
\end{aligned} \end{eq}
and
\begin{eq} \label{eq: tilde g(b_1)<tilde f(b_1), convex case} \begin{aligned}
\tilde{g}(b_1)&\leqslant\int_{a_2-\delta}^{a_2} \int_y^{a_2} g''(x)\dif x\dif y+\int_{b_1}^{a_2-\delta} \int_{a_2-\delta}^{a_2} g''(x)\dif x\dif y+g(a_2)+g'(a_2)(b_1-a_2) \\
&=g(a_2-\delta)-g'(a_2-\delta)(a_2-\delta-b_1)<f(b_1)=\tilde{f}(b_1).
\end{aligned} \end{eq}
Moreover, it's easy to see that $\tilde{f}\comma\tilde{g}\in\mathrm{C}^\infty(\mathbb{R})$, 
\begin{eq} \label{eq: tilde f=f, tilde g=g, convex case}
\tilde{f}\Bigr|_{[a_1\comma b_1]}=f\bigr|_{[a_1\comma b_1]}\comma\quad\tilde{g}\Bigr|_{[a_2\comma b_2]}=g\bigr|_{[a_2\comma b_2]}\comma
\end{eq}
and
\begin{eq} \label{eq: tilde f'', tilde g'', convex case}
\tilde{f}''(t)=\xi(t)f''(t)\geqslant 0\comma\quad\tilde{g}''(t)=\eta(t)g''(t)\geqslant 0
\end{eq}
for any $t\in\mathbb{R}$. As a result,
\begin{eq} \label{eq: varepsilon, convex case}
\varepsilon\triangleq\min\left\{\frac{\tilde{f}(b_1)-\tilde{g}(b_1)}{2}\comma\frac{\tilde{g}(a_2)-\tilde{f}(a_2)}{2}\right\}
\end{eq}
is a positive constant and we can let $\rho_{\varepsilon}$ be the smooth function in \lemref{lem: regularization of |t|}. We claim that the function
\begin{eq}
h(t)\triangleq\begin{dcases}
\tilde{f}(t)\comma & t\leqslant b_1 \\
\frac{\tilde{f}(t)+\tilde{g}(t)+\rho_{\varepsilon}\bigl(\tilde{f}(t)-\tilde{g}(t)\bigr)}{2}\comma & b_1<t<a_2 \\
\tilde{g}(t)\comma & t\geqslant a_2
\end{dcases} \end{eq}
satisfies all our requirements. In fact, when $t$ is sufficiently close to $b_1+0$, we have $\tilde{f}(t)-\tilde{g}(t)>\varepsilon$ and therefore $h(t)=\tilde{f}(t)$; when $t$ is sufficiently close to $a_2-0$, we have $\tilde{f}(t)-\tilde{g}(t)<-\varepsilon$ and therefore $h(t)=\tilde{g}(t)$. Thus, we have $h\in\mathrm{C}^\infty(\mathbb{R})$. By \eqref{eq: tilde f=f, tilde g=g, convex case} and \eqref{eq: tilde f'', tilde g'', convex case} we see that
\begin{eq}
h\bigr|_{[a_1\comma b_1]}=\tilde{f}\Bigr|_{[a_1\comma b_1]}=f\bigr|_{[a_1\comma b_1]}\comma\quad h\bigr|_{[a_2\comma b_2]}=\tilde{g}\Bigr|_{[a_2\comma b_2]}=g\bigr|_{[a_2\comma b_2]}\comma
\end{eq}
and
\begin{eq}
h''(t)=\tilde{f}''(t)\geqslant 0\comma\forall t\leqslant b_1\text{;} \quad h''(t)=\tilde{g}''(t)\geqslant 0\comma\forall t\geqslant a_2.
\end{eq}
Moreover, for any $t\in(b_1\comma a_2)$ there holds
\begin{eq} \begin{aligned}
h''(t)&=\frac12 \Bigl(\tilde{f}''(t)+\tilde{g}''(t)+\rho_{\varepsilon}''\bigl(\tilde{f}(t)-\tilde{g}(t)\bigr)\bigl(\tilde{f}'(t)-\tilde{g}'(t)\bigr)^2 \\
&\hphantom{=\frac12 \Bigl(}+\rho_{\varepsilon}'\bigl(\tilde{f}(t)-\tilde{g}(t)\bigr)\bigl(\tilde{f}''(t)-\tilde{g}''(t)\bigr)\Bigr) \\
&\geqslant \frac{1+\rho_{\varepsilon}'\bigl(\tilde{f}(t)-\tilde{g}(t)\bigr)}{2}\tilde{f}''(t)+\frac{1-\rho_{\varepsilon}'\bigl(\tilde{f}(t)-\tilde{g}(t)\bigr)}{2}\tilde{g}''(t)\geqslant 0.
\end{aligned} \end{eq}
As a result, we have $\inf\limits_{\mathbb{R}} h''\geqslant0$, i.e. $h$ is convex on $\mathbb{R}$. Next we prove \eqref{eq: supremum of h'', convex case}. For any $t\in(b_1\comma a_2)$, by \eqref{eq: xi, eta, convex case} and the second line of \eqref{eq: delta, convex case} there hold
\begin{eq} \begin{aligned}
\tilde{g}'(t)-\tilde{f}'(t)&=\int_{a_2}^t \eta(x)g''(x)\dif x+g'(a_2)-\int_{b_1}^t \xi(x)f''(x)\dif x-f'(b_1) \\
&\geqslant -\int_{a_2-\delta}^{a_2} g''(x)\dif x+g'(a_2)-\int_{b_1}^{b_1+\delta} f''(x)\dif x-f'(b_1) \\
&=g'(a_2-\delta)-f'(b_1+\delta)>0 \\
\end{aligned} \end{eq}
and
\begin{eq}
\tilde{g}'(t)-\tilde{f}'(t)\leqslant g'(a_2)-f'(b_1).
\end{eq}
The two inequalities above imply
\begin{eq} \label{eq: (tilde g'-tilde f')^2, convex case}
\left(\tilde{g}'(t)-\tilde{f}'(t)\right)^2\leqslant\bigl(g'(a_2)-f'(b_1)\bigr)^2
\end{eq}
for any $t\in(b_1\comma a_2)$. On the other hand, by \eqref{eq: varepsilon, convex case}, \eqref{eq: tilde f(a_2)<tilde g(a_2), convex case}, \eqref{eq: tilde g(b_1)<tilde f(b_1), convex case} and \eqref{eq: delta, convex case} we have
\begin{eq} \label{eq: varepsilon, estimate, convex case} \begin{aligned}
\varepsilon&\geqslant\frac12\min\Bigl\{f(b_1)-g(a_2-\delta)+g'(a_2-\delta)(a_2-\delta-b_1) \comma \\
&\mathrel{\hphantom{\geqslant}}\hphantom{\frac12\min\Bigl\{}g(a_2)-f(b_1+\delta)-f'(b_1+\delta)(a_2-b_1-\delta)\Bigr\} \\
&\geqslant\frac{a_2-b_1-\delta}{4}\min\left\{g'(a_2)-\frac{g(a_2)-f(b_1)}{a_2-b_1}\comma\frac{g(a_2)-f(b_1)}{a_2-b_1}-f'(b_1)\right\} \\
&\geqslant\frac{a_2-b_1}{8}\min\left\{g'(a_2)-\frac{g(a_2)-f(b_1)}{a_2-b_1}\comma\frac{g(a_2)-f(b_1)}{a_2-b_1}-f'(b_1)\right\}.
\end{aligned} \end{eq}
Combining \eqref{eq: (tilde g'-tilde f')^2, convex case}, \eqref{eq: varepsilon, estimate, convex case} and \lemref{lem: regularization of |t|}, we obtain
\begin{eq} \begin{aligned}
&\mathrel{\phantom{=}}\rho_{\varepsilon}''\bigl(\tilde{f}(t)-\tilde{g}(t)\bigr)\bigl(\tilde{f}'(t)-\tilde{g}'(t)\bigr)^2 \\
&\leqslant\frac{8M\bigl(g'(a_2)-f'(b_1)\bigr)^2}{(a_2-b_1)\min\left\{\frac{g(a_2)-f(b_1)}{a_2-b_1}-f'(b_1)\comma g'(a_2)-\frac{g(a_2)-f(b_1)}{a_2-b_1}\right\}}
\end{aligned} \end{eq}
for any $t\in(b_1\comma a_2)$. Furthermore, by the first line of \eqref{eq: delta, convex case} we have
\begin{al}
\tilde{f}''(t)&=\xi(t)f''(t)\leqslant\sup_{[a_1-\delta\comma b_1+\delta]} f''\leqslant\sup_{[a_1\comma b_1]} f''+1\comma \\
\tilde{g}''(t)&=\eta(t)g''(t)\leqslant\sup_{[a_2-\delta\comma b_2+\delta]} g''\leqslant\sup_{[a_2\comma b_2]} g''+1
\end{al}
for any $t\in\mathbb{R}$. As a result, we have
\begin{eq}
h''(t)=\tilde{f}''(t)\leqslant\sup_{[a_1\comma b_1]} f''+1\comma\forall t\leqslant b_1\text{;}\quad h''(t)=\tilde{g}''(t)\leqslant\sup_{[a_2\comma b_2]} g''+1\comma\forall t\geqslant a_2\text{;}
\end{eq}
and
\begin{eq} \begin{aligned}
h''(t)&=\frac{1+\rho_{\varepsilon}'\bigl(\tilde{f}(t)-\tilde{g}(t)\bigr)}{2}\tilde{f}''(t)+\frac{1-\rho_{\varepsilon}'\bigl(\tilde{f}(t)-\tilde{g}(t)\bigr)}{2}\tilde{g}''(t) \\
&\hphantom{=}+\frac12\rho_{\varepsilon}''\bigl(\tilde{f}(t)-\tilde{g}(t)\bigr)\bigl(\tilde{f}'(t)-\tilde{g}'(t)\bigr)^2 \\
&\leqslant\frac{4M\bigl(g'(a_2)-f'(b_1)\bigr)^2}{(a_2-b_1)\min\left\{\frac{g(a_2)-f(b_1)}{a_2-b_1}-f'(b_1)\comma g'(a_2)-\frac{g(a_2)-f(b_1)}{a_2-b_1}\right\}} \\
&\phantom{=}+1+\max\left\{\sup_{[a_1\comma b_1]}f''\comma\sup_{[a_2\comma b_2]}g''\right\}\comma\forall t\in(b_1\comma a_2).
\end{aligned} \end{eq}
We have justified \eqref{eq: supremum of h'', convex case} and therefore our claim that the function $h$ satisfies all our requirements.
\end{proof}


\section[Properties of the function f_ε in Example 7.7]{Properties of the function $f_{\varepsilon}$ in Example~7.7} \label{app: Properties of the function f_varepsilon in Example 7.7}

Recall that we have defined the function $f_\varepsilon$ in \eqref{eq: f_varepsilon} as
\begin{eq}
f_{\varepsilon}(t)\triangleq-\frac{\log2}{2}\log\Biggl(1+\log\biggl(1+\log\biggl(1+\log\left(1+\frac{1}{t+\varepsilon}\right)\biggr)\biggr)\Biggr)\comma t\in\left[0\comma\frac{1}{16}\right]\comma
\end{eq}
where $a$ is a positive constant independent of $\varepsilon$. In this appendix, we prove some properties of this function which have been used in \exref{ex: L^infty estimates fail when r leqslant n-1}. For simplicity, we will omit the coefficient $\frac{\log2}{2}$ of $f_{\varepsilon}$ and write $f_{\varepsilon}(t)$ as
\begin{eq}
-\log\Biggl(1+\log\biggl(1+\log\biggl(1+\log\left(1+\frac{1}{t+\varepsilon}\right)\biggr)\biggr)\Biggr).
\end{eq}

\begin{prop} \label{prop: properties of f_varepsilon}
Fix $\varepsilon\in\left(0\comma\frac{1}{4}\right]$. Let $f$ denote $f_\varepsilon$.
\begin{enumerate}[(1)]
\item For any $t\in\left[0\comma\frac{1}{4}\right]$, we have $f'(t)>0$ and $f'(t)+tf''(t)>0$.
\item Fix $t_0\in\left(0\comma\frac{1}{4}\right)$. For any $t\in\left[t_0\comma\frac{1}{4}\right]$, we have $f'(t)+tf''(t)\leqslant C$, where $C$ is a positive constant depending only on $\frac{1}{t_0}$.
\item Fix $n\in\mathbb{N}\colon n\geqslant 2$ and let $F(t)$ denote
\begin{eq}
\bigl(f'(t)\bigr)^{n-1}\bigl(f'(t)+tf''(t)\bigr).
\end{eq}
Then we have
\begin{eq} \label{eq: F_L^1 log^nL log^(n-1)logL}
\int_0^{\frac14} t^{n-1}F(t)\log^n\bigl(1+F(t)\bigr)\log^{n-1}\Bigl(1+\log\bigl(1+F(t)\bigr)\Bigr)\dif t\leqslant C\comma
\end{eq}
where $C$ is a positive constant depending only on $n$.
\end{enumerate}
\end{prop}

\begin{proof}
(1) For any $t\in\left[0\comma\frac{1}{4}\right]$, straightforward calculations show that
\begin{eq} \label{eq: f'(t)}
f'(t)=\frac{\biggl(1+\log\Bigl(1+\log\Bigl(1+\log\left(1+\frac{1}{t+\varepsilon}\right)\Bigr)\Bigr)\biggr)^{-1}\left(\frac{1}{t+\varepsilon}\right)^2}{\biggl(1+\log\Bigl(1+\log\left(1+\frac{1}{t+\varepsilon}\right)\Bigr)\biggr)\Bigl(1+\log\left(1+\frac{1}{t+\varepsilon}\right)\Bigr)\left(1+\frac{1}{t+\varepsilon}\right)}>0
\end{eq}
and
\begin{eq} \label{eq: f'(t)+tf''(t)} \begin{aligned}
&\mathrel{\phantom{=}}f'(t)+tf''(t) \\
&=\frac{\biggl(1+\log\Bigl(1+\log\Bigl(1+\log\left(1+\frac{1}{t+\varepsilon}\right)\Bigr)\Bigr)\biggr)^{-2}}{\biggl(1+\log\Bigl(1+\log\left(1+\frac{1}{t+\varepsilon}\right)\Bigr)\biggr)^2\Bigl(1+\log\left(1+\frac{1}{t+\varepsilon}\right)\Bigr)^2(1+t+\varepsilon)^2(t+\varepsilon)^2}g(t)\comma
\end{aligned} \end{eq}
where
\begin{eq} \begin{aligned}
g(t)&\triangleq t\Bigl(2+\log\bigl(1+\log(1+s)\bigr)\Bigr) \\
&\phantom{=}+t(1-t-ts)\bigl(1+\log(1+s)\bigr)\Bigl(1+\log\bigl(1+\log(1+s)\bigr)\Bigr)  \\
&\phantom{=}+\varepsilon(1+\varepsilon)(1+s)\bigl(1+\log(1+s)\bigr)\Bigl(1+\log\bigl(1+\log(1+s)\bigr)\Bigr) \\
\end{aligned} \end{eq} 
and
\begin{eq}
s\triangleq\log\left(1+\frac{1}{t+\varepsilon}\right).
\end{eq}
Note that
\begin{eq}
0<s\leqslant(t+\varepsilon)^{-\frac12}
\end{eq}
and therefore
\begin{eq}
g(t)\geqslant t\left(1-t-t^{\frac12}\right)\bigl(1+\log(1+s)\bigr)\Bigl(1+\log\bigl(1+\log(1+s)\bigr)\Bigr)+\varepsilon>0\comma
\end{eq}
where the last ``$>$'' is due to $t\leqslant\frac14$ and $\varepsilon>0$. Thus, we have $f'(t)+tf''(t)>0$.

(2) For any $t\in\left[t_0\comma\frac{1}{4}\right]$, by \eqref{eq: f'(t)+tf''(t)} and $\varepsilon\in\left(0\comma\frac14\right]$ we have
\begin{eq}
f'(t)+tf''(t)\leqslant\frac{\Bigl(1+\log\bigl(1+\log(1+\log3)\bigr)\Bigr)^{-2}}{\bigl(1+\log(1+\log3)\bigr)^2(1+\log3)^2t_0^2}g(t)\comma
\end{eq}
where
\begin{eq} \begin{aligned}
g(t)&\leqslant \frac14 \Bigl(2+\log\bigl(1+\log(1+s)\bigr)\Bigr)+\frac14 \bigl(1+\log(1+s)\bigr)\Bigl(1+\log\bigl(1+\log(1+s)\bigr)\Bigr)  \\
&\phantom{=}+\frac{5}{16}(1+s)\bigl(1+\log(1+s)\bigr)\Bigl(1+\log\bigl(1+\log(1+s)\bigr)\Bigr) \\
\end{aligned} \end{eq} 
and
\begin{eq}
s\triangleq\log\left(1+\frac{1}{t+\varepsilon}\right)\leqslant\frac{1}{t+\varepsilon}\leqslant\frac{1}{t_0}.
\end{eq}
Thus, there holds $f'(t)+tf''(t)\leqslant C$, where $C$ depends only on $\frac{1}{t_0}$.

(3) For any $t\in\left[0\comma\frac{1}{4}\right]$, by \eqref{eq: f'(t)} and \eqref{eq: f'(t)+tf''(t)} there holds
\begin{eq}
F(t)\leqslant\frac{\biggl(1+\log\Bigl(1+\log\Bigl(1+\log\left(1+\frac{1}{t+\varepsilon}\right)\Bigr)\Bigr)\biggr)^{-n-1}}{\biggl(1+\log\Bigl(1+\log\left(1+\frac{1}{t+\varepsilon}\right)\Bigr)\biggr)^{n+1}\Bigl(1+\log\left(1+\frac{1}{t+\varepsilon}\right)\Bigr)^{n+1}(t+\varepsilon)^{n+1}}g(t)\comma
\end{eq}
where
\begin{eq} \begin{aligned}
g(t)&\leqslant 3t\bigl(1+\log(1+s)\bigr)\Bigl(1+\log\bigl(1+\log(1+s)\bigr)\Bigr)  \\
&\phantom{=}+\frac54\varepsilon(1+s)\bigl(1+\log(1+s)\bigr)\Bigl(1+\log\bigl(1+\log(1+s)\bigr)\Bigr)
\end{aligned} \end{eq}
and
\begin{eq}
s\triangleq\log\left(1+\frac{1}{t+\varepsilon}\right).
\end{eq}
Thus, we have
\begin{eq}
F(t)\leqslant h_1(t)+h_2(t)\comma
\end{eq}
where
\begin{al}
h_1(t)&\triangleq\frac{3\biggl(1+\log\Bigl(1+\log\Bigl(1+\log\left(1+\frac{1}{t+\varepsilon}\right)\Bigr)\Bigr)\biggr)^{-n}\left(\frac{1}{t+\varepsilon}\right)^{n}}{\biggl(1+\log\Bigl(1+\log\left(1+\frac{1}{t+\varepsilon}\right)\Bigr)\biggr)^{n}\Bigl(1+\log\left(1+\frac{1}{t+\varepsilon}\right)\Bigr)^{n+1}}\comma \\
h_2(t)&\triangleq\frac{\frac54\varepsilon\biggl(1+\log\Bigl(1+\log\Bigl(1+\log\left(1+\frac{1}{t+\varepsilon}\right)\Bigr)\Bigr)\biggr)^{-n}\left(\frac{1}{t+\varepsilon}\right)^{n+1}}{\biggl(1+\log\Bigl(1+\log\left(1+\frac{1}{t+\varepsilon}\right)\Bigr)\biggr)^{n}\Bigl(1+\log\left(1+\frac{1}{t+\varepsilon}\right)\Bigr)^{n}}.
\end{al}
Since $\frac{1}{t+\varepsilon}\geqslant2$, we see that
\begin{eq}
F(t)\leqslant C_1\left(\frac{1}{t+\varepsilon}\right)^n\comma
\end{eq}
where $C_1>1$ is a constant depending only on $n$. As a result, we obtain
\begin{eq} \label{eq: F_L^1 log^nL log^(n-1)logL, estimate} \begin{aligned}
&\mathrel{\phantom{=}}\int_0^{\frac14} t^{n-1}F(t)\log^n\bigl(1+F(t)\bigr)\log^{n-1}\Bigl(1+\log\bigl(1+F(t)\bigr)\Bigr)\dif t \\
&\leqslant C_2\int_0^{\frac14} t^{n-1}\bigl(h_1(t)+h_2(t)\bigr)\log^n\left(1+\frac{1}{t+\varepsilon}\right)\log^{n-1}\Biggl(1+\log\left(1+\frac{1}{t+\varepsilon}\right)\Biggr)\dif t.
\end{aligned} \end{eq}
Note that
\begin{eq} \label{eq: h_1(t), estimate} \begin{aligned}
&\mathrel{\phantom{=}}\int_0^{\frac14} t^{n-1}h_1(t)\log^n\left(1+\frac{1}{t+\varepsilon}\right)\log^{n-1}\Biggl(1+\log\left(1+\frac{1}{t+\varepsilon}\right)\Biggr)\dif t  \\
&\leqslant 3\int_0^{\frac14} \frac{\biggl(1+\log\Bigl(1+\log\Bigl(1+\log\left(1+\frac{1}{t+\varepsilon}\right)\Bigr)\Bigr)\biggr)^{-n}\frac{1}{t+\varepsilon}}{\biggl(1+\log\Bigl(1+\log\left(1+\frac{1}{t+\varepsilon}\right)\Bigr)\biggr)\Bigl(1+\log\left(1+\frac{1}{t+\varepsilon}\right)\Bigr)}\dif t  \\
&\leqslant 3\int_0^{\frac12} \frac{\biggl(1+\log\Bigl(1+\log\Bigl(1+\log\left(1+\frac{1}{t}\right)\Bigr)\Bigr)\biggr)^{-n}}{t\biggl(1+\log\Bigl(1+\log\left(1+\frac{1}{t}\right)\Bigr)\biggr)\Bigl(1+\log\left(1+\frac{1}{t}\right)\Bigr)}\dif t<{+}\infty
\end{aligned} \end{eq}
and
\begin{eq} \label{eq: h_2(t), estimate} \begin{aligned}
&\mathrel{\phantom{=}}\int_0^{\frac14} t^{n-1}h_2(t)\log^n\left(1+\frac{1}{t+\varepsilon}\right)\log^{n-1}\Biggl(1+\log\left(1+\frac{1}{t+\varepsilon}\right)\Biggr)\dif t  \\
&=\varepsilon^{n}\int_0^{\frac{1}{4\varepsilon}} t^{n-1}h_2(\varepsilon t)\log^n\left(1+\frac{1}{\varepsilon(t+1)}\right)\log^{n-1}\Biggl(1+\log\left(1+\frac{1}{\varepsilon(t+1)}\right)\Biggr)\dif t \\
&\leqslant\frac54 \int_0^{{+}\infty} \frac{\biggl(1+\log\Bigl(1+\log\Bigl(1+\log\left(1+\frac{1}{t+1}\right)\Bigr)\Bigr)\biggr)^{-n}\left(\frac{1}{t+1}\right)^{2}}{1+\log\Bigl(1+\log\left(1+\frac{1}{t+1}\right)\Bigr)}\dif t<{+}\infty.
\end{aligned} \end{eq}
Combining \eqref{eq: F_L^1 log^nL log^(n-1)logL, estimate}, \eqref{eq: h_1(t), estimate} and \eqref{eq: h_2(t), estimate}, we have proved \eqref{eq: F_L^1 log^nL log^(n-1)logL}.
\end{proof}

\begin{cor} \label{cor: properties of f_varepsilon}
Fix $\varepsilon\in\left(0\comma\frac{1}{4}\right]$. Let $f$ denote $f_\varepsilon$.
\begin{enumerate}[(1)]
\item $f\bigl(|\bm z|^2\bigr)$ is a strictly plurisubharmonic function on $\overline{\mathrm{B}_{\frac 12}(\bm 0)}\subset\mathbb{C}^n$.
\item Fix $t_0\in\left(0\comma\frac{1}{4}\right)$. For any $t\in\left[\log t_0\comma\log\frac{1}{4}\right]$, we have $\bigl(f(\mathrm{e}^t)\bigr)''\leqslant C$, where $C$ is a positive constant depending only on $\frac{1}{t_0}$.
\item Fix $n\in\mathbb{N}\colon n\geqslant 2$. Then we have
\begin{eq}
\left\|\det\biggl(\mathrm{D}_{\mathbb{C}}^2\Bigl(f\left(|\bm z|^2\right)\Bigr)\biggr)\right\|_{\mathrm{L}^1(\log\mathrm{L})^n(\log\log\mathrm{L})^{n-1}\left(\mathrm{B}_{\frac12}(\bm 0)\comma\dif V_{\mathbb{C}^n}\right)}\leqslant C\comma
\end{eq}
where $C$ is a positive constant depending only on $n$.
\end{enumerate}
\end{cor}

\begin{proof}
(1) We have seen in \eqref{eq: eigenvalues of D_C^2(f(|z|^2))} that all eigenvalues of $\mathrm{D}_{\mathbb{C}}^2\Bigl(f\bigl(|\bm z|^2\bigr)\Bigr)$ are
\begin{eq}
f'\bigl(|\bm z|^2\bigr)\mbox{ ($n-1$ times) and }f'\bigl(|\bm z|^2\bigr)+f''\bigl(|\bm z|^2\bigr)|\bm z|^2.
\end{eq}
Then by the property (1) in \propref{prop: properties of f_varepsilon} it's obvious that $f\bigl(|\bm z|^2\bigr)$ is strictly plurisubharmonic on $\overline{\mathrm{B}_{\frac 12}(\bm 0)}$.

(2) Since
\begin{eq}
\bigl(f(\mathrm{e}^t)\bigr)''=\mathrm{e}^t\bigl(f'(\mathrm{e}^t)+\mathrm{e}^tf''(\mathrm{e}^t)\bigr)\comma
\end{eq}
by the property (2) in \propref{prop: properties of f_varepsilon} it's obvious that $\bigl(f(\mathrm{e}^t)\bigr)''\leqslant C$ for any $t\in\left[\log t_0\comma\log\frac{1}{4}\right]$ and a positive constant $C$ depending only on $\frac{1}{t_0}$.

(3) Let $F(t)$ denote
\begin{eq}
\bigl(f'(t)\bigr)^{n-1}\bigl(f'(t)+tf''(t)\bigr).
\end{eq}
Then we have
\begin{eq}
\det\biggl(\mathrm{D}_{\mathbb{C}}^2\Bigl(f\left(|\bm z|^2\right)\Bigr)\biggr)=\Bigl(f'\bigl(|\bm z|^2\bigr)\Bigr)^{n-1}\Bigl(f'\bigl(|\bm z|^2\bigr)+f''\bigl(|\bm z|^2\bigr)|\bm z|^2\Bigr)=F\bigl(|\bm z|^2\bigr).
\end{eq}
On the other hand, we note that
\begin{eq} \begin{aligned}
&\mathrel{\phantom{=}}\int_{\mathrm{B}_{\frac 12}(\bm 0)} F\bigl(|\bm z|^2\bigr)\log^n\Bigl(1+F\bigl(|\bm z|^2\bigr)\Bigr)\log^{n-1}\biggl(1+\log\Bigl(1+F\bigl(|\bm z|^2\bigr)\Bigr)\biggr)\dif V_{\mathbb{C}^n} \\
&=C_1\int_0^{\frac12} t^{2n-1}F(t^2)\log^n\bigl(1+F(t^2)\bigr)\log^{n-1}\Bigl(1+\log\bigl(1+F(t^2)\bigr)\Bigr)\dif t \\
&=\frac12 C_1\int_0^{\frac14} t^{n-1}F(t)\log^n\bigl(1+F(t)\bigr)\log^{n-1}\Bigl(1+\log\bigl(1+F(t)\bigr)\Bigr)\dif t.
\end{aligned} \end{eq}
This implies
\begin{eq}
\left\|F\bigl(|\bm z|^2\bigr)\right\|_{\mathrm{L}^1(\log\mathrm{L})^n(\log\log\mathrm{L})^{n-1}\left(\mathrm{B}_{\frac12}(\bm 0)\comma\dif V_{\mathbb{C}^n}\right)}\leqslant C
\end{eq}
by \propref{prop: estimates of norms} and the property (3) in \propref{prop: properties of f_varepsilon}.
\end{proof}
 

\section*{Acknowledgements}
\addcontentsline{toc}{section}{Acknowledgements}

The author would like to thank his advisor Professor Gang Tian for his helpful guidance.

\section*{Declarations}
\addcontentsline{toc}{section}{Declarations}

\noindent{\bfseries Conflict of interest} The author has no relevant financial or non-financial interests to disclose.

\noindent{\bfseries Data availability} No datasets were generated or analysed during the current study.

\phantomsection
\addcontentsline{toc}{section}{References}
\setstretch{1}

\end{document}